\documentclass[reqno]{amsart}

\usepackage{amsmath,amssymb,amsthm,comment,enumerate,mathrsfs}

\usepackage{hyperref}

\usepackage{graphicx}

\usepackage{color}

\swapnumbers

\numberwithin{equation}{section}

\newtheorem{Thm}{Theorem}[section]

\newtheorem{Lem}[Thm]{Lemma}
\newtheorem{Cor}[Thm]{Corollary}
\theoremstyle{definition}
\newtheorem{Rem}[Thm]{Remark}
\newtheorem{Def}[Thm]{Definition}
\newtheorem{Expl}[Thm]{Example}

\newcommand{\R}{\mathbb{R}}
\newcommand{\Z}{\mathbb{Z}}

\newcommand{\E}[1]{{\rm E}(#1)}

\newcommand{\e}{{\rm e}}

\newcommand{\cA}{\mathscr{A}}

\newcommand{\eps}{\varepsilon}

\renewcommand{\rho}{\varrho}

\newcommand{\sub}{\subset}

\title{
Injective Hulls of Infinite Totally Split-Decomposable Metric Spaces 
}

\author{Ma\"el Pav\' on}
\address{Department of Mathematics, ETH Z\"urich, 8092 Z\"urich, Switzerland}
\email{mael.pavon@math.ethz.ch}

\date{\today}

\begin{document}

\maketitle
\begin{abstract}
We consider the class of (possibly) infinite metric spaces with integer-valued totally split-decomposable metric and possessing an injective hull which has the structure of a  polyhedral complex. For this class, we give a characterization for the injective hull to be combinatorially equivalent to a CAT(0) cube complex. In order to obtain these results, we extend the decomposition theory introduced by Bandelt and Dress in 1992 as well as results on the tight span of totally split-decomposable metric spaces proved by Huber, Koolen and Moulton in 2006.

As an application, and using results of Lang of 2013, we obtain proper actions on CAT(0) cube complexes for finitely generated groups endowed with a totally split-decomposable word metric whose associated splits satisfy an easy combinatorial property. In the case of Gromov hyperbolic groups, the action is proper as well as cocompact. 
\end{abstract}

\section{Introduction}\label{sec:Introduction}

In 1992, Bandelt and Dress (cf. \cite{BanD}) introduced a decomposition theory for finite metric spaces which is canonical, namely it is the only one which is, in a sense, compatible with Isbell's injective hull.

Our first goal is to extend the canonical decomposition theory to the class of (possibly) infinite metric spaces with integer-valued totally split-decomposable metric and possessing an injective hull which has the structure of a  polyhedral complex. For this class, we then provide necessary and sufficient conditions for the injective hull to be combinatorially equivalent to a CAT(0) cube complex. 

The basic definitions of the canonical decomposition theory of Bandelt and Dress do not need to be modified to suit our more general situation. A \textit{split} (also called \textit{cut}) $S = \{A,B\}$ of a set $X$ is a pair of non-empty subsets of $X$ such that $A \cap B = \emptyset$ and $X = A \cup B$, or in other words $X = A \sqcup B$. For $x \in X$, we denote by $S(x)$ the element of $S$ that contains $x$. The \textit{split metric}  associated to $S$ is then a pseudometric $\delta_S$ on $X$ such that
\[
\delta_S(x,y) \; = \; \left\lbrace{\begin{matrix} 
1 &\text{if }  S(x) \neq S(y) , \\
0 &\text{if }  S(x) = S(y). \end{matrix}}\right.
\]
For a pseudometric $d$ on $X$, we call $S=\{A,B\}$ a \textit{$d$-split} (of $X$) if the \textit{isolation index} given by \eqref{eq:IsoI} satisfies 
\[
\alpha^d_S >0.
\]
The pseudometric $d$ is said to be \textit{totally split-decomposable} if $d = \sum_{S \in \mathcal{S}} \alpha^d_S \delta_S$ where $\mathcal{S}$ is the set of all $d$-splits. 

A split subsystem $\bar{\mathcal{S}} \subset \mathcal{S}$ is called \textit{octahedral} if and only if there is a partition of $X$ into a disjoint union of six non-empty sets $X = Y^1_1 \sqcup Y^{-1}_1 \sqcup Y^1_2 \sqcup Y^{-1}_2 \sqcup Y^1_3 \sqcup Y^{-1}_3$ such that $\bar{\mathcal{S}}$ consists of the four splits 
\begin{align*}
S_1 &:= \{ Y^1_1 \sqcup Y^1_2 \sqcup  Y^1_3 , Y^{-1}_1 \sqcup Y^{-1}_2 \sqcup Y^{-1}_3 \}, \\
S_2 &:= \{ Y^1_1 \sqcup Y^1_2 \sqcup Y^{-1}_3 , Y^{-1}_1 \sqcup Y^{-1}_2 \sqcup Y^1_3 \}, \\
S_3 &:= \{ Y^1_1 \sqcup Y^{-1}_2 \sqcup Y^1_3 , Y^{-1}_1 \sqcup Y^1_2 \sqcup Y^{-1}_3 \}, \\
S_4 &:= \{ Y^1_1 \sqcup Y^{-1}_2 \sqcup Y^{-1}_3 , Y^{-1}_1 \sqcup Y^1_2 \sqcup Y^1_3 \}.
\end{align*}
$\mathcal{S}$ is called \textit{octahedral-free} if it does not contain any octahedral split subsystem. Two splits $S:=\{A,B\}$ and $S':=\{A',B'\}$ are said to be \textit{compatible} if $A' \subset A$ (and $B \subset B'$) or alternatively $A \subset A'$ (and $B' \subset B$).

For general facts regarding injective hulls, we shall refer to \cite{Lan}. Injective hulls can be characterized in several different ways. In the sequel, \textit{the} injective hull refers to Isbell's injective hull construction $(X,d) \mapsto \mathrm{E}(X,d)$. Recall at this point that the injective hull $\mathrm{E}(X,d)$ of a pseudometric space $(X,d)$ is given by
\begin{equation}\label{eq:DefIH}
\mathrm{E}(X,d) = \bigl\{ f \in \R^X : 
\text{$f(x) = \textstyle\sup_{y \in X}(d(x,y) - f(y))$ for all $x \in X$} 
\bigr\}.
\end{equation}

The difference between two elements of $\mathrm{E}(X,d)$ has finite $\left\| \cdot \right\|_{\infty}$-norm and $\mathrm{E}(X,d)$ is endowed with the metric
\[
d_{\infty}(f,g) := \left\| f-g \right\|_{\infty}.
\]
It is easy to see that for $f \in \mathrm{E}(X,d)$, if $d(x,x')=0$ then $f(x) = f(x')$. Hence, if $(X,d)$ is a pseudometric space and $(Y,d')$ is the associated metric space obtained by collapsing every maximal set of diameter zero to a single point, then $\mathrm{E}(X,d)$ and $\mathrm{E}(Y,d')$ are isometric. Accordingly, the statements involving the injective hull will be stated for metric spaces instead of pseudometric spaces. As it is shown in \cite[Theorem~4.5]{Lan} and as we shall recall later in this Introduction, as soon as $(X,d)$ is a metric space with integer-valued metric verifying \textit{the local rank condition (LRC)}, which is discussed below, there is a canonical locally finite dimensional polyhedral structure on $\mathrm{E}(X,d)$. 

In the case where $d$ is totally split-decomposable, our goal is to provide necessary and sufficient conditions ensuring that $\mathrm{E}(X,d)$ is combinatorially equivalent to a $\mathrm{CAT}(0)$ cube complex. Accordingly, we have

\begin{Thm}\label{Thm:t11}
Let $(X,d)$ be a metric space with integer-valued totally split-decomposable metric satisfying the local rank condition. Let $\mathcal{S}$ be the set of all $d$-splits. Then, the following are equivalent:
\begin{enumerate}[(i)]
\item \label{it:it1} $\mathcal{S}$ does not contain any octahedral split subsystem $\bar{\mathcal{S}}$ satisfying that for every $S=\{A,B\} \in \mathcal{S} \setminus \bar{\mathcal{S}}$, there is $S':= \{A',B'\} \in \bar{\mathcal{S}}$ such that $S$ and $S'$ are compatible.
\item \label{it:it2} Each cell of $\mathrm{E}(X,d)$ is a parallelotope.
\end{enumerate}
If (\ref{it:it1}) or (\ref{it:it2}) holds, there is a CAT(0) cube complex $\mathrm{K}(X,d)$ and a canonical bijective cell complex isomorphism $\sigma \colon \mathrm{E}(X,d) \to \mathrm{K}(X,d)$ mapping cells affinely to cells.
\end{Thm}

By a \textit{parallelotope} we mean a Minkowski sum of a finite collection of linearly independent closed segments (see for instance \cite{Gru}).
When condition \eqref{it:it1} in Theorem~\ref{Thm:t11} holds, we say that the family of all $d$-splits of $X$ has \textit{no compatibly octahedral decomposition}. 
If the diameters of the cells of $\mathrm{E}(X,d)$ are uniformly bounded, $\sigma$ in Theorem~\ref{Thm:t11} can be chosen to be bi-Lipschitz.

For a metric space $(X,d)$, let $I(x, y) := \{z \in X : d(x, z) + d(z, y) = d(x, y)\}$. $(X,d)$ is called \textit{discretely geodesic} if the metric is integer-valued and for every pair of points $x,y \in X$ there exists an isometric embedding $\gamma \colon \{0, 1, \dots, d(x, y) \} \to  X$ such that $\gamma(0)=x$ and $\gamma(d(x,y))=y$. Moreover, we say that a discretely geodesic metric space $X$ has \textit{$\beta$-stable intervals}, for some constant $\beta \ge 0$, if for every triple of points $x,y,y' \in X$ such that $d(y,y')=1$ we have
\[
d_H(I(x,y),I(x,y')) \le \beta
\]
where $d_H$ denotes the Hausdorff distance in $X$.

The injective hull has among other features, applications to geometric group theory. Let $\Gamma$ be a finitely generated group, $G$ a finite generating set, and let $\Gamma$ be equipped with the word metric $d_G$ with respect to the alphabet $G \cup G^{-1}$. It is shown in \cite[Theorem~1.4]{Lan} that if $(\Gamma,d_G)$ has $\beta$-stable intervals, then $(\Gamma,d_G)$ satisfies the (LRC) and thus $\mathrm{E}(\Gamma,d_G)$ has a polyhedral complex structure. The isometric action of $\Gamma$ on $(\Gamma,d_G)$ given by $(x,y) \mapsto L_x(y) := xy$ induces consequently a proper action by cell isometries of $\Gamma$ on $\mathrm{E}(\Gamma,d_G)$ given by
\[
(x,f) \mapsto \bar{L}_x (f) = f \circ L_x^{-1}.
\]
Moreover, if $(\Gamma,d_G)$ is $\delta$-hyperbolic (in particular it has $\beta$-stable intervals), then $\mathrm{E}(\Gamma,d_G)$ is a finite dimensional polyhedral complex with a uniform bound on the diameter of its cells and the action is cocompact. As an immediate consequence of these observations, we thus obtain:

\begin{Thm}\label{Thm:ttt2}
Let $\Gamma$ be a finitely generated group and $(\Gamma,d_G)$ the associated metric space with respect to the alphabet $G \cup G^{-1}$. Assume that $d_G$ is totally split-decomposable and (\ref{it:it1}) in Theorem~\ref{Thm:t11} holds. Then, the following hold:
\begin{enumerate}[(i)]
\item if $(\Gamma,d_G)$ has $\beta$-stable intervals, there is a proper action of $\Gamma$ on $\mathrm{K}(\Gamma,d_G)$ given by
\[
(x,y) \mapsto (\sigma \circ \bar{L}_x \circ \sigma^{-1})(y).
\]
\item If $(\Gamma,d_G)$ is $\delta$-hyperbolic, the action of $\Gamma$ on $\mathrm{K}(\Gamma,d_G)$ is proper as well as cocompact.
\end{enumerate}
\end{Thm}

We give an outline of the structure of Isbell's injective hull and describe when it corresponds to that of a polyhedral complex, following \cite{Lan}. Given a pseudometric space $(X,d)$, let us consider the vector space $\R^X$ of real-valued functions on $X$ and
\[
\Delta(X,d) := \{ f \in \R^X : 
\text{$f(x) + f(y) \ge d(x,y)$ for all $x,y \in X$}\}.
\]
We call $f \in \Delta(X,d)$ \textit{extremal} if there is no $g \le f$ in $\Delta(X,d)$ distinct from $f$. The set $\mathrm{E}(X,d)$ of extremal functions is equivalently given by \eqref{eq:DefIH}.
To be able to describe the structure of $\mathrm{E}(X,d)$ further, one can assign to every $f \in \mathrm{E}(X,d)$ the undirected graph with vertex set $X$ and edge set
\begin{equation}\label{eq:e000000}
A(f) := \bigl\{ \{x,y\} : x,y \in X \text{ and } f(x) + f(y) = d(x,y) \bigr\}.
\end{equation}
Furthermore, we let
\[
\mathrm{E}'(X,d) := \bigl\{ f \in \Delta(X,d) : \textstyle\bigcup A(f) = X \bigr\}.
\]
Note that if $f \in \mathrm{E}'(X,d)$, the graph $(X,A(f))$ has no isolated vertices
(although it may be disconnected). A set $A$ of unordered pairs of (possibly equal) points in $X$ is called {\em admissible} if there exists an $f \in \mathrm{E}'(X,d)$ with $A(f) = A$, and we denote by $\cA(X)$ the collection of admissible sets. 

To every $A \in \cA(X)$, we associate the affine subspace $H(A)$ of $\R^X$ given by 
\begin{align*}
H(A) 
&:= \{ g \in \R^X : A \subset A(g) \} \\
&= \{ g \in \R^X : g(x) + g(y) = d(x, y) \text{ for all }  {x, y} \in A \}.
\end{align*}
We define the \textit{rank} of $A$ to be the dimension of $H(A)$,
\[
\mathrm{rank}(A) := \mathrm{dim}(H(A)) \in \mathbb{N} \cup \{0,\infty\}.
\]
We can compute $\mathrm{rank}(A)$ as follows: if $f,g$ are two elements of $H(A)$ and $\{x,y\} \in A$, one has $f(x) + f(y) = d(x,y) = g(x) + g(y)$, hence $f(y) - g(y) = -(f(x) - g(x))$, which means that the difference 
$f-g$ has alternating sign along all edge paths in the graph $(X,A)$. 
Therefore, there is either none or exactly one degree of freedom 
for the values of $f \in H(A)$ on every connected component of $(X,A)$, 
depending upon whether or not the component contains an odd cycle. 
We call such components (viewed as subsets of $X$) {\em odd} or 
{\em even $A$-components}, respectively. 

If $(X,d)$ is a finite metric space, $\mathrm{E}(X,d)$ is a finite polyhedral complex. If $(X,d)$ is infinite, we say that $(X,d)$ satisfies the \textit{local rank condition (LRC)} if and only if for every $f \in \mathrm{E}(X,d)$, there exist $\eps,N > 0$ such that for all $g \in \mathrm{E}'(X,d)$ with $ \|f - g\|_\infty < \epsilon $, one has $\mathrm{rank}(A(g))\le N$. Recall (cf. \cite[Theorem~4.4]{Lan}) that if $(X,d)$ is a metric space with integer-valued metric and satisfying the (LRC), then $\mathrm{E}(X,d)=\mathrm{E}'(X,d)$. In this case, let
\[
P(A) = \mathrm{E}(X,d) \cap H(A) = \Delta(X,d) \cap H(A).
\]
The family $\{P(A)\}_{A \in \cA(X)}$ then defines a polyhedral structure on $\mathrm{E}(X,d)$ and note that $P(A')$ is a face of $P(A)$ if and only if $A \sub A'$.

In order to prove Theorem~\ref{Thm:t11}, we need to be able to decompose any pseudometric $d$ on a set $X$ in a way that is coherent with the structure of $\mathrm{E}(X,d)$ which we have just introduced. The \textit{isolation index} of a pair $S:=\{A,B\}$ of non-empty subsets  with respect to a pseudometric $d$ on $X$ is the non-negative number $\alpha^d_S$ (equivalently $\alpha^d_{\{A,B\}}$ or simply $\alpha_{S}$) given by
\begin{align}\label{eq:IsoI}
\alpha^d_{S} := \frac{1}{2}  \inf_{ \substack{ a,a' \in A \\ b,b' \in B}}  \Bigl[  \max \bigl \{d(a,b) + d(a',b'), d(a',b)+ d(a,b'),d(a,a') + d(b,b') \bigr \}& \nonumber \\ 
-d(a,a') - d(b,b') &\Bigr].
\end{align}

Moreover, we call a pseudometric $d_0$ on $X$ \textit{split-prime} if $\alpha_{S}^{d_0}=0$ for any split $S$ of $X$. Note that by Lemma~\ref{Lem:FinitelyManySeparatingSplits}, there are for any integer-valued pseudometric only finitely many $d$-splits separating any pair of points. 

\begin{Thm}\label{Thm:thm2} 
Let $(X,d)$ be a pseudometric space with integer-valued pseudometric, let $\mathcal{S}_X$ be the set of all splits of $X$ and let $\mathcal{S}$ be the set all $d$-splits. Let
\[
\left\lbrace{\begin{matrix} 
\lambda_S \in (-\infty,\alpha^{d}_S] &\text{if }  S \in \mathcal{S} , \\ \\
\lambda_S = 0 &\text{if } S \in \mathcal{S}_X \setminus \mathcal{S}.
\end{matrix}}\right.
\]
Then,
\[
\widetilde{d} := d - \sum_{S \in \mathcal{S} } \lambda_{S} \delta_S
\]
is a pseudometric such that for every split $S \in \mathcal{S}_X$, one has
\[
\alpha_S^{\widetilde{d}} = \alpha_S^d - \lambda_S.
\]
In particular, there is a split-sprime pseudometric $d_0$ such that 
\[
d = d_0 +  \sum_{S \in \mathcal{S}_X } \alpha^{d}_{S} \delta_S.
\]
\end{Thm}

The decomposition given by Theorem \ref{Thm:thm2} can be characterized uniquely in a corollary to Theorem~\ref{Thm:thm4}:

\begin{Cor}\label{Cor:CorEI}
Let $(X,d)$ be a metric space with integer-valued metric satisfying the (LRC). Let $\mathcal{S}$ be the family of all $d$-splits of $X$ so that $d = d_0 + \sum_{S \in \mathcal{S}} \alpha_S \delta_S$ and let $\lambda_S \in [0,\alpha_S]$ for every $S \in \mathcal{S}$. Then, setting $d_1 := d - \sum_{S \in \mathcal{S}} \lambda_S \delta_S$, we have
\begin{equation}\label{eq:InjectiveHullDecomposition}
\mathrm{E}(X,d) \subset \R^X \cap \left( \mathrm{E}(X,d_1) + \sum_{S \in \mathcal{S}} \lambda_S \mathrm{E}(X,\delta_S) \right).
\end{equation}
Moreover, for any split $S=\{A,B\}$ of $X$ and any $\lambda_S>0$ such that $d = d_1 + \lambda_S \delta_{S}$, if $\mathrm{E}(X,d) \subset \mathrm{E}(X,d_1) + \lambda_S \mathrm{E}(X,\delta_{S})$, then the following hold:
\begin{enumerate}[$(i)$]
\item $S$ is a $d$-split of $X$ and
\item $\lambda_S \le \alpha_S$.
\end{enumerate}
\end{Cor}

The present work is divided into three main parts: Section 2 deals with the generalization of the decomposition theory of Bandelt and Dress (cf. \cite{BanD}) to (possibly) infinite metric spaces with integer-valued metric. Section 3 deals with the proof of Theorem \ref{Thm:t11}, we adapt and generalize the arguments of \cite{HubKM} to infinite metric spaces and infinite split systems. Section 4 starts with the observation that the Buneman complex $B(\mathcal{S},\alpha)$, which is a well-known object in discrete mathematics (cf. \cite{DreHM} and the references there) satisfies the CAT(0) link condition and continues with the proof that $\mathrm{E}(X,d)$ satisfies this same condition. Finally, Section 5 deals with several examples.

\section{Decomposition Theory}
It is easy to see that $\mathrm{E}(X,d)$ as defined at the beginning of the Introduction is a subset of
\[
\Delta_1(X,d) := \{f \in \Delta(X,d): \text{$f$ is $1$-Lipschitz}\}.
\]
Note that a function $f \in \R^X$ belongs to $\Delta_1(X,d)$ if and only if 
\begin{equation} \label{eq:f-d}
\|f - d_x\|_\infty 
= f(x) \quad \text{for all $x \in X$.}
\end{equation}
The metric $d_{\infty}(f,g) := \|f - g\|_\infty$ on $\Delta_1(X,d)$ is thus well-defined since for any $x \in X$, one has
\[
\|f - g\|_\infty \le \|f - d_x\|_\infty + \|g - d_x\|_\infty = f(x) + g(x) < \infty.
\] 
The set $\mathrm{E}(X,d) \sub \Delta_1(X,d)$ is equipped with the induced metric and one has the canonical isometric map
\[
\e \colon (X,d) \to \mathrm{E}(X,d), \quad \e(x) = d_x.
\]
In case $(X,d)$ is a metric space, $\e$ is an isometric embedding and Isbell showed that $(\e,\mathrm{E}(X,d))$ is indeed an injective hull of $X$. That is, $\mathrm{E}(X,d)$ is an injective metric space, and every isometric embedding of $X$ into another injective metric space factors through $\e$. A metric space $(X,d)$ is called \textit{injective} if for any isometric embedding $i \colon Y \to Z$ of metric spaces and any $1$-Lipschitz (i.e., distance nonincreasing) map $f \colon Y \to X$ there exists a $1$-Lipschitz map $g \colon Z \to X$, so that $g \circ i = f$ (cf. \cite[Section~9]{AdaHS} for the general categorical definition). For a recent survey of injective metric spaces, we refer to \cite[Section~2]{Lan}.

Let $(X,d)$ be any pseudometric space. A \textit{partial split} $S = \{A,B\}$ of $X$ is a pair of non-empty subsets of $X$ such that $A \cap B = \emptyset$. If in addition $X = A \cup B$ holds, then $S = \{A,B\}$ is a split of $X$. A \textit{partial $d$-split} is a partial split $S = \{A,B\}$ for which $\alpha^d_S >0$. For any $\{a,a'\} \subset A$ and $\{b,b'\} \subset B$, let
\begin{align*}
\beta_{ \{ \{a,a'\},\{b,b'\} \} }^d := \frac{1}{2} \bigl[ \max\{d(a,b) + d(a',b'), d(a',b) + d(a,b'),d(a,a') + d(b,b')\}& \\
- d(a,a') - d(b,b') & \bigr] .
\end{align*}
When the reference to the pseudometric is unnecessary, we shall omit it and write simply $\alpha_{\{A,B\}}$ as well as $\beta_{ \{ \{a,a'\},\{b,b'\} \} }$. Note that for any pseudometric $d$ on $X$, one has 
\begin{equation}\label{eq:AlphaBeta}
\alpha^d_{\{\{r,s\},\{t,u\}\}}=\beta^d_{\{\{r,s\},\{t,u\}\}}.
\end{equation}
Indeed, 
\[
d(r,t) + d(s,u) - d(r,s) - d(t,u) 
\le d(r,t) + d(u,r) - d(t,u)
= 2\beta^d_{\{\{r\},\{t,u\}\}}
\]
and
\[
d(r,t) + d(s,u) - d(r,s) - d(t,u) 
\le d(r,t) + d(u,r) - d(t,u)
\le 2d(r,t)
= 2\beta^d_{\{\{r\},\{t\}\}}.
\]

\begin{Lem}\label{Thm:BetaAlphaPlusAlpha}
Let $(X,d)$ be a pseudometric space and $\{A_0,B_0\}$ a partial $d$-split of $X$. For any $\{a,a'\} \subset A_0$, $\{b,b'\} \subset B_0$ and $x \in X \setminus ( A_0 \cup B_0)$, one has
\begin{equation}\label{eq:myequation0}
\beta_{ \{ \{a,a'\},\{b,b'\} \} } \ge \alpha_{ \{ \{a,a',x\}  ,\{b,b'\} \} } + \alpha_{ \{ \{a,a'\}  ,\{b,b',x\}  \} }.
\end{equation}
\end{Lem}
\begin{proof}
Assume that \eqref{eq:myequation0} fails for some $a_1,a_2,b_1,b_2,x$, then all three quantities must be positive. For simplicity, we write $xy$ for $d(x,y)$. Let $\{i,j\} = \{1,2\}$ be so that
\[
\beta_{ \{ \{a_1,x\},\{b_1,b_2\} \} } = \frac{1}{2}\Bigl(a_1b_j + xb_i - a_1x - b_1b_2 \Bigr).
\]
It follows that
\begin{align*}
&\frac{1}{2}\Bigl(a_1b_j + xb_i - a_1x - b_1b_2 + \max \{ a_1x + a_2b_i,a_1b_i+a_2x \} -a_1a_2 - xb_i\Bigr) = \\
& \beta_{ \{ \{a_1,x\},\{b_1,b_2\} \} } + \beta_{ \{ \{a_1,a_2\},\{x,b_i\} \} } \ge \\
&\alpha_{ \{ \{a_1,a_2,x\} , \{ b_1,b_2 \}  \} } + \alpha_{ \{ \{a_1,a_2\}  ,\{b_1,b_2,x\}  \} } > \\
&\beta_{ \{ \{a_1,a_2\},\{b_1,b_2\} \} } =\\
&\frac{1}{2}\Bigl(\max \{ a_1b_1+a_2b_2,a_1b_2+a_2b_1 \} -a_1a_2 - b_1b_2\Bigr).
\end{align*}
Hence
\[
\delta := a_1b_j - a_1x + \max \{ a_1x+a_2b_i,a_1b_i+a_2x \} > \max \{ a_1b_1 + a_2b_2,a_1b_2 + a_2b_1 \}.
\]
The above strict inequality can only hold if $a_1b_i+a_2x > a_1x+a_2b_i$ which implies that $\delta= a_1b_j - a_1x + a_1b_i + a_2x$.  Therefore, one has
\[
a_1b_1 + a_1b_2 - a_1x +a_2x = \delta > \max \{ a_1b_1 + a_2b_2, a_1b_2 + a_2b_1\}.
\]
Hence for each $k \in \{1,2\}$, one has
\[
a_1b_k + a_2x > a_1x + a_2b_k.
\]
By interchanging the role of $a_1$ and $a_2$, we also obtain the reverse strict inequality and this is a contradiction. This proves \eqref{eq:myequation0}. 
\end{proof}

\begin{Thm}\label{Thm:IsolationIndexExtension} 
Let $(X,d)$ be a finite pseudometric space and let $\{A_0,B_0\}$ be a partial $d$-split. Then 
\[
\sum \bigl \{ \alpha_{\{A,B\}} : \{A,B\} \text{ is a } d\text{-split extending } \{A_0,B_0\} \bigr \} \le \alpha_{\{A_0,B_0\}}.
\]
\end{Thm}
\begin{proof} 
Note that since $X$ is finite, there are $\{a,a'\} \subset A_0$ and $\{b,b'\} \subset B_0$ such that 
\[
\alpha_{\{A_0,B_0\}} = \beta_{ \{ \{a,a'\},\{b,b'\} \} }
\]
and thus we obtain
\begin{equation}\label{eq:IsoI1}
\alpha_{\{A_0,B_0\}} \ge \alpha_{ \{ A_0 \cup \{ x \}  ,B_0 \} } + \alpha_{ \{ A_0  ,B_0 \cup \{ x \}  \} }.
\end{equation}
By induction, the desired result follows.
\end{proof}

Note now that if $(X,d)$ is a pseudometric space with integer-valued pseudometric, then for any split $S$, one has 
\[
\alpha_S \in [0,\infty) \cap \frac{1}{2}\Z.
\]

\begin{Lem}\label{Lem:FinitelyManySeparatingSplits}
Let $(X,d)$ is a pseudometric space with integer-valued pseudometric. For every $x,y \in X$, there are at most $2 d(x,y)$ distinct $d$-splits $S$ satisfying $S(x) \neq S(y)$.
\end{Lem}
\begin{proof}
Assume by contradiction that one can find $2 d(x,y)+1$ distinct $d$-splits $S_i$ such that $S_i(x) \neq S_i(y)$. For each $i \in \{1,\dots, 2 d(x,y)+1\}$, pick $a_i,a_i' \in S_i(x)$ and $b_i,b_i' \in S_i(y)$ such that $\alpha_{\{ \{a_i,a_i'\},\{b_i,b_i'\} \}}=\alpha_{S_i} \ge 1/2$. Set $Z:= \{x,y\} \cup \bigcup_i \{a_i,a_i'\} \cup \bigcup_i \{b_i,b_i'\} \subset X$, one has:
\[
d(x,y)+\frac{1}{2} \le \sum_i \alpha_{S_i} \le \sum_i \alpha_{\{ S_i(x) \cap Z, S_i(y) \cap Z \}}
\]
and by Theorem~\ref{Thm:IsolationIndexExtension} applied to $x,y \in Z$, the right-hand side is less than or equal to $\alpha_{\{ \{x\},\{y\} \}} = d(x,y)$, which is a contradiction.
\end{proof}

Furthermore, it is easy to see that if $S=\{A,B\}$ is a $d$-split of $X$ such that $S(x)\neq S(y)$ for some pair of points $x,y$ such that $d(x,y)=1$, then one has $C(x,y):=\{z \in X : d(x,y) + d(y,z) = d(x,z)\}  \subset S(y)$ and $C(y,x) \subset S(x)$. Note however that one might have $X = C(x,y) \cup C(y,x)$ without $\{C(x,y),C(y,x)\}$ being a $d$-split, this is the case for instance if $X=\mathrm{Cayley}(\Delta(3,3,3))$, i.e. the vertices of the plane tessellation by hexagons, endowed with the path-metric. Indeed, in this case, $X$ is bipartite and thus $X = C(x,y) \cup C(y,x)$. However, $d$ is in this case a split-prime metric.

\begin{Thm}\label{Thm:thm1} 
Let $(X,d)$ be a pseudometric space with integer-valued pseudometric and let $\{A_0,B_0\}$ be a partial $d$-split. Then 
\[
\sum \bigl \{ \alpha_{\{A,B\}} : \{A,B\} \text{ is a } d\text{-split extending } \{A_0,B_0\} \bigr \} \le \alpha_{\{A_0,B_0\}}.
\]
\end{Thm}
\begin{proof} 
Note that since $X$ has integer-valued metric, there are $\{a,a'\} \subset A_0$ and $\{b,b'\} \subset B_0$ such that 
\[
\alpha_{\{A_0,B_0\}} = \beta_{ \{ \{a,a'\},\{b,b'\} \} }
\]
and thus using \ref{eq:myequation0}, we obtain
\begin{equation}\label{eq:IsoI1}
\alpha_{\{A_0,B_0\}} \ge \alpha_{ \{ A_0 \cup \{ x \}  ,B_0 \} } + \alpha_{ \{ A_0  ,B_0 \cup \{ x \}  \} }.
\end{equation}

Since by Lemma~\ref{Lem:FinitelyManySeparatingSplits}, there are only finitely many $d$-splits separating any pair of points in $X$, there are only finitely many $d$-splits $\{S_i\}_{i \in \{1,\dots,m\}}$ extending $\{ A_0,B_0\}$, i.e. such that one has $S_i :=\{A_i,B_i\}$ where $A_0 \subset A_i$ and $B_0 \subset B_i$. For every $1 \le i < j \le m $, choose $x_{ij} \in X \setminus (A_0 \cup B_0)$ such that either $x_{ij} \in A_i \cap B_j$ or $x_{ij} \in B_i \cap A_j$ (in particular neither $x_{ij} \in A_i \cap A_j$ nor $x_{ij} \in B_i \cap B_j$). Writing $Z:=(x_{ij})_{1 \le i < j \le m}$ and applying \eqref{eq:IsoI1} recursively for each element of $Z$, we obtain:
\[
\alpha_{\{A_0,B_0\}} \ge \sum_{Z_A \subset Z} \alpha_{\{A_0 \cup Z_A,B_0 \cup (Z \setminus Z_A) \}}.
\]
Note that by choice of $Z$, there is for any $\{A_i,B_i\}$ a unique $Z_A \subset Z$ (namely $Z_A := Z \cap A_i$) such that $\{A_i,B_i\}$ extends $\{A_0 \cup Z_A,B_0 \cup (Z \setminus Z_A) \}$ which implies
\[
\alpha_{\{A_0 \cup Z_A,B_0 \cup (Z \setminus Z_A) \}} \ge \alpha_{\{A_i,B_i\}}.
\]
It follows that
\[
\alpha_{\{A_0,B_0\}} \ge \sum_{i \in \{1,\dots,m\}} \alpha_{\{A_i,B_i\}}.
\]
This concludes the proof.
\end{proof}

\begin{Lem}\label{Lem:DecompositionInduction}
Let $(X,d)$ be a pseudometric space with integer-valued pseudometric. Let $\mathcal{S}$ be the set of all $d$-splits and let $\widetilde{\mathcal{S}} \subset \mathcal{S}$ be any finite subset. If $\lambda_S \in (-\infty,\alpha^d_S]$ for every $S \in \widetilde{\mathcal{S}}$ and $\lambda_S := 0$ for every other split, then $\widetilde{d} := d - \sum_{S \in \widetilde{\mathcal{S}} } \lambda_{S} \delta_S$ is a pseudometric such that for every split $S$ of $X$, one has
\[
\alpha_S^{\widetilde{d}} = \alpha_S^d - \lambda_S.
\]
\end{Lem}

\begin{proof}
Let us denote by $\mathcal{S}$ the family of all $d$-splits and let $\widetilde{S}$ be any subset of $\mathcal{S}$. Moreover, $\widetilde{d} := d - \sum_{S \in \widetilde{S}} \lambda_S \delta_S$ where for each $S \in \widetilde{S}$, we require $\lambda_S \le \alpha^d_S$ (for simplicity, we shall denote $d(x,y)$ simply by $xy$).

We first treat the case where $|\widetilde{S}| < \infty$. It suffices to proves the assertion for 
\[
\widetilde{d} := d - \lambda \delta_{\{A_0,B_0\}}
\]
where $\{A_0,B_0\} \in \mathcal{S}$ and $\lambda \le \alpha^d_{\{A_0,B_0\}}$. We verify that $\widetilde{d}$ is indeed a pseudometric by showing the triangle inequality. Let $x,y,z \in X$ and assume without loss of generality that $x,y \in A_0$. If $x,y,z \in A_0$, then $d$ and $\widetilde{d}$ agree on $\{x,y,z\}$ and we are done. Otherwise, $z \in B_0$, in which case we get
\[
\widetilde{d}(x,z) = xz - \lambda \le xy + yz -\lambda =  \widetilde{d}(x,y) + \widetilde{d}(y,z).
\]
On the other side, since $\lambda \le \alpha^d_{\{A_0,B_0\}} \le \beta^d_{\{\{x,y\},\{z\}\}} = \frac{1}{2} ( xz + yz - xy)$, we obtain by rearranging
\[
xy \le xz - \lambda + yz - \lambda = \widetilde{d}(x,z) + \widetilde{d}(y,z).
\]
Thus, $\widetilde{d}$ is a pseudometric.

Let $\{x,y\}$ and $\{z,w\}$ be two disjoint subsets of $X$. If $\{ \{x,y\},\{z,w\} \}$ extends to $\{A_0,B_0\}$, then clearly
\begin{equation}\label{eq:myequation1}
\beta^{\widetilde{d}}_{\{\{x,y\},\{z,w\}\}} = \beta^d_{\{\{x,y\},\{z,w\}\}} - \lambda.
\end{equation}
Now, we prove that if $\{ \{x,y\},\{z,w\} \}$ does not extend to $\{A_0,B_0\}$, then 
\begin{equation}\label{eq:myequation2}
\beta^{\widetilde{d}}_{\{\{x,y\},\{z,w\}\}} = \beta^d_{\{\{x,y\},\{z,w\}\}}.
\end{equation}
First, if either $A_0$ or $B_0$ contains at least three of $x,y,z,w$, then \eqref{eq:myequation2} clearly holds. We may thus assume without loss of generality that $x,z \in A_0$ and $y,w \in B_0$. Since 
\[
\beta^d_{\{\{x,z\},\{y,w\}\}} 
= \frac{1}{2} \Bigl( \max \{ xy + zw , xw + yz \} - xz - yw \Bigr)
\ge \alpha^d_{\{A_0,B_0\}}
\ge \lambda,
\]
we obtain that $\max \{ xy + zw - 2 \lambda, xw + yz - 2 \lambda \} \ge xz + yw$. Hence
\begin{align*}
\beta^{\widetilde{d}}_{\{\{x,y\},\{z,w\}\}} 
&= \frac{1}{2} \Bigl( \max \{ xz + yw, xw + yz - 2 \lambda, xy + zw - 2 \lambda\} - xy - zw + 2 \lambda \Bigr) \\
&= \frac{1}{2} \Bigl( \max \{ xw + yz - 2 \lambda , xy + zw - 2 \lambda\} - xy - zw + 2 \lambda \Bigr) \\
&= \frac{1}{2} \Bigl( \max \{ xz + yw, xw + yz, xy + zw \} - xy - zw  \Bigr) \\
&= \beta^d_{\{\{x,y\},\{z,w\}\}},
\end{align*}
as required. Finally, it remains to prove that for every split $\{A,B\}$ of $X$, one has
\[
\alpha^{\widetilde{d}}_{\{A,B\}} \; = \; \left\lbrace{\begin{matrix} 
\alpha^d_{\{A,B\}}- \lambda &\text{if }  \{A,B\}=\{A_0,B_0\} , \\ \\
\alpha^d_{\{A,B\}} &\text{otherwise}.
\end{matrix}}\right.
\]
By \eqref{eq:myequation1}, one has $ \alpha^{\widetilde{d}}_{\{A_0,B_0\}} = \alpha^d_{\{A_0,B_0\}} -\lambda$. Let now $\{A,B\}$ be a split of $X$ different from $\{A_0,B_0\}$. Since $X$ has integer-valued metric (it is enough to assume the metric to be taking a discrete set of values) and by \eqref{eq:AlphaBeta}, there are $a,a' \in A$ and $b,b' \in B$ such that 
\[
\alpha^d_{\{A,B\}} = \alpha^d_{\{\{a,a'\},\{b,b'\}\}}=\beta^d_{\{\{a,a'\},\{b,b'\}\}}.
\]
Since $\{A_0,B_0\}$ is a $d$-split, if $\alpha^d_{\{A,B\}}=0$, then $\{A_0,B_0\}$ cannot extend $\{\{a,a'\},\{b,b'\}\}$, and if $\alpha^d_{\{A,B\}}>0$, then $\{A_0,B_0\}$ cannot extend $\{\{a,a'\},\{b,b'\}\}$ either by Theorem \ref{Thm:thm1}. Hence by \eqref{eq:myequation2}, one has
\[
\alpha^d_{\{A,B\}} 
= \beta^d_{\{\{a,a'\},\{b,b'\}\}} 
= \beta^{\widetilde{d}}_{\{\{a,a'\},\{b,b'\}\}}
\ge \alpha^{\widetilde{d}}_{\{A,B\}}.
\]
To prove the reverse inequality, assume that $a,a' \in A$ and $b,b' \in B$ are such that 
\[
\alpha^{\widetilde{d}}_{\{A,B\}} = \alpha^{\widetilde{d}}_{\{\{a,a'\},\{b,b'\}\}}=\beta^{\widetilde{d}}_{\{\{a,a'\},\{b,b'\}\}}.
\]
If $\{A_0,B_0\}$ does not extend $\{\{a,a'\},\{b,b'\}\}$, one has by \eqref{eq:myequation2} that
\[
\alpha^d_{\{A,B\}} \le \beta^d_{\{\{a,a'\},\{b,b'\}\}} = \beta^{\widetilde{d}}_{\{\{a,a'\},\{b,b'\}\}}.
\]
Now, if $\{A_0,B_0\}$ extends $\{\{a,a'\},\{b,b'\}\}$, it follows from Theorem \ref{Thm:thm1} and \eqref{eq:myequation1} that
\begin{align*}
\alpha^d_{\{A,B\}} 
\le \alpha^d_{\{A,B\}} + \alpha^d_{\{A_0,B_0\}} -\lambda
\le \alpha^d_{\{\{a,a'\},\{b,b'\}\}} -\lambda
&\le \beta^d_{\{\{a,a'\},\{b,b'\}\}} -\lambda \\
&= \beta^{\widetilde{d}}_{\{\{a,a'\},\{b,b'\}\}} \\
&= \alpha^{\widetilde{d}}_{\{A,B\}}.
\end{align*}
Using induction, this concludes the case $|\widetilde{S}| < \infty$.
\end{proof}

\begin{proof}[Proof of Theorem~\ref{Thm:thm2}]
We now go on to the case where $|\widetilde{S}| = \infty$. For any $Y:=\{x,y,z\} \subset X$, we can consider the set $\widetilde{\mathcal{S}}_Y \subset \widetilde{\mathcal{S}}$ of all splits in $\widetilde{\mathcal{S}}$ that restrict to a split of $Y$. Since $\widetilde{\mathcal{S}} \subset \mathcal{S}$, $\widetilde{\mathcal{S}}_Y$ is a finite set by Lemma~\ref{Lem:FinitelyManySeparatingSplits}. We set 
\begin{equation}\label{eq:FormulaDYMinusFinite}
\widetilde{d}_Y := d - \sum_{S \in \widetilde{\mathcal{S}}_Y} \lambda_S \delta_S.
\end{equation}
Now, $\widetilde{d}|_{Y \times Y}$ coincides with $\widetilde{d}_Y|_{Y \times Y}$  and the latter satisfies the triangle inequality as we proved in the case $|\widetilde{S}| < \infty$, hence so does the former. It follows that $\widetilde{d}$ is a pseudometric. Let now $S_0:=\{A_0,B_0\}$ be any split of $X$, we consider three cases:

\begin{enumerate}[$(a)$]
\item Assume that $S_0 \in \widetilde{\mathcal{S}}$. Note that $\widetilde{d}$ is in general not integer-valued. For any $\eps > 0$, we can choose $\widetilde{a},\widetilde{a}' \in A_0$ and $\widetilde{b},\widetilde{b}' \in B_0$ such that 
\[
\alpha^{\widetilde{d}}_{\{A_0,B_0\}} \ge \alpha^{\widetilde{d}}_{\{\{\widetilde{a},\widetilde{a}'\},\{\widetilde{b},\widetilde{b}'\}\}} - \eps.
\]
Moreover, we can choose $a,a' \in A_0$ and $b,b' \in B_0$ such that 
\[
\alpha^{d}_{\{A_0,B_0\}} = \alpha^d_{\{\{a,a'\},\{b,b'\}\}}.
\]
By Theorem \ref{Thm:thm1}, $\{A_0,B_0\}$ is the unique $d$-split extending $\{\{a,a'\},\{b,b'\}\}$. Furthermore, set 
\[
Y := \{\widetilde{a},\widetilde{a}',\widetilde{b},\widetilde{b}',a,a',b,b'\}.
\]
We can assume without loss of generality that $S_0 \subset \widetilde{\mathcal{S}}_Y$. We have 
\begin{align*}
\alpha^{\widetilde{d}}_{\{A_0,B_0\}} 
\ge \alpha^{\widetilde{d}}_{\{\{\widetilde{a},\widetilde{a}'\},\{\widetilde{b},\widetilde{b}'\}\}} - \eps 
&\ge \alpha^{\widetilde{d}}_{\{\{\widetilde{a},\widetilde{a}',a,a'\},\{\widetilde{b},\widetilde{b}',b,b'\}\}}- \eps \\ &=\alpha^{\widetilde{d}_Y}_{\{\{\widetilde{a},\widetilde{a}',a,a'\},\{\widetilde{b},\widetilde{b}',b,b'\}\}}- \eps
\end{align*}
Note that $d$-splits of $X$ that restrict to splits of $Y$ are $d|_{Y \times Y}$-splits. Let $\widetilde{\mathcal{S}}_Y$ the set $\widetilde{\mathcal{S}}_Y \subset \widetilde{\mathcal{S}}$ of all splits in $\widetilde{\mathcal{S}}$ that restrict to a split of $Y$. Since $S_0 \subset \widetilde{\mathcal{S}}_Y$, $S_0$ is a $d$-split and it is the unique $d$-split of $X$ that restricts to $\{\{\widetilde{a},\widetilde{a}',a,a'\},\{\widetilde{b},\widetilde{b}',b,b'\}\}$ on $Y$, we can apply the Lemma~\ref{Lem:DecompositionInduction} with $\widetilde{d}_Y$ given by \eqref{eq:FormulaDYMinusFinite} to deduce that
\[
\alpha^{\widetilde{d}_Y|_{Y \times Y}}_{\{\{\widetilde{a},\widetilde{a}',a,a'\},\{\widetilde{b},\widetilde{b}',b,b'\}\}} 
=\alpha^{d|_{Y \times Y}}_{\{\{\widetilde{a},\widetilde{a}',a,a'\},\{\widetilde{b},\widetilde{b}',b,b'\}\}} - \lambda_{\{A_0,B_0\}}.
\]
Finally, the right-hand side is equal to 
\[
\alpha^d_{\{\{a,a'\},\{b,b'\}\}} - \lambda_{\{A_0,B_0\}} =
\alpha^d_{\{A_0,B_0\}} - \lambda_{\{A_0,B_0\}}.
\]
Hence, $\alpha^{\widetilde{d}}_{\{A_0,B_0\}} \ge \alpha^d_{\{A_0,B_0\}} - \lambda_{\{A_0,B_0\}} -\eps$. Since this holds for any $\eps >0$, we get $\alpha^{\widetilde{d}}_{\{A_0,B_0\}} \ge \alpha^d_{\{A_0,B_0\}} - \lambda_{\{A_0,B_0\}}$. The other inequality is obtained similarly, noting that
\begin{align*}
\alpha^{\widetilde{d}}_{\{A_0,B_0\}} 
\le \alpha^{\widetilde{d}}_{\{\{\widetilde{a},\widetilde{a}'\},\{\widetilde{b},\widetilde{b}'\}\}} 
&\le \alpha^{\widetilde{d}}_{\{\{\widetilde{a},\widetilde{a}',a,a'\},\{\widetilde{b},\widetilde{b}',b,b'\}\}} + \eps \\ 
&=\alpha^{\widetilde{d}_Y}_{\{\{\widetilde{a},\widetilde{a}',a,a'\},\{\widetilde{b},\widetilde{b}',b,b'\}\}} + \eps
\end{align*}

\item Assume that $S_0 \in \mathcal{S} \setminus \widetilde{\mathcal{S}}$. Let $Y$ and $\mathcal{S}_Y$ be defined as in the former case. Similarly to the former case, we have 
\begin{align*}
\alpha^{\widetilde{d}}_{\{A_0,B_0\}} 
\ge \alpha^{\widetilde{d}}_{\{\{\widetilde{a},\widetilde{a}'\},\{\widetilde{b},\widetilde{b}'\}\}} - \eps
&\ge \alpha^{\widetilde{d}}_{\{\{\widetilde{a},\widetilde{a}',a,a'\},\{\widetilde{b},\widetilde{b}',b,b'\}\}} - \eps \\
&=\alpha^{\widetilde{d}_Y}_{\{\{\widetilde{a},\widetilde{a}',a,a'\},\{\widetilde{b},\widetilde{b}',b,b'\}\}} - \eps.
\end{align*}
as well as
\begin{align*}
\alpha^{\widetilde{d}}_{\{A_0,B_0\}} 
\le \alpha^{\widetilde{d}}_{\{\{\widetilde{a},\widetilde{a}'\},\{\widetilde{b},\widetilde{b}'\}\}} 
&\le \alpha^{\widetilde{d}}_{\{\{\widetilde{a},\widetilde{a}',a,a'\},\{\widetilde{b},\widetilde{b}',b,b'\}\}} + \eps \\
&=\alpha^{\widetilde{d}_Y}_{\{\{\widetilde{a},\widetilde{a}',a,a'\},\{\widetilde{b},\widetilde{b}',b,b'\}\}} + \eps.
\end{align*}
Since $\widetilde{\mathcal{S}}_Y \subset \widetilde{\mathcal{S}}$, we have $S_0 \notin \widetilde{\mathcal{S}}_Y$ by assumption. Since $S_0$ is the unique $d$-split of $X$ that restricts to $\{\{\widetilde{a},\widetilde{a}',a,a'\},\{\widetilde{b},\widetilde{b}',b,b'\}\}$ on $Y$, it follows that no element of $\widetilde{\mathcal{S}}_Y$ restricts to $\{\{\widetilde{a},\widetilde{a}',a,a'\},\{\widetilde{b},\widetilde{b}',b,b'\}\}$ on $Y$. We can apply Lemma~\ref{Lem:DecompositionInduction} with $\widetilde{d}_Y$ given by \eqref{eq:FormulaDYMinusFinite} to obtain
\[
\alpha^{\widetilde{d}_Y|_{Y \times Y}}_{\{\{a,a',\bar{a},\bar{a}'\},\{b,b',\bar{b},\bar{b}'\}\}} =
\alpha^{d|_{Y \times Y}}_{\{\{\widetilde{a},\widetilde{a}',a,a'\},\{\widetilde{b},\widetilde{b}',b,b'\}\}}.
\]
The right-hand side is now equal to
\[
\alpha^d_{\{\{a,a'\},\{\bar{b},\bar{b}'\}\}} =
\alpha^d_{\{A_0,B_0\}}.
\]
Since $\eps>0$ can be taken arbitrarily small, we thus get $\alpha^{\widetilde{d}}_{\{A_0,B_0\}} = \alpha^d_{\{A_0,B_0\}}$.

\item Assume that $S_0 \notin \mathcal{S}$, let $Y$ and $\mathcal{S}_Y$ be defined as in the former case, we get
\begin{align*}
\alpha^{\widetilde{d}}_{\{A_0,B_0\}} 
\ge \alpha^{\widetilde{d}}_{\{\{\widetilde{a},\widetilde{a}'\},\{\widetilde{b},\widetilde{b}'\}\}} - \eps
&\ge \alpha^{\widetilde{d}}_{\{\{\widetilde{a},\widetilde{a}',a,a'\},\{\widetilde{b},\widetilde{b}',b,b'\}\}} - \eps \\
&=\alpha^{\widetilde{d}_Y}_{\{\{\widetilde{a},\widetilde{a}',a,a'\},\{\widetilde{b},\widetilde{b}',b,b'\}\}} - \eps.
\end{align*}
as well as
\begin{align*}
\alpha^{\widetilde{d}}_{\{A_0,B_0\}} 
\le \alpha^{\widetilde{d}}_{\{\{\widetilde{a},\widetilde{a}'\},\{\widetilde{b},\widetilde{b}'\}\}} 
&\le \alpha^{\widetilde{d}}_{\{\{\widetilde{a},\widetilde{a}',a,a'\},\{\widetilde{b},\widetilde{b}',b,b'\}\}} + \eps \\
&=\alpha^{\widetilde{d}_Y}_{\{\{\widetilde{a},\widetilde{a}',a,a'\},\{\widetilde{b},\widetilde{b}',b,b'\}\}} + \eps.
\end{align*}
Note that $S_0$ restricts to $\{\{\widetilde{a},\widetilde{a}',a,a'\},\{\widetilde{b},\widetilde{b}',b,b'\}\}$ on $Y$  and since 
\[
\alpha^{d}_{\{A_0,B_0\}} = \alpha^d_{\{\{a,a'\},\{b,b'\}\}}=0,
\]
it follows that $\{\{\widetilde{a},\widetilde{a}',a,a'\},\{\widetilde{b},\widetilde{b}',b,b'\}\}$ is not a $d|_{Y \times Y}$-split of $Y$. Hence by the finite case with $\widetilde{d}_Y$ given by \eqref{eq:FormulaDYMinusFinite}, it follows that
\begin{align*}
\alpha^{\widetilde{d}_Y|_{Y \times Y}}_{\{\{\widetilde{a},\widetilde{a}',a,a'\},\{\widetilde{b},\widetilde{b}',b,b'\}\}} 
= \alpha^{d|_{Y \times Y}}_{\{\{\widetilde{a},\widetilde{a}',a,a'\},\{\widetilde{b},\widetilde{b}',b,b'\}\}} 
&= \alpha^{d}_{\{\{\widetilde{a},\widetilde{a}',a,a'\},\{\widetilde{b},\widetilde{b}',b,b'\}\}} \\
&= \alpha^{d}_{\{A_0,B_0\}} \\
&= 0.
\end{align*}
Since $\eps>0$ can be taken arbitrarily small, we thus get $\alpha^{\widetilde{d}}_{\{A_0,B_0\}} = 0 = \alpha^d_{\{A_0,B_0\}}$.
\end{enumerate}
This concludes the proof.
\end{proof}

We say that a collection $\mathcal{S}$ of splits of $X$ is \textit{weakly compatible} if there are no four points $\{x_0,x_1,x_2,x_3\} \subset X$ and three splits $\{S_1,S_2,S_3\} \subset \mathcal{S}$ such that for any $i,j\in \{1,2,3\}$, one has
\[
S_i(x_0) = S_i(x_j) \Longleftrightarrow i = j.
\]

It is clear from \eqref{eq:AlphaBeta} that for a pseudometric space $(X,d)$ and every set of four different points $\{x_0,x_1,x_2,x_3\} \subset X$ at least one of 
\[
\alpha_{ \{ \{x_0,x_1\}, \{x_2,x_3\} \} } , \ \alpha_{ \{ \{x_0,x_2\}, \{x_1,x_3\} \} } \text{ and } \alpha_{ \{ \{x_0,x_3\}, \{x_1,x_2\} \} }
\]
is equal to zero. From Theorem \ref{Thm:thm1}, it thus follows that the $d$-splits with respect to any integer-valued pseudometric $d$ on $X$ are weakly compatible. Now, for pseudometric spaces with integer-valued pseudometric, the following holds:

\begin{Thm}\label{Thm:thm3} 
The $d$-splits with respect to any integer-valued pseudometric $d$ on a set $X$ are weakly compatible. Conversely, let $\mathcal{S}_0$ be any collection of weakly compatible splits of $X$. For each $S \in \mathcal{S}_0$, choose some $\lambda_S \in (0,\infty)$ such that
\[
d := \sum_{S \in \mathcal{S}_0} \lambda_S \delta_S \colon X \times X \to \mathbb{Z} \cap [0,\infty).
\]
Then, $\mathcal{S}_0$ is the set of all $d$-splits and for each $S \in \mathcal{S}_0$, the isolation index $\alpha_S= \alpha_S^d$ equals $\lambda_S$.
\end{Thm}
\begin{proof}

Let $S:= \{A,B\} \in \mathcal{S}_0$. Pick $x,y \in A$ and $z,w \in B$ such that 
\[
\alpha^d_{ \{ \{x,y\}, \{z,w\} \} } = \alpha^d_{ \{ A,B \} }.
\]
By weak compatibility of $\mathcal{S}_0$, we can assume that there is no split in $\mathcal{S}_0$ extending (without loss of generality) $\{ \{x,w\}, \{y,z\} \}$. Let 
\begin{align*}
\mathcal{S}_1 &:= \{ S \in \mathcal{S}_0: S \text{ extends } \{ \{x,y\}, \{z,w\} \} \}, \\
\mathcal{S}_2 &:= \{ S \in \mathcal{S}_0: S \text{ extends } \{ \{x,z\}, \{y,w\} \} \},
\end{align*}
noting that $S \subset \mathcal{S}_1$. All splits in $\mathcal{S}_0 \setminus ( \mathcal{S}_1 \cup \mathcal{S}_2)$ equally contribute to each of the three distance sums involving $x,y,z,w$ in $\beta^d_{ \{ \{x,y\}, \{z,w\} \} } $, so that by \eqref{eq:AlphaBeta}, we get
\begin{align*}
\alpha^d_{ \{ \{x,y\}, \{z,w\} \} } 
= \beta^d_{ \{ \{x,y\}, \{z,w\} \} } 
&= \frac{1}{2} \left( \max \{ xz + yw, xw + yz \} - xy - zw \right)\\
&= \max \left \{ \sum_{S \in \mathcal{S}_1} \lambda_S , \sum_{S \in \mathcal{S}_1 \cup \mathcal{S}_2} \lambda_S \right \} - \sum_{S \in \mathcal{S}_1} \lambda_S \\
&= \sum_{S \in \mathcal{S}_1} \lambda_S \\
&\ge \lambda_{\{A,B\}} \\
&> 0.
\end{align*}
Therefore, $\{A,B\}$ is a $d$-split. Let us denote by $\mathcal{S}$ the set of all $d$-splits of $X$, we have just proved that $\mathcal{S}_0 \subset \mathcal{S}$. We can decompose $d$ according to Theorem~\ref{Thm:thm2} to obtain
\[
d = d_0 + \sum_{S \in \mathcal{S}} \alpha^d_S \delta_S
\ge \sum_{S \in \mathcal{S}_0} \alpha^d_S \delta_S
\ge \sum_{S \in \mathcal{S}_0} \lambda_S \delta_S
= d,
\]
which implies that equality holds throughout. This finally yields $d_0 \equiv 0$ as well as $\mathcal{S}_0 = \mathcal{S}$. Furthermore, for each $S \in \mathcal{S}$, one has $\alpha^d_S = \lambda_S$ and for each $S \notin \mathcal{S}$, one has $\alpha^d_S = 0$. This concludes the result.
\end{proof}

In the next Lemma, we denote by $\Sigma^0(\mathrm{E}(X,d))$ the set of vertices of the polyhedral complex $\mathrm{E}(X,d)$. Equivalently, it is the set of all functions $f \in \mathrm{E}(X,d)$ such that $\mathrm{rank}(A(f))=0$.

\begin{Lem}
Let $(X,d)$ be a metric space with integer-valued metric satisfying the (LRC) and let $f \in \Sigma^0(\mathrm{E}(X,d))$. Let $\mathcal{S}_{< \infty}$ be any finite subset of the set $\mathcal{S}$ of all $d$-splits of $X$. If for every $S \in \mathcal{S}_{< \infty}$, one picks $\lambda_S \in (-\infty,\alpha^d_S]$, then there are functions $f_S \in \mathrm{E}(X,\delta_S)$ such that
\[
f - \sum_{S \in \mathcal{S}_{< \infty}} \lambda_S f_S \in \Delta \left(X, d - \sum_{S \in \mathcal{S}_{< \infty}} \lambda_S \delta_S \right).
\]
\end{Lem}
\begin{proof}
Let $S :=\{A,B\} \in \mathcal{S}_{< \infty}$ where $\mathcal{S}_{< \infty}$ is any finite subset $\mathcal{S}$. Since $\mathrm{rank}(A(f))=0$ and thus $A(f)$ is in particular not bipartite, there are $a,a' \in X$ such that $\{a,a'\} \in A(f)$ and either $a,a' \in A$ or $a,a' \in B$. Assume without loss of generality that $a,a' \in A$. Note that if there are $b,b' \in B$ such that one has $\{b,b'\} \in A(f)$, then 
\begin{align*}
\max \{d(a,b)+d(a',b'),d(a,b')+d(a',b)\} 
&\le f(a) + f(a') + f(b) + f(b') \\
&= d(a,a') + d(b,b'),
\end{align*}
and thus $\alpha^d_S =0$, which contradicts our assumption. Hence for any $b,b' \in B$, one has $\{b,b'\} \notin A(f)$. We set
\[
f_S(x) \; = \; \left\lbrace{\begin{matrix} 
0 &\text{if }  x \in A, \\
1 &\text{if } x \in B, \end{matrix}}\right.
\]
We show that $f^{(1)}:= f - \lambda_S f_S$ and $d^{(1)} := d -\lambda_S f_S$ satisfy $f^{(1)} \in \Delta \left(X, d^{(1)} \right)$. We denote distances $d(x,y)$ simply by $xy$ and we distinguish three cases:
\begin{enumerate}[$(a)$]
\item  if $x,y \in A$, then
\[
f^{(1)}(x) + f^{(1)}(y) = f(x) + f(y)\ge xy = d^{(1)}(x,y),
\]
\item if $x \in A$ and $y \in B$, then
\[
f^{(1)}(x) + f^{(1)}(y) = f(x) + f(y) - \lambda_S \ge  xy - \lambda_S  = d^{(1)}(x,y),
\]
\item if $x,y \in B$, there are $ \{a,a'\} \subset A$ such that 
\[
f(a) + f(a') - aa' =0
\]
hence
\begin{align*}
f^{(1)}(x) + f^{(1)}(y) 
&= f(x) + f(y) - 2 \lambda_S \\
&\ge \left[ f(x) + f(y) - 2 \lambda_S \right] + \left[ f(a) + f(a') - aa'  \right]\\
&\ge \max \left\{ xa + ya', xa' + ya\right\} -aa' -2 \lambda_S\\
&\ge xy .
\end{align*}
\end{enumerate}
Hence $f^{(1)} \in \Delta \left(X, d^{(1)} \right)$ as desired. Now, note that if $\{x,y\} \in A^d(f)$, then either $x,y \in A$ or $x \in A$ and $y \in B$. In both cases one sees that equalities hold in $(a)$ and $(b)$ above and thus that $\{x,y\} \in A^{d^{(1)}}(f^{(1)})$, i.e. 
\begin{equation}\label{eq:InclusionAdmissibleGraphs}
A^d(f) \subset A^{d^{(1)}}(f^{(1)})
\end{equation}
hence in particular $f \in \Sigma^0(\mathrm{E}(X,d^{(1)}))$. For $S':=\{A',B'\} \in \mathcal{S} \setminus \{S\}$, we can proceed as above finding $c,c' \in A'$ such that $\{c,c'\} \in A^{d^{(1)}}(f^{(1)})$. By Theorem~\ref{Thm:thm2}, it follows that $\alpha^{d^{(1)}}_{S'} = \alpha^{d}_{S'} > 0$ which implies that for $e,e' \in B'$, one has $\{e,e'\} \notin A^{d^{(1)}}(f^{(1)})$. We set
\[
f^{(1)}_{S'}(x) \; = \; \left\lbrace{\begin{matrix} 
0 &\text{if }  x \in A', \\
1 &\text{if } x \in B', \end{matrix}}\right.
\]
and as above in $(a)-(c)$, we obtain $f^{(1)} - \lambda_{S'} f^{(1)}_{S'} \in \Delta \left(X, d^{(1)} - \lambda_{S'} \delta_{S'} \right)$. By \eqref{eq:InclusionAdmissibleGraphs}, it follows that $f_{S'}=f^{(1)}_{S'}$. Hence, we obtain 
\[
f-\lambda_S f_S -\lambda_{S'} f_{S'} 
= f^{(1)} - \lambda_{S'} f^{(1)}_{S'} 
\in \Delta \left(X, d^{(1)} - \lambda_{S'} \delta_{S'}\right)
= \Delta \left(X, d^{(2)} \right)
\]
where $d^{(2)} = d  - \lambda_S \delta_S - \lambda_{S'} \delta_{S'}$. By induction, we now get the desired result.
\end{proof}

Now,

\begin{Lem}\label{Lem:FirstPartDecompositionDelta}
Let $(X,d)$ be a metric space with integer-valued metric satisfying the (LRC). For any split $S=\{A,B\}$ of $X$ and any $\lambda_S>0$ such that $d = d_1 + \lambda_S \delta_{S}$, if $\Delta(X,d) = \Delta(X,d_1) + \lambda_S \Delta(X,\delta_{S})$, then the following hold:
\begin{enumerate}[$(i)$]
\item $S$ is a $d$-split of $X$ and
\item $\lambda_S \le \alpha_S$.
\end{enumerate}
\end{Lem}
\begin{proof}
Assume that there is a split $d$-split $S=\{A,B\}$ of $X$ and $\lambda_S>0$ such that $d = d_1 + \lambda_S \delta_{S}$, if $\Delta(X,d) = \Delta(X,d_1) + \lambda_S \Delta(X,\delta_{S})$. We show that for any $A_0 := \{a,a'\}$ and $B_0 :=\{b,b'\}$ in $X$ so that $A_0 \cap B_0 = \emptyset$ (and writing simply $xy$ instead of $d(x,y)$), one has
\[
\frac{1}{2} \Bigl( \max\{ab + a'b', a'b + ab'\} -aa' - bb' \Bigr) \ge \lambda.
\]  
Note that by using that for any $Y \subset X$ and any split $S$ of $X$, one has
\[
\Delta(d|_{Y \times Y}) = \Delta(d_1|_{Y \times Y}) + \lambda_S \Delta(\delta_S |_{Y \times Y})
\]
and thus in particular for $Y := A_0 \cup B_0$ and the split $S := \{A_0,B_0\}$ of $Y$. Define now the map $f \colon Y \to \R$ as follows:
\begin{align*}
f(a) &= \frac{1}{2} (aa' + ab -a'b), \\
f(a') &= \frac{1}{2} (aa' + a'b -ab), \\
f(b) &= \frac{1}{2} (ab + a'b -aa'), \\
f(b') &= \max \{ ab' - f(a), a'b' - f(a'),bb' - f(b) \}.
\end{align*}
It is then easy to see that
\begin{align*}
f(a) +f(a') &= aa', \\
f(a) +f(b) &= ab, \\
f(a') +f(b) &= a'b.
\end{align*}
Furthermore, $f(a)$, $f(a')$ and $f(b)$ are clearly non-negative. From
\[
ab' - f(a) + a'b' - f(a') = ab' + a'b' -aa' \ge 0
\]
we deduce that $f(b') \ge 0$. Therefore, $f \in \mathrm{E}(Y,d) \subset \Delta(Y,d)$. By assumption, there exist $f_1 \in \Delta(Y,d_1)$ and $f_S \in \Delta(Y,\delta_S)$ such that
\[
f=f_1 + \lambda_S f_S.
\] 
Since $\delta_S(a,a')=0$, we have
\[
d_1(a,a') = aa' = f(a) + f(a') \ge f_1(a) + f_1(a')
\]
hence $f_1(a) + f_1(a') = aa'$ and thus $f_S(a) = f_S(a') = 0$ which implies 
\[
f_S(b) \ge 1 \ \ \text{ and } \ \ f_S(b') \ge 1.
\]
Moreover, we have
\begin{align*}
f_1(a) + f_1(b) + \lambda_S f_S(b) = f(a) + f(b) = ab &= d_1(a,b) + \lambda_S \\
& \le  f_1(a) + f_1(b) + \lambda_S.
\end{align*}
Therefore 
\[
f_S(b) =1.
\]
Similarly, 
\begin{align*}
f_1(a') + f_1(b) + \lambda_S f_S(b) = f(a') + f(b) = a'b &= d_1(a',b) + \lambda_S \\
& \le  f_1(a') + f_1(b) + \lambda_S,
\end{align*}
and thus
\[
f(a') + f(b) = f_1(a') + f_1(b) + \lambda_S = a'b.
\]
Observe that since $\lambda_S f_S(b) = \lambda_S >0$, one has 
\[
f(b) + f(b') > f_1(b) + f_1(b') \ge d_1(b,b')= bb'.
\]
Since we may interchange the role of $a$ and $a'$ in the following, we can assume that $f(a) + f(b') = ab'$ since $f \in \mathrm{E}(Y,d)$. Hence
\begin{align*}
f_1(a) + f_1(b') + \lambda_S f_S(b) = f(a) + f(b') = ab' &= d_1(a,b') + \lambda_S \\
& \le  f_1(a) + f_1(b') + \lambda_S.
\end{align*}
and therefore
\[
f_S(b') =1.
\]
Since 
\[
f_1(b) + f_1(b') \ge d_1(b,b') = bb',
\]
we finally obtain
\begin{align*}
&\frac{1}{2} \Bigl( \max\{ab + a'b', a'b + ab'\} -aa' - bb' \Bigr) \\
&\ge \frac{1}{2} \Bigl( a'b + ab' -aa' - bb' \Bigr) \\
&= \frac{1}{2} \Bigl( f(a') + f(b) + f(a) + f(b')- f(a) - f(a') -bb'\Bigr) \\
&= \frac{1}{2} \Bigl( f_1(a') + f_1(b) + \lambda + f_1(a) + f_1(b') + \lambda -f_1(a) - f_1(a') -bb'\Bigr) \\
&= \frac{1}{2} \Bigl( f_1(b) + f_1(b') -bb' \Bigr) + \lambda \\
&\ge \lambda,
\end{align*}
i.e. $\alpha^d_S \ge \lambda$ and this is the desired result.
\end{proof}

Note that 
\[
\Delta(X,d_0) + \sum_{S \in \mathcal{S}} \alpha_S \Delta(X,\delta_S) \subset [0,\infty]^X,
\]
and since $\mathcal{S}$ is possibly infinite, we cannot replace $[0,\infty]^X$ by $\R^X$ in general.

\begin{Thm}\label{Thm:thm4}
Let $(X,d)$ be a metric space with integer-valued metric satisfying the (LRC). Let $\mathcal{S}$ be the family of all $d$-splits of $X$ and let $\lambda_S \in (0,\alpha_S]$ for every $S \in \mathcal{S}$ such that $d = d_0 + \sum_{S \in \mathcal{S}} \alpha_S \delta_S$ and $d_1 := d - \sum_{S \in \mathcal{S}} \lambda_S \delta_S$.
Then
\[
\Delta(X,d) = \R^X \cap \left( \Delta(X,d_1) + \sum_{S \in \mathcal{S}} \lambda_S \Delta(X,\delta_S) \right).
\]
\end{Thm}
\begin{proof}
Let $\mathcal{S}$ be the set of all $d$-splits and let $S \in \mathcal{S}$. For $f \in \Sigma^0(\mathrm{E}(X,d))$, since $\mathrm{rank}(A(f))=0$ and thus $A(f)$ is in particular not bipartite, there are $a,a' \in X$ such that $\{a,a'\} \in A(f)$ and either $a,a' \in A$ or $a,a' \in B$. Assume without loss of generality that $a,a' \in A$. Then, for any $b,b' \in B$, one has $\{b,b'\} \notin A(f)$ as we argued in the proof of Lemma~\ref{Lem:DecompositionInduction}. We set
\[
f_S(x) \; = \; \left\lbrace{\begin{matrix} 
0 &\text{if }  x \in A, \\
1 &\text{if } x \in B. \end{matrix}}\right.
\]
We first show that for every $x \in X$, one has $\sum_{S \in \mathcal{S}} \lambda_Sf_S(x) \neq \infty$. Note that for every $x \in X$, there exists $x' \in X$ such that $\{x,x'\} \in A^d(f)$. Furthermore, we have
\[
d(x,x') = d_0(x,x') + \sum_{\substack{S \in \mathcal{S} \\ S(x) \neq S(x')}} \alpha^d_S \delta_S(x,x').
\]
Since $\alpha^d_S \in [0,\infty) \cap \frac{1}{2}\mathbb{Z}$, we deduce that the set
\[
\mathcal{S}_{xx'} := \{S \in \mathcal{S} : S(x) \neq S(x')\}
\]
is finite. Moreover, for every $S := \{A,B\} \in \mathcal{S} \setminus \mathcal{S}_{xx'}$, since $\{x,x'\} \in A^d(f)$, one has $x,x' \in A$ and $f_S(x) = 0 = f_S(x')$ by definition. It follows that 
$\sum_{S \in \mathcal{S}} \lambda_Sf_S \in \R^X$ as well as
\[
f_1 := f - \sum_{S \in \mathcal{S}} \lambda_Sf_S \in \R^X.
\]
For $d_1 := d - \sum_{S \in \mathcal{S}} \lambda_S \delta_S$, it now remains to show that $f_1 \in \Delta(X,d_1)$. For every $x,y \in X$, there are $x',y' \in X$ such that $\{x,x'\},\{y,y'\} \in A^d(f)$. Since for every $S \in \mathcal{S}_{xy}$, one has $f_S(x) + f_S(y)=1$, it follows that $ \mathcal{S}_{xy} \subset \mathcal{S}_{xx'} \cup \mathcal{S}_{yy'} =: \mathcal{S}_{<\infty}$ where $|\mathcal{S}_{<\infty}| < \infty$. By  Lemma~\ref{Lem:DecompositionInduction} and setting $d_1^{<\infty} := d - \sum_{S \in \mathcal{S}_{< \infty}} \lambda_S \delta_S$, it follows that 
\[
f_1^{<\infty} := f - \sum_{S \in \mathcal{S}_{< \infty}} \lambda_S f_S \in \Delta \left(X, d_1^{<\infty}\right)
\]
and thus
\begin{align*}
f_1(x)+f_1(y) 
= f_1^{<\infty}(x) + f_1^{<\infty}(y)
&\ge d_1^{<\infty}(x,y) \\
&\ge d(x,y)- \sum_{S \in \mathcal{S}_{xy}} \lambda_S \delta_S(x,y) \\
&= d_1(x,y).
\end{align*}
This shows that 
\[
\Sigma^0(\mathrm{E}(X,d)) \subset \left( \Delta(X,d_1) + \sum_{S \in \mathcal{S}} \lambda_S \Delta(X,\delta_S) \right).
\]
Now, using that $(X,d)$ satisfies the (LRC), we can take for each cell of $\mathrm{E}(X,d)$, all finite convex combinations of its vertices and by convexity of $\Delta(X,d_1)$ and $\Delta(X,\delta_S)$ for every $S \in \mathcal{S}$, we deduce that
\[
\mathrm{E}(X,d) \subset \left( \Delta(X,d_1) + \sum_{S \in \mathcal{S}} \lambda_S \Delta(X,\delta_S) \right).
\]
Adding finally $[0,\infty)^X$ on both sides and intersecting with $\mathbb{R}^X$, we get
\[
\Delta(X,d) \subset \mathbb{R}^X \cap \left( \Delta(X,d_1) + \sum_{S \in \mathcal{S}} \lambda_S \Delta(X,\delta_S) \right).
\]
Since the other inclusion is easy to see, we obtain the desired result.
\end{proof}

Let $f \in \mathrm{E}(X,d) \subset \Delta(X,d)$. From 
\begin{equation}\label{eq:DeltaDecomposition}
\Delta(X,d) = \R^X \cap \left( \Delta(X,d_1) + \sum_{S \in \mathcal{S}} \lambda_S \Delta(X,\delta_S) \right),
\end{equation}
we have a decomposition $f = f_1 + \sum_{S \in \mathcal{S}} \lambda_S f_S$. Note that if there are for $S \in \mathcal{S}$, functions $ f_S \ge g_S \in \Delta(X,\delta_S)$ and $f_1 \ge g_1 \in \Delta(X,d_1)$ where not all inequalities are equalities, then $g := g_1 + \sum_{S \in \mathcal{S}} \lambda_S g_S \in \Delta(X,d)$ by \eqref{eq:DeltaDecomposition}. Since $g \le f \in \mathrm{E}(X,d)$, this contradicts the minimality of $f$. We must therefore have that 
\[
f \in \R^X \cap \left( \mathrm{E}(X,d_1) + \sum_{S \in \mathcal{S}} \lambda_S \mathrm{E}(X,\delta_S) \right).
\]
We have thus shown that \eqref{eq:DeltaDecomposition} implies
\[
\mathrm{E}(X,d) \subset \R^X \cap \left( \mathrm{E}(X,d_1) + \sum_{S \in \mathcal{S}} \lambda_S \mathrm{E}(X,\delta_S) \right).
\]

Hence, we obtain:

\begin{Cor}\label{Cor:CorEI}
Let $(X,d)$ be a metric space with integer-valued metric satisfying the (LRC). Let $\mathcal{S}$ be the family of all $d$-splits of $X$ such that $d = d_0 + \sum_{S \in \mathcal{S}} \alpha_S \delta_S$ and let $\lambda_S \in (0,\alpha_S]$ for every $S \in \mathcal{S}$ so that $d_1 := d - \sum_{S \in \mathcal{S}} \lambda_S \delta_S$. Then
\begin{equation}\label{eq:InjectiveHullDecomposition}
\mathrm{E}(X,d) \subset \R^X \cap \left( \mathrm{E}(X,d_1) + \sum_{S \in \mathcal{S}} \lambda_S \mathrm{E}(X,\delta_S) \right).
\end{equation}
\end{Cor}

It is also easy to see that \eqref{eq:InjectiveHullDecomposition} implies \eqref{eq:DeltaDecomposition}. Remember that from \cite{Lan}, there is a $1$-Lipschitz map $p \colon \Delta(X,d) \to \mathrm{E}(X,d)$ such that for every $f \in \Delta(X,d)$, one has $p(f) \le f$. Then $f \ge p(f) =:g \in \mathrm{E}(X,d)$. From \eqref{eq:InjectiveHullDecomposition}, we obtain a decomposition of $g$ as $g := g_1 + \sum_{S \in \mathcal{S}} \lambda_S g_S$. Moreover, $f-g \in [0,\infty)^X$, hence $g_1 + (f-g) \in \Delta(X,d_1)$ and thus
\[
f = g_1 + (f-g) + \sum_{S \in \mathcal{S}} \lambda_S g_S \in \mathbb{R}^X \cap \left( \Delta(X,d_1) + \sum_{S \in \mathcal{S}} \lambda_S \Delta(X,\delta_S) \right)
\]
which is the desired result.


\section{The Buneman Complex and Related Topics}

If $\mathcal{S}$ is a split system (on a set $X$) and $\alpha \colon \mathcal{S} \to (0,\infty)$ is any map $S \mapsto \alpha_S$, the pair $(\mathcal{S},\alpha)$ is called a \textit{split system pair (of $X$)}. If $\mathcal{S}$ is weakly compatible in addition, then $(\mathcal{S},\alpha)$ is called a weakly compatible \textit{split system pair}. 

Let now $\mathcal{S}$ be a weakly compatible split system on a pseudometric space with integer-valued pseudometric $(X,d)$ and assume that 
\[
d = \sum_{S \in \mathcal{S}} \alpha_S \delta_S.
\]
By Theorem~\ref{Thm:thm3}, $\mathcal{S}$ is the set of all $d$-splits of $X$ and $d$ is thus totally split-decomposable. Let $\alpha \colon \mathcal{S} \to (0,\infty)$ denote the map given by $S \mapsto \alpha_S$. The weakly compatible split system pair $(\mathcal{S},\alpha)$ is called \textit{the split system pair associated to $(X,d)$}. Unless otherwise stated, this is the split system pair that we refer to in the sequel, when considering a totally split-decomposable pseudometric space $(X,d)$. 

Let now $(\mathcal{S},\alpha)$ be any split system pair on a set $X$,
\[
U(\mathcal{S}) := \{ A \subset X : \text{there is } S \in \mathcal{S} \text{ such that } A \in S \}
\]
and for $\mu \colon U(\mathcal{S}) \to [0,\infty)$, let
\[
\mathrm{supp}(\mu) = \{ A \in U(\mathcal{S}) : \mu(A) > 0 \}.
\]
For $A \subset X$, we shall denote its complement $X \setminus A$ by $\bar{A}$. If $A \in U(\mathcal{S})$, we shall denote the split $\{A,\bar{A}\} \in \mathcal{S}$ by $S_A$.
Moreover, we set
\begin{align*}
P(\mathcal{S},\alpha) :=  \{ \mu \colon U(\mathcal{S}) \to [0,\infty) :\ &\text{for all } A \in U(\mathcal{S}), \text{ one has } \\
 &\mu(A) \ge 0 \text{ and } \mu(A) + \mu(\bar{A}) \ge \alpha_{S_A}/2  \}
\end{align*}
and we define an infinite dimensional hypercube 
\begin{align*}
H(\mathcal{S},\alpha) :=  \{ \mu \colon U(\mathcal{S}) \to [0,\infty) :\ &\text{for all } A \in U(\mathcal{S}), \text{ one has } \\
 &\mu(A) \ge 0 \text{ and } \mu(A) + \mu(\bar{A}) = \alpha_{S_A}/2  \}.
\end{align*}
Note that $H(\mathcal{S},\alpha)$ has a natural cell complex structure, cells are sets of the form 
\[
[\mu] := \{ \mu' \in H(\mathcal{S},\alpha): \mathrm{supp}(\mu') \subset \mathrm{supp}(\mu)\}
\]
where $\mu \in H(\mathcal{S},\alpha)$. Furthermore, the cells of $H(\mathcal{S},\alpha)$ are (possibly) infinite dimensional hypercubes. The \textit{Buneman complex} is the subcomplex of $H(\mathcal{S},\alpha)$ given by 
\begin{align*}
&B(\mathcal{S},\alpha) := \\
&\left \{ \mu \in H(\mathcal{S},\alpha) : \text{if } A,B \in \mathrm{supp}(\mu) \text{ and } A \cup B = X, \text{ then } A \cap B = \emptyset \right \}.
\end{align*}
Furthermore, let
\begin{align*}
&\bar{T}(\mathcal{S},\alpha) := \\
&\left \{ \mu \in H(\mathcal{S},\alpha) : \text{if } \{A_i \}_{i \in I} \subset \mathrm{supp}(\mu) \text{ and } \bigcup_{i \in I} A_i = X, \text{ then } \bigcap_{i \in I} A_i = \emptyset  \right \}.
\end{align*}
It is easy to see that $\bar{T}(\mathcal{S},\alpha)$ is a subcomplex of $B(\mathcal{S},\alpha)$, cf. \cite[Section~4]{DreHM}. Indeed, $\bar{T}(\mathcal{S},\alpha) \subset B(\mathcal{S},\alpha)$ and $\bar{T}(\mathcal{S},\alpha)$ is a subcomplex of $H(\mathcal{S},\alpha)$ since if $\mu \in \bar{T}(\mathcal{S},\alpha)$, then for any $\mu' \in H(\mathcal{S},\alpha)$ such that $\mathrm{supp}(\mu') \subset \mathrm{supp}(\mu)$, one obviously has $\mu' \in \bar{T}(\mathcal{S},\alpha)$. We can thus denote by $[\psi]$ the smallest cell of $H(\mathcal{S},\alpha)$ containing $\psi \in H(\mathcal{S},\alpha)$ and thus if $\psi \in B(\mathcal{S},\alpha)$, then $[\psi] \subset B(\mathcal{S},\alpha)$ and similarly if $\psi \in \bar{T}(\mathcal{S},\alpha)$, then $[\psi] \subset \bar{T}(\mathcal{S},\alpha)$.

Now, for $x \in X$ consider the map $\phi_x \colon U(\mathcal{S}) \to [0,\infty)$ which is defined as
\[
\phi_x(A) \; = \; \left\lbrace{\begin{matrix} 
\frac{\alpha_{S_A}}{2}  &\text{if }  x \notin A  , \\
0  &\text{if } x \in A.\end{matrix}}\right.
\]
Moreover, let $d_1 \colon \R^{U(\mathcal{S})} \times \R^{U(\mathcal{S})} \to [0,\infty]^X$ be given by
\[
(\mu,\psi) \mapsto  \sum_{A \in U(\mathcal{S})} | \mu(A) - \psi(A)|.
\]
The map $\kappa \colon \R^{U(\mathcal{S})} \to [0,\infty]^{X}$ given by $\mu \mapsto \kappa(\mu)$ where for $x \in X$, one has:
\[
\kappa(\mu)(x) = d_1(\mu,\phi_x).
\]

Let $(\mathcal{S},\alpha)$ be a split system pair on a set $X$ and assume $d:=\sum_{S \in \mathcal{S}} \alpha_S \delta_S$ defines a pseudometric on $X$. It follows easily that the following hold:
\begin{enumerate}[$(i)$] 
\item For every $x,y \in X$, one has $d_1(\phi_x,\phi_y) = d(x,y)$.
\item For every $x \in X$, one has $\phi_x \in \bar{T}(\mathcal{S},\alpha) \subset B(\mathcal{S},\alpha)$.
\item $\kappa$ is $1$-Lipschitz.
\end{enumerate}
For $x \in X$ and $S = \{A,\bar{A} \} \in \mathcal{S}$, let $S(x) := A$ if $x \in A$ and $S(x) := \bar{A}$ if $x \in \bar{A}$. For a further $y \in X$ and $\psi \colon \mathcal{S} \to \R$, one has:
\begin{align*}
\kappa(\psi)(x) + \kappa(\psi)(y) &= 
\begin{aligned}[t]
&\sum_{\substack{S \in \mathcal{S} \\ S(x) \neq S(y)}} 
\begin{aligned}[t]
&\biggl[ \bigl| \psi(S(x)) - \phi_x(S(x)) \bigr| +  \bigl| \psi(S(x)) - \phi_y(S(x)) \bigr|  \\ 
&+  \bigl| \psi(\overline{S(x)}) - \phi_x(\overline{S(x)}) \bigr| +   \bigl| \psi(\overline{S(x)}) - 
\phi_y(\overline{S(x)}) \bigr| \biggr]
\end{aligned} \\
&+ \sum_{\substack{S \in \mathcal{S} \\ S(x) = S(y)}}
\begin{aligned}[t]
&\biggl[  \bigl| \psi(S(x)) - \phi_x(S(x)) \bigr| +  \bigl| \psi(S(x)) - \phi_y(S(x)) \bigr|  \\ 
&+  \bigl| \psi(\overline{S(x)}) - \phi_x(\overline{S(x)}) \bigr | +   \bigl| \psi(\overline{S(x)}) - 
\phi_y(\overline{S(x)})\bigr | \biggr]
\end{aligned} \\
\end{aligned}\\
&\ge 
\begin{aligned}[t]
&\sum_{\substack{S \in \mathcal{S} \\ S(x) \neq S(y)}} \alpha_S \delta_S(x,y) \\
&+ \sum_{\substack{S \in \mathcal{S} \\ S(x) = S(y)}} \biggl[ 2  \bigl| \psi(S(x)) \bigr| + 2 \bigl| \psi(\overline{S(x)}) - \alpha_S /2 \bigr |  \biggr],
\end{aligned}
\end{align*}
Hence, it follows that
\begin{equation}\label{eq:e1}
\frac{1}{2}[\kappa(\psi)(x) + \kappa(\psi)(y) - d(x,y)] \ge \sum_{\substack{S \in \mathcal{S} \\ S(x) = S(y)}} \biggl[  \bigl| \psi(S(x)) \bigr| +  \bigl| \psi(\overline{S(x)}) - \alpha_S /2 \bigr|  \biggr],
\end{equation}
and equality holds if $\psi \in H(\mathcal{S},\alpha)$.

For simplicity, we denote the injective hull $\mathrm{E}(X,d)$ by $\E d$ when the underlying space $X$ is clear (unlike in \cite{Lan} where it is denoted by $\E X$) and $\mathrm{E}'(X,d)$ by $\mathrm{E}'(d)$. Analoguously, $\mathrm{\Delta}(X,d)$ is denoted by $\Delta(d)$ (it corresponds to the space $\Delta(X)$ in \cite{Lan}).

\begin{Lem}\label{Lem:l11}
Let $(\mathcal{S},\alpha)$ be a split system pair on a set $X$ and assume that $d:=\sum_{S \in \mathcal{S}} \alpha_S \delta_S$ defines a pseudometric on $X$. For every $\mu \in H(\mathcal{S},\alpha)$, the following are equivalent:
\begin{enumerate}[$(i)$]
\item $\kappa(\mu) \in \mathrm{E}'(d)$.
\item $\mu \in \bar{T}(\mathcal{S},\alpha)$.
\end{enumerate}
\end{Lem}
\begin{proof}
Consider $\mu \in H(\mathcal{S},\alpha)$. Let us first show that $(ii)$ implies $(i)$. Assume that $\mu \in \bar{T}(\mathcal{S},\alpha)$ and suppose by contradiction that $\kappa(\mu) \notin \mathrm{E}'(d)$. For $x, y \in X$ we have
\[
\kappa(\mu)(x) + \kappa(\mu)(y) = d_1(\mu,\phi_x) + d_1(\mu,\phi_y) \ge d_1(\phi_x,\phi_y) = d(x,y)
\]
which proves that $\kappa(\mu) \in \Delta(d)$. By our contradiction assumption, there is $x \in X$ such that for every $y \in X$, one has:
\[
\kappa(\mu)(x) + \kappa(\mu)(y) > d(x,y)
\]
and thus by \eqref{eq:e1}, there is $S_y \in \mathcal{S}$ such that
\begin{enumerate}[$(a)$]
\item $S_y(x) = S_y(y)$ and
\item $\mu(S_y(x)) > 0$ (noting that this follows from the fact that $\mu \in H(\mathcal{S},\alpha)$).
\end{enumerate}
Thus, we have $X = \bigcup_{y \in X} S_y(x)$ where for every $y$, one has $\mu(S_y(x)) > 0$. Moreover, $\bigcap_{y \in X} S_y(x)$ is non-empty since it contains $x$. It follows that $\mu$ does not satisfy $(ii)$, which is a contradiction.

Now, we show that $(i)$ implies $(ii)$. Assume that $\kappa(\mu) \in \mathrm{E}'(d)$. For every $x\in X$, there is $y \in X$ such that $\kappa(\mu)(x) + \kappa(\mu)(y) = d(x,y)$. Hence, by \eqref{eq:e1}, it follows that for any $S \in \mathcal{S}$:
\begin{equation}\label{eq:e2}
\text{if } S(x) = S(y), \text{ then } \mu (S(x)) = 0 = \mu(S(y)).
\end{equation}
Note that for any $A_i \in \mathrm{supp}(\mu)$, there exists $S_i \in \mathcal{S}$ and $x_i \in X$ such that
\[
A_i = S_i(x_i).
\]
Now, if $X = \bigcup_{i \in I} A_i = \bigcup_{i \in I} S_i(x_i)$ and if by contradiction $\bigcap_{i \in I} S_i(x_i) $ is assumed to be non-empty, so that we can an arbitrary pick $z \in  \bigcap_{i \in I} S_i(x_i) $, then for every $i \in I$ one has:
\[
S_i(x_i) = S_i(z)
\]
and thus
\[
X = \bigcup_{i \in I} S_i(z).
\]
It is now easy to see that the existence of $z$ contradicts \eqref{eq:e2}. Indeed, for any $y \in X$, there is $S_j \in \{ S_i \}_{i \in I} \subset \mathcal{S}$ such that $y \in S_j(z)$ and hence 
\begin{equation}\label{eq:e222}
S_j(y) = S_j(z).
\end{equation}
However,
\begin{equation}\label{eq:e2222}
S_j(z) = S_j(x_j) = A_j \in \mathrm{supp}(\mu)
\end{equation}
and $y$ can be chosen so that $\kappa(\mu)(z) + \kappa(\mu)(y) = d(z,y)$. Thus by \eqref{eq:e222} and \eqref{eq:e2}, one has $\mu (S_j(y)) = 0 = \mu(S_j(z))$, which is a contradiction to \eqref{eq:e2222}. This finishes the proof.
\end{proof}

\begin{Lem}\label{Lem:l12}
Let $(X,d)$ be a totally split-decomposable metric space with integer-valued metric satisfying the (LRC). Then, the map $\bar{\kappa} := \kappa|_{\bar{T}(\mathcal{S},\alpha)} \colon \bar{T}(\mathcal{S},\alpha) \to \mathrm{E}(d)$ is surjective.
\end{Lem}
\begin{proof}
Let $f \in \mathrm{E}(d)$. By Corollary \ref{Cor:CorEI}, we have $\mathrm{E}(d) \subset \sum_{S \in \mathcal{S}} \alpha_S \mathrm{E}(\delta_S)$. Thus, we have a decomposition
\[
f = \sum_{S \in \mathcal{S}} \alpha_S f_S
\]
where $f_S \in \mathrm{E}(\delta_S)$ implies that if $S=\{A,\bar{A}\}$, then for any $x \in A$ and $y \in \bar{A}$ one has that $f_S$ is constantly equal to $f_S(x)$ on $A$, constantly equal to $f_S(y)$ on $\bar{A}$ and $f_S(x),f_S(y) \ge 0$ as well as $f_S(x) + f_S(y) = 1$. Define the map $\psi_f \colon U(\mathcal{S}) \to [0,\infty)$ by setting for every $S=\{A,\bar{A}\}$ as well as for arbitrarily chosen $x \in A$ and $y \in \bar{A}$:
\[
\psi_f(A) := \frac{\alpha_S}{2} f_S(x) \ \ \text{ and } \ \ \psi_f(\bar{A}) := \frac{\alpha_S}{2} f_S(y).
\]
It is clear that $\psi_f$ is well-defined, i.e., the above definition does not depend on the particular choice of $x$ and $y$. Furthermore, it is easy to see that
\begin{enumerate}[$(a)$]
\item $\psi_f \in H(\mathcal{S},\alpha)$ and
\item $\kappa(\psi_f) = f \in \mathrm{E}(d) = \mathrm{E}'(d)$.
\end{enumerate}
It then follows from Lemma \ref{Lem:l11} that $\psi_f \in \bar{T}(\mathcal{S},\alpha)$ and this finishes the proof.
\end{proof}
For a map $\phi \colon \mathcal{S} \to [0,\infty)$, let
\[
\mathcal{S}(\phi) := \{ S \in \mathcal{S} : S \subset \mathrm{supp}(\phi) \}.
\] 
Let us define for a cell $[\phi]$ of $\bar{T}(\mathcal{S},\alpha)$ and $x \in X$, the map $\gamma^x_{[\phi]} \colon U(\mathcal{S}) \to [0,\infty)$ given by
\[
\gamma^x_{[\phi]}(A) \; = \; \left\lbrace{\begin{matrix} 
\phi_x(A)  &\text{if }  A \in U( \mathcal{S}(\phi))   , \\
\phi(A)  &\text{if } A \in U(\mathcal{S} \setminus \mathcal{S}(\phi))  .\end{matrix}}\right.
\]
Note that one has $\psi \in [\phi]$ if and only if $\mathrm{supp}(\psi) \subset \mathrm{supp}(\phi)$ and it follows that if $A \in U(\mathcal{S} \setminus \mathcal{S}(\phi))$, then
\[
\psi(A) = \phi(A) = \gamma^x_{[\phi]}(A).
\]
Hence $\gamma^x_{[\phi]} \in [\phi]$ since $\mathrm{supp}(\gamma^x_{[\phi]}) \subset \mathrm{supp}(\phi)$.

\begin{Def}\label{Def:d1}
Let $i \colon (X,d_X) \to (Y,d_Y)$ be an isometric map of pseudometric spaces. We say that $Z \subset Y$ is \textit{$X$-gated} (for $i$ and with respect to $d_Y$) if and only if for every $x \in X$, there is $y_x \in Z$ such that for every $z \in Z$, one has
\[
d_Y(i(x),z) = d_Y(i(x),y_x) + d_Y(y_x,z).
\]
\end{Def}

\begin{Lem}\label{Lem:l13}
Let $(\mathcal{S},\alpha)$ be a split system pair on a set $X$ and assume that $d:=\sum_{S \in \mathcal{S}} \alpha_S \delta_S$ defines a pseudometric on $X$. Then, every cell $[\phi]$ of $\bar{T}(\mathcal{S},\alpha)$ is $X$-gated with respect to the restriction of $d_1$ to $\bar{T}(\mathcal{S},\alpha)$.
\end{Lem}
\begin{proof}
We already noted that $\phi_x \in \bar{T}(\mathcal{S},\alpha)$ for every $x \in X$ and that the map $x \mapsto \phi_x$ is an isometric embedding of $X$ into $\bar{T}(\mathcal{S},\alpha)$. By Lemma \ref{Lem:l11}, if $\psi \in \bar{T}(\mathcal{S},\alpha)$, then for any $x \in X$, there is $y \in X$ such that $\kappa(\psi)(x) + \kappa(\psi)(y) = d(x,y)$ and thus $d_1(\phi_x,\psi) = \kappa(\psi)(x) < \infty$. In particular, the restriction of $d_1$ to $\bar{T}(\mathcal{S},\alpha)$ is a metric. Now, for $x \in X$ and $\psi \in [\phi]$, one has:
\begin{align*}
d_1(\phi_x,\psi) 
&= \sum_{A \in U(\mathcal{S})} | \phi_x(A) - \psi(A)| \\
&= \sum_{A \in U(\mathcal{S} \setminus \mathcal{S}(\phi))} | \phi_x(A) - \psi(A)| + \sum_{A \in U( \mathcal{S}(\phi))} | \phi_x(A) - \psi(A)| \\
&= \sum_{A \in U(\mathcal{S} \setminus \mathcal{S}(\phi))} | \phi_x(A) - \gamma^x_{[\phi]}(A)| + \sum_{A \in U( \mathcal{S}(\phi))} | \gamma^x_{[\phi]}(A) - \psi(A)| \\
&=d_1(\phi_x,\gamma^x_{[\phi]}) + d_1(\gamma^x_{[\phi]},\psi),
\end{align*}
which shows that $\gamma^x_{[\phi]}$ is a gate for $\phi_x$ in $[\phi]$ with respect to the metric $d_1$.
\end{proof}

\begin{Lem}\label{Lem:l14}
Let $(\mathcal{S},\alpha)$ be a split system pair on a set $X$ and assume that $d:=\sum_{S \in \mathcal{S}} \alpha_S \delta_S$ defines a pseudometric on $X$. Then, for every $\phi \in \bar{T}(\mathcal{S},\alpha)$, the split system $\mathcal{S}(\phi) \subset \mathcal{S}$ is antipodal, which means that for any $x \in X$, there is $y \in X$ such that: 
\begin{equation}\label{eq:e1333}
\text{ for every } S \in \mathcal{S}(\phi), \text{ one has } S(x) \neq S(y).
\end{equation}
For $x,y \in X$, if $d(x,y) = \kappa(\phi)(x) + \kappa(\phi)(y)$, then $x$ and $y$ satisfy \eqref{eq:e1333}.
\end{Lem}
\begin{proof}
By Lemma \ref{Lem:l11}, $\kappa(\phi) \in \mathrm{E}'(d)$. Thus for any $x \in X$, there is $y \in X$ such that 
\[
d(x,y) = \kappa(\phi)(x) + \kappa(\phi)(y)
\]
which can be rewritten as
\begin{equation}\label{eq:33}
\sum_{S \in \mathcal{S}} \alpha_S \delta_S(x,y) = d_1(\phi,\phi_x) + d_1(\phi,\phi_y).
\end{equation}
It is easy to see that for every $S \in \mathcal{S}$, one has
\[
\alpha_S \delta_S(x,y) = \sum_{A \in S} | \phi_x(A) -  \phi_y(A)| \le \sum_{A \in S} \bigl[ | \phi_x(A) -  \phi(A)| +  | \phi(A) -  \phi_y(A)| \bigr]
\]
which together with \eqref{eq:33} imply
\begin{equation}\label{eq:132}
\alpha_S \delta_S(x,y) = \sum_{A \in S} \bigl[ | \phi_x(A) -  \phi(A)| +  | \phi(A) -  \phi_y(A)| \bigr].
\end{equation}
Assume now that there is $S \in \mathcal{S}(\phi)$ such that $S(x) = S(y)$, then by \eqref{eq:132}, one has
\[
0 = \alpha_S \delta_S(x,y) = \sum_{A \in S} \bigl[ | \phi_x(A) -  \phi(A)| +  | \phi(A) -  \phi_y(A)| \bigr] = 4 \phi(S(x)),
\]
which implies $S(x) \notin \mathrm{supp}(\phi)$ and thus $S \notin \mathcal{S}(\phi)$, which is a contradiction. This finishes the proof.
\end{proof}

By Lemma~\ref{Lem:l13} every cell $[\phi]$ of $\bar{T}(\mathcal{S},\alpha)$ is $X$-gated. Let $(\Gamma([\phi]),d_1)$ denote the sets of all $X$-gates of $[\phi]$ endowed with the restriction of $d_1$. A metric space $(X,d)$ is called \textit{antipodal} if there exists an involution $\sigma \colon X \to X$ such that for every $x,y \in X$, one has
\[
d(x,\sigma(x) ) = d(x,y ) + d(y,\sigma(x) ).
\]

\begin{Lem}\label{Lem:l15}
Let $(\mathcal{S},\alpha)$ be a split system pair on a set $X$ and assume that $d:=\sum_{S \in \mathcal{S}} \alpha_S \delta_S$ defines a pseudometric on $X$. Then, for every cell $[\phi]$ of $\bar{T}(\mathcal{S},\alpha)$, the metric space $(\Gamma([\phi]),d_1)$ is antipodal.
\end{Lem}
\begin{proof}
Let $x \in X$. Lemma \ref{Lem:l14} implies that there is $y \in X$ such that for every $S \in \mathcal{S}(\phi)$, one has $S(x) \neq S(y)$. Now, define the map $\sigma \colon \Gamma([\phi]) \to \Gamma([\phi])$ by picking for each $\gamma^u_{[\phi]} \in \Gamma([\phi])$ an arbitrary element 
\[
v \in \bigcap_{S \in \mathcal{S}(\phi)} \overline{S(u)}
\]
and letting
\[
\gamma^u_{[\phi]} \mapsto \gamma^v_{[\phi]}.
\]
First note that if $\gamma^u_{[\phi]} = \gamma^{u'}_{[\phi]}$, one then has for each $A \in U(\mathcal{S}(\phi))$:
\[
\phi_u(A) = \gamma^u_{[\phi]}(A) = \gamma^{u'}_{[\phi]}(A) = \phi_{u'}(A)
\]
that is
\[
\bigcap_{S \in \mathcal{S}(\phi)} \overline{S(u)} = \bigcap_{S \in \mathcal{S}(\phi)} \overline{S(u')}.
\]
Note moreover that $\sigma$ is well-defined since if $v \neq v'$  are such that $S(u) \neq S(v)$ and $ S(u) \neq S(v')$ for all $S \in \mathcal{S}(\phi)$, then $ S(v) = S(v')$ for all $S \in \mathcal{S}(\phi)$ and thus $\gamma^v_{[\phi]} = \gamma^{v'}_{[\phi]}$ by definition of $\gamma^v_{[\phi]}$ and $\gamma^{v'}_{[\phi]}$. Hence $\sigma$ is a well-defined map. It is easy to see that $\sigma$ is also an involution since
\[
v \in \bigcap_{S \in \mathcal{S}(\phi)} \overline{S(u)}
\]
implies that for every $S \in \mathcal{S}(\phi)$, one has $S(v) = \overline{S(u)}$ that is $\overline{S(v)} = S(u)$ and thus
\[
u \in \bigcap_{S \in \mathcal{S}(\phi)} \overline{S(v)}.
\] 
Furthermore, it is easy to deduce from $S(x) \neq S(y)$ for all $S \in \mathcal{S}(\phi)$ that for $z \in X$ and $A \in U(\mathcal{S}(\phi))$, one has
\begin{equation}\label{eq:14}
| \phi_x(A) - \phi_y(A)| = | \phi_x(A) - \phi_z(A)| + | \phi_z(A) - \phi_y(A)|.
\end{equation}
Indeed, if $A = S(x)$ then both sides are equal to $\alpha_S/2$ and the same happens if $A = \overline{S(x)}$. Finally, since $\gamma^x_{[\phi]},\gamma^y_{[\phi]},\gamma^z_{[\phi]} \in [\phi]$, it follows that for every $A \in U(\mathcal{S} \setminus \mathcal{S}(\phi))$, one has:
\[
\gamma^x_{[\phi]}(A) = \gamma^z_{[\phi]}(A) = \gamma^y_{[\phi]}(A)
\]
which together with \eqref{eq:14} imply
\[
d_1(\gamma^x_{[\phi]},\gamma^y_{[\phi]}) = d_1(\gamma^x_{[\phi]},\gamma^z_{[\phi]}) + d_1(\gamma^z_{[\phi]},\gamma^y_{[\phi]}).
\]
Since $x$ and $z$ were chosen arbitrarily in $X$, it follows that $(\Gamma([\phi]),d_1)$ is antipodal. 
\end{proof}

For $\psi \in \bar{T}(\mathcal{S},\alpha)$, we know from Lemma \ref{Lem:l11}, that $\kappa(\psi) \in \mathrm{E}'(d)$. Let us then denote by $[\kappa(\psi)]$ the smallest cell of $\mathrm{E}'(d)$ containing $\kappa(\psi)$ in its relative interior.

\begin{Lem}\label{Lem:l16}
Let $(\mathcal{S},\alpha)$ be a split system pair on a set $X$ and assume that $d:=\sum_{S \in \mathcal{S}} \alpha_S \delta_S$ defines a pseudometric on $X$. Then, for every cell $[\psi]$ of $\bar{T}(\mathcal{S},\alpha)$, one has
\[
\kappa([\psi]) \subset [\kappa(\psi)].
\]
\end{Lem}
\begin{proof}
For each $x \in X$, there is $y \in X$ such that $\kappa(\psi)(x) + \kappa(\psi)(y) = d(x,y)$. It follows from \eqref{eq:132} that
\[
\sum_{S \in \mathcal{S} \setminus \mathcal{S}(\psi)} \alpha_S \delta_S(x,y)
= \sum_{A \in U(\mathcal{S} \setminus \mathcal{S}(\psi))} \bigl[ | \phi_x(A) - \psi(A) | + | \psi(A) - \phi_y(A) | \bigr]
\]
and from the definitions of $\gamma^x_{[\psi]},\gamma^y_{[\psi]}$ it follows that
\[
d_1(\gamma^x_{[\psi]}, \gamma^y_{[\psi]}) = \sum_{S \in \mathcal{S}(\psi)} \alpha_S \delta_S(x,y).
\]
We thus obtain using again the definitions of $\gamma^x_{[\psi]},\gamma^y_{[\psi]}$ for the last equality below:
\begin{align}
d_1(\phi_x,\phi_y) 
&= d(x,y) \nonumber \\
&= \sum_{S \in \mathcal{S} } \alpha_S \delta_S(x,y) \nonumber \\
&= \sum_{S \in \mathcal{S}(\psi) } \alpha_S \delta_S(x,y)+ \sum_{S \in \mathcal{S} \setminus \mathcal{S}(\psi)} \alpha_S \delta_S(x,y)  \nonumber \\
&= d_1(\gamma^x_{[\psi]}, \gamma^y_{[\psi]}) + \sum_{A \in U(\mathcal{S} \setminus \mathcal{S}(\psi))} \bigl[ | \phi_x(A) - \psi(A) | + | \psi(A) - \phi_y(A) | \bigr]  \nonumber \\
&= d_1(\gamma^x_{[\psi]}, \gamma^y_{[\psi]}) + d_1(\phi_x, \gamma^x_{[\psi]}) + d_1(\phi_y, \gamma^y_{[\psi]}) \label{eq:1116}.
\end{align}
Besides, using Lemma \ref{Lem:l14} we obtain for every $\mu \in [\psi]$:
\begin{align}
d_1(\gamma^x_{[\psi]}, \gamma^y_{[\psi]}) 
&= \sum_{S \in \mathcal{S}(\psi) } \alpha_S \delta_S(x,y) \nonumber \\
&= \sum_{S \in \mathcal{S}(\psi) } \alpha_S \nonumber \\
&= \sum_{S \in \mathcal{S}(\psi) } 
\begin{aligned}[t]
&\bigl[ (\mu(S(x)) - 0 ) + (\alpha_S/2 -  \mu(S(x)) ) \nonumber \\
&+ (\alpha_S/2 -  \mu(\overline{S(x)}) ) + (\mu(\overline{S(x)}) - 0 ) \bigr] 
\end{aligned} \nonumber\\
&= \sum_{S \in U(\mathcal{S}(\psi)) } \bigl[ |\phi_x(A) -\mu(A)| + |\phi_y(A) -\mu(A)| \bigr] \nonumber \\
&= d_1(\gamma^x_{[\psi]}, \mu) + d_1(\mu, \gamma^y_{[\psi]}) \label{eq:1117}
\end{align}
where the last equality follows from the fact that since $\mu,\gamma^x_{[\psi]},\gamma^y_{[\psi]} \in [\psi]$ one has for every $A \in U(\mathcal{S} \setminus \mathcal{S}(\psi))$ that $\psi(A) = \mu(A)$ as well as $\gamma^x_{[\psi]}(A) = \psi(A) = \gamma^y_{[\psi]}(A)$. We thus deduce from \eqref{eq:1116} and \eqref{eq:1117} that 
\[
d_1(\phi_x,\phi_y)  = d_1(\phi_x, \gamma^x_{[\psi]}) + d_1(\gamma^x_{[\psi]}, \mu) + d_1(\mu, \gamma^y_{[\psi]}) +  d_1(\gamma^y_{[\psi]},\phi_y).
\]
It thus follows that
\[
d(x,y) = d_1(\phi_x,\phi_y) = d_1(\phi_x, \mu) + d_1(\mu,\phi_y) = \kappa(\mu)(x) + \kappa(\mu)(y).
\]
Since $\{x,y\}$ was an arbitrary edge of $A(\kappa(\psi))$, we obtain that $A(\kappa(\psi)) \subset A(\kappa(\mu))$ and this is equivalent to $\kappa(\mu) \in [\kappa(\psi)]$. This finishes the proof.
\end{proof}

For $|X|<\infty$, it is easy to see using Lemma~\ref{Lem:FirstPartDecompositionDelta} that $\mathrm{E}'(d)=\mathrm{E}(d)$ is a subcomplex of the zonotope $\sum_{S \in \mathcal{S}} \alpha_S \mathrm{E}(\delta_S)$. Indeed, for every cell $[f]$ of $\mathrm{E}(d)$, we set $H(A(f)):= \bigcap_{\{x,y\} \in A(f)} \partial H_{\{x,y\}}$ where $H_{\{x,y\}}:=\{ g \in \R^X : g(x) + g(y) \ge d(x,y)\}$ and we have cf. \cite{Lan}:
\[
[f] = \mathrm{E}(d) \cap H(A(f)) = \Delta(d) \cap H(A(f)).
\]
It follows that every cell of $\mathrm{E}(d)$ is a face of the convex polyhedron $\Delta(d)=\sum_{S \in \mathcal{S}} \alpha_S \Delta(\delta_S)$. Thus, we can write $[f]:= \partial H \cap \Delta(d)$ for every cell $[f]$ of $\mathrm{E}(d)$, where $H \subset \R^X$ is a half-space containing $\Delta(d)$. It easily follows that $[f]= \partial H \cap \sum_{S \in \mathcal{S}} \alpha_S \mathrm{E}(\delta_S)$ and thus $[f]$ is a face of $\sum_{S \in \mathcal{S}} \alpha_S \mathrm{E}(\delta_S)$. In case $|X|=\infty$, one can make the following observations:

\begin{Rem}\label{Rem:subcomplex}
Assume $(X,d)$ has integer-valued metric, is totally split-decomposable, and satisfies the (LRC). In the proof Lemma \ref{Lem:l17} below, we only need that every cell $[f]$ of $\mathrm{E}(d)$ can be written as
\begin{equation}\label{eq:Cel0}
[f] = \sum_{S \in \mathcal{S}_{[f]}} \alpha_S \mathrm{E}(\delta_S) + \sum_{S \in \mathcal{S} \setminus \mathcal{S}_{[f]}} \alpha_S p_S
\end{equation}
where for each $S \in \mathcal{S} \setminus \mathcal{S}_{[f]}$, one has $p_S \in \{0,1\}^{X}$. To see that this holds, note first that as above in the finite case, one has:
\[
[f] = \mathrm{E}(d) \cap H(A(f)) = \Bigl( \sum_{S \in \mathcal{S}} \alpha_S \mathrm{E}(\delta_S) \Bigr) \cap H(A(f)).
\]
Set 
\[
Z_{\{x,y\}}:= \Bigl( \sum_{S \in \mathcal{S}} \alpha_S \mathrm{E}(\delta_S) \Bigr) \cap \partial H_{\{x,y\}}
\]
where as above
\[
H_{\{x,y\}} = \Bigl \{ g \in \R^X : g(x) + g(y) \ge d(x,y) = \sum_{ \substack{S \in \mathcal{S} \\ S(x) \neq S(y)} } \alpha_S \Bigr \}.
\]
It is then easy to see that by definition of the sets $\mathrm{E}(\delta_S)$ and since one has a decomposition $f = \sum_{S \in \mathcal{S}} \alpha_S f_S$ with $f_S \in \mathrm{E}(\delta_S)$ for every $S \in \mathcal{S}$, it follows that:
\[
Z_{\{x,y\}} = \sum_{\substack{S \in \mathcal{S} \\ S(x) \neq S(y)}} \alpha_S \mathrm{E}(\delta_S) + \sum_{\substack{S \in \mathcal{S} \\ S(x) = S(y)}} \alpha_S p_S
\]
where $p_S|_{S(x)} \equiv 0$ and $p_S|_{\overline{S(x)}} \equiv 1$ and additionally, for every $\sum_{S \in \mathcal{S}} \alpha_S g_S \in Z_{\{x,y\}}$, it follows that $g_S=p_S$ for $S \in \mathcal{S}$ so that $S(x) = S(y)$. In other words, one has the stronger property:
\begin{align*}
Z_{\{x,y\}} = \Bigl( \sum_{S \in \mathcal{S}} \alpha_S \mathrm{E}(\delta_S) \Bigr) \setminus \Bigl \{ &\sum_{S \in \mathcal{S}} \alpha_S g_S \in \sum_{S \in \mathcal{S}} \alpha_S \mathrm{E}(\delta_S) : \\
&\exists S \in \mathcal{S} \text{ such that } S(x) = S(y) \text{ and } g_S \neq p_S  \Bigr \}.
\end{align*}
Taking repeatedly intersections of $Z_{\{x,y\}}$ with sets of the form $\partial H_{\{x',y'\}}$ where $\{x',y'\} \in A(f)$, we obtain after finitely many steps (the set $\{S \in \mathcal{S}: S(x) \neq S(y) \}$ being finite):
\begin{equation}\label{eq:Cel1}
[f] = \Bigl( \sum_{S \in \mathcal{S}} \alpha_S \mathrm{E}(\delta_S) \Bigr) \cap \bigcap_{\{x,y\} \in A(f)} \partial H_{\{x,y\}} = \sum_{\substack{S \in \mathcal{S}: \\ \forall \{x,y\} \in A(f) \\ S(x) \neq S(y)} } \alpha_S \mathrm{E}(\delta_S) + \sum_{\substack{S \in \mathcal{S} \\ \exists \{x,y\} \in A(f): \\ S(x) = S(y)}} \alpha_S p_S
\end{equation}
where $p_S|_{S(x)} \equiv 0$ and $p_S|_{\overline{S(x)}} \equiv 1$ for $\{x,y\} \in A(f)$ satisfying $S(x) = S(y)$. As before, the following stronger property holds:
\begin{align}\label{eq:Cel2}
[f] = \Bigl( \sum_{S \in \mathcal{S}} \alpha_S \mathrm{E}(\delta_S) \Bigr)
\setminus \bigcup_{\{x,y\} \in A(f)} \Bigl \{ &\sum_{S \in \mathcal{S}} \alpha_S g_S \in \sum_{S \in \mathcal{S}} \alpha_S \mathrm{E}(\delta_S) : \nonumber \\
&\exists S \in \mathcal{S} \text{ such that } S(x) = S(y) \text{ and } g_S \neq p_S  \Bigr \}.
\end{align}
\end{Rem}

We go on with a more concrete description of the representation of the cells of $\mathrm{E}(X,d)$ in the case where each of them is a combinatorial hypercube.

\begin{Rem}\label{Rem:cellsdimension}
It is not difficult to see if $(X,d)$ is as in Remark \ref{Rem:subcomplex} and if every cell $[f]$ of $\mathrm{E}(X,d)$ is a combinatorial hypercube, then the representation \eqref{eq:Cel0}
verifies 
\begin{equation}\label{eq:CelDim}
|\mathcal{S}_{[f]}|=\mathrm{dim}([f]).
\end{equation}
Indeed, for every $1$-cell $[g]$ of $\mathrm{E}(X,d)$, one can represent $[g]$ as in  \eqref{eq:Cel0}, namely: 
\[
[g]=\sum_{S \in \mathcal{S}_{[g]}} \alpha_S \mathrm{E}(\delta_S) + \sum_{S \in \mathcal{S} \setminus \mathcal{S}_{[g]}} \alpha_S p^g_S.
\]
Now, note that the affine hull $\mathrm{aff}([g])$ is a $1$-dimensional affine subspace of $\R^X$ which contains for every $S \in \mathcal{S}_{[g]}$, a translate of $\mathrm{aff}(\mathrm{E}(\delta_S))$. Hence, if $|\mathcal{S}_{[g]}| \ge 2$, then for $S,S' \in \mathcal{S}_{[g]}$ with $S' \neq S''$, $\mathrm{aff}(\mathrm{E}(\delta_{S'}))$ and $\mathrm{aff}(\mathrm{E}(\delta_{S''}))$ have colinear directional vectors. This is however impossible since for any $S \in \mathcal{S}$, the directional vector of $\mathrm{E}(\delta_S)$ with $S=\{A,B\}$ is a scalar multiple of the function $h \in \R^X$ that satisfies $h|_A \equiv 1$ and $h|_B \equiv -1$.

It is easy to see by induction that if $Z$ is a zonotope combinatorially equivalent to an $n$-hypercube, then for any vertex $z$ of $Z$, it follows that $Z$ is the Minkowski sum of all its edges incident to $z$. Indeed, assume $z$ is any vertex of $Z$. By the combinatorial $n$-hypercube equivalence of $Z$, $z$ is incident to exactly $n$ different edges. All facets of $Z$ incident to $z$ are again zonotopes and are combinatorially equivalent to $(n-1)$-hypercubes. There are $n$ such facets and each of them is by induction the sum of $n-1$ edges among those $n$ edges incident to $z$. Since $Z$ is a zonotope, it is centrally symmetric. Thus, the symmetric image of each facet of $Z$ incident to $z$ is again a facet of $Z$ that can be written as a sum of edges. Hence, there are $2n$ facets of $Z$ that can be written as a Minkowski sum of edges of $Z$. Since $Z$ must have exactly $2n$ facets, the result follows. 

Now, since every cell $[f]$ is a combinatorial hypercube as well as a zonotope, it is thus equal to the Minkowski sum of its edges which are in turn $1$-cells of $\mathrm{E}(X,d)$. It follows that for $i \in \{1,\dots,\mathrm{dim}([f])\}$, we can pick  $1$-cells $[f_i]$ of $\mathrm{E}(X,d)$, so that they all intersect in the vertex $f_0$ of $[f]$, and we can write:
\[
[f] = f_0 + \sum_{i \in \{1,\dots,\mathrm{dim}([f])\}} ([f_i] - f_0)
\]
Using the representation \eqref{eq:Cel0} for $[f]$ and for each $[f_i]$, we can thus write:
\[
[f] = f_0 + \sum_{i \in \{1,\dots,\mathrm{dim}([f])\}} \left( \alpha_{S_i} \mathrm{E}(\delta_{S_i}) + \sum_{S \in \mathcal{S} \setminus \{ S_i\} } \alpha_S p_S - f_0\right).
\]
Since $f_0 = \sum_{S \in \mathcal{S} } \alpha_S p_S = \sum_{S \in \mathcal{S} \setminus\{ S_1,\dots,S_{\mathrm{dim}([f])} \}} \alpha_S p_S + \sum_{i \in \{1,\dots,\mathrm{dim}([f])\}} \alpha_{S_i} p_{S_i} $, it follows:
\begin{align*}
[f] 
&= f_0 + \sum_{i \in \{1,\dots,\mathrm{dim}([f])\}} \bigl( \alpha_{S_i} \mathrm{E}(\delta_{S_i}) - \alpha_{S_i} p_{S_i} \bigr) \\
&= \sum_{S \in \mathcal{S} \setminus\{ S_1,\dots,S_{\mathrm{dim}([f])} \}} \alpha_S p_S + \sum_{i \in \{1,\dots,\mathrm{dim}([f])\}} \alpha_{S_i} \mathrm{E}(\delta_{S_i}),
\end{align*}
which implies that $\mathcal{S}_{[f]}=\{S_i\}_{i \in \{1,\dots,\mathrm{dim}([f])\}}$ since it is already clear that $\{S_i\}_{i \in \{1,\dots,\mathrm{dim}([f])\}} \subset \mathcal{S}_{[f]}$ and thus they must be equal by the above (i.e., in $\mathrm{aff}(P+P')$, if $P + P' = P$, then $P'=\{0\}$).
\end{Rem}

\begin{Rem}\label{Rem:CellForm}
If $(X,d)$ is again as in Remark \ref{Rem:subcomplex} and if every cell $[f]$ of $\mathrm{E}(X,d)$ is a combinatorial hypercube, it is easy to see that if $f =  \sum_{S \in \mathcal{S} } \alpha_S f_S$ as usual with $\alpha_S>0$ and $f_S \in \mathrm{E}(\delta_S)$ and setting
\[
\mathcal{S}_f := \{ S \in \mathcal{S} : f_S(X) \subset (0,1) \},
\]
then
\[
[f] = \sum_{S \in \mathcal{S}_f} \alpha_S \mathrm{E}(\delta_S) + \sum_{S \in \mathcal{S} \setminus \mathcal{S}_f} \alpha_S p_S,
\]
i.e. $\mathcal{S}_{[f]}=\mathcal{S}_f$. Let $[f]$ be given by its representation as in \eqref{eq:Cel0}. Note first that for any $S \in \mathcal{S}_f$ and for any pair of points $x,y \in X$ such that $f(x) + f(y)=d(x,y)$, one has equivalently
\[
\sum_{S \in \mathcal{S} } \alpha_S f_S(x) + \sum_{S \in \mathcal{S} } \alpha_S f_S(y) = \sum_{S \in \mathcal{S} } \alpha_S \delta_S(x,y),
\]
hence one necessarily has $S(x) \neq S(y)$ and thus by \eqref{eq:Cel1}, it follows that $\mathcal{S}_f \subset \mathcal{S}_{[f]}$. Now, for the other inclusion, assume that $\mathcal{S}_f \subsetneq \mathcal{S}_{[f]}$. Since $[f]$ is a hypercube, it follows that
\[
\sum_{S \in \mathcal{S}_f} \alpha_S \mathrm{E}(\delta_S) + \sum_{S \in \mathcal{S} \setminus \mathcal{S}_f} \alpha_S p_S
\]
is a strict subcell of $[f]$ containing $f$ and this contradicts the definition of $[f]$.
\end{Rem}

\begin{Rem}\label{Rem:kappainverse} 
For $(X,d)$ as in Remark \ref{Rem:subcomplex}, let us define the map 
\[
\lambda \colon \mathrm{E}(X,d) \to \bar{T}(\mathcal{S},\alpha)
\]
by the assignement $f \mapsto \psi_f$ where $\psi_f$ (as defined in the proof of Lemma \ref{Lem:l12}) is depending on a choice of a representation $\sum_{S \in \mathcal{S}}\alpha_S f_S$ for $f$, and this choice is not unique in general. Furthermore, note that one always has $\kappa \circ \lambda = \mathrm{id}_{\mathrm{E}(X,d)}$.  It follows that $\kappa$ is surjective. In general however, $\lambda \circ \kappa  \neq \mathrm{id}_{\bar{T}(\mathcal{S},\alpha)}$. 
\end{Rem}

We go on with a more concrete description of the maps $\kappa$ and $\lambda$:

\begin{Rem}\label{Rem:inducedcell}
Again, if $(X,d)$ is as in Remark \ref{Rem:subcomplex} and if every cell $[f]$ of $\mathrm{E}(X,d)$ is a combinatorial hypercube, note that $\kappa \colon \bar{T}(\mathcal{S},\alpha) \to \mathrm{E}(X,d)$ is given by 
\begin{equation}\label{eq:Defkappa}
\psi \mapsto \left( \kappa(\psi) : x \mapsto  \sum_{S \in \mathcal{S}} 2 \psi(S(x)) =  \sum_{S \in \mathcal{S}} \alpha_S \frac{ \psi(S(x))}{\alpha_S/2} =:
\sum_{S \in \mathcal{S}} \alpha_S f_S(x) \right)
\end{equation}
and $\lambda \colon \mathrm{E}(X,d) \to \bar{T}(\mathcal{S},\alpha)$ is given by $f \mapsto \psi_f$ where recall that $\psi_f$ is given for any $A \in U(\mathcal{S})$ and for an arbitrarily chosen $x \in A$ by 
\[
\psi_f : A \mapsto \frac{\alpha_{S_A}}{2} f_{S_A}(x).
\]
By Remark \ref{Rem:CellForm}, it follows that one has:
\[
[\kappa(\psi)]=
\sum_{S \in \mathcal{S}_{\kappa(\psi)}} \alpha_S \mathrm{E}(\delta_S) + \sum_{S \in \mathcal{S} \setminus \mathcal{S}_{\kappa(\psi)}} \alpha_S p^{\kappa(\psi)}_S
\]
and thus:
\begin{enumerate}[$(a)$]
\item $\lambda([\kappa(\psi)])=[\psi]$,
\item moreover,
\[
[\kappa(\psi)] = (\kappa \circ \lambda)([\kappa(\psi)]) = \kappa([\psi])
\]
where the first equality was already noted in Remark \ref{Rem:kappainverse} and the second equality follows from $(a)$.
\item To see that $\kappa$ is injective, assume $g:=\kappa(\psi)=\kappa(\psi')=:g'$. Then, in particular, one has $[g]=[g']$. Thus, by \eqref{eq:Cel2} one has $\mathcal{S}_g=\mathcal{S}_{g'}$ as well as $p^g_S=p^{g'}_S$ for each $S \in \mathcal{S} \setminus \mathcal{S}_g$. It follows that $\sum_{S \in \mathcal{S}_{g}} \alpha_S g_S= \sum_{S \in \mathcal{S}_{g'}} \alpha_S g_S'$. Since  $\sum_{S \in \mathcal{S}_g} \alpha_S \mathrm{E}(\delta_S)$ is a finite dimensional zonotope combinatorially equivalent to a hypercube, it follows that $g_S=g_S'$ for every  $S \in \mathcal{S}_{g}$ and thus by \eqref{eq:Defkappa}, it follows $\psi=\psi'$. Hence together with Remark \ref{Rem:kappainverse}, it follows that $\kappa$ is bijective with inverse $\lambda$.
\item In addition:
\[
\lambda([f]) = \lambda( [(\kappa \circ \lambda)(f)]) 
= (\lambda \circ \kappa)[ \lambda(f)]
= [ \lambda(f)]
\]
where the first equality was already noted in Remark \ref{Rem:kappainverse}, the second equality follows from $(b)$ and the last equality follows from $(c)$.
\item By Remarks~\ref{Rem:cellsdimension} and \ref{Rem:CellForm}, we have that $\mathrm{dim}([f]) = k$ implies $\mathrm{dim}([\lambda(f)]) = k$.
\item Finally, if $\mathrm{dim}([\psi]) = k$, then $\mathrm{dim}([\kappa(\psi)]) = k$ since if we assume that $\mathrm{dim}([\kappa(\psi)]) > k$, it then follows as in $(e)$ from Remarks~\ref{Rem:cellsdimension} and \ref{Rem:CellForm} that $\mathrm{dim}(\lambda([\kappa(\psi)])) > k$ and by $(a)$ it follows that $\mathrm{dim}([\psi]) > k$ which is a contradiction.
\end{enumerate}
Hence, the map 
\[
\kappa \colon \bar{T}(\mathcal{S},\alpha) \to \mathrm{E}(X,d)
\]
defines a bijection as well as an isomorphism of cell complexes.

\end{Rem}

In view of Remark \ref{Rem:subcomplex}, we have:

\begin{Lem}\label{Lem:l17}
Let $(X,d)$ be a totally split-decomposable metric space with integer-valued metric which satisfies the (LRC). For every cell $[f]$ of $\mathrm{E}(d)$, there is $\bar{f} \in [f]$ such that $[\bar{f}] = [f]$ and so that for $\psi_{\bar{f}} \in \bar{T}(S,\alpha)$ as defined in the proof of Lemma~\ref{Lem:l12}, one has
\[
[\kappa(\psi_{\bar{f}})] \subset \kappa([\psi_{\bar{f}}]).
\]
\end{Lem}
\begin{proof}
By the above remark, there is $\mathcal{S}_{[f]} \subset \mathcal{S}$ such that
\[
[f] = \sum_{S \in \mathcal{S}_{[f]}} \alpha_S \mathrm{E}(\delta_S) + \sum_{S \in \mathcal{S} \setminus \mathcal{S}_{[f]}} \alpha_S p_S
\]
where for each $S \in \mathcal{S} \setminus \mathcal{S}_{[f]}$, one has $p_S \in \{0,1\}^{X}$. We can thus write
\[
f = \sum_{S \in \mathcal{S}_{[f]}} \alpha_S f_S + \sum_{S \in \mathcal{S} \setminus \mathcal{S}_{[f]}} \alpha_S p_S
\]
where for every $S \in \mathcal{S}$, one has $f_S \in \mathrm{E}(\delta_S)$. Let us moreover define
\[
\mathcal{S}_f := \{ S \in \mathcal{S} : f_S(X) \subset (0,1) \}.
\]
It is clear that $\mathcal{S}_f \subset \mathcal{S}_{[f]}$. Moreover, $\mathcal{S}(\psi_f) = \mathcal{S}_f$ for $\psi_f$ defined as in the proof of Lemma \ref{Lem:l12}. Let us now define $\bar{f} \in [f]$ as
\[
\bar{f} := \sum_{S \in \mathcal{S}_{[f]}} \alpha_S \bar{f}_S + \sum_{S \in \mathcal{S} \setminus \mathcal{S}_{[f]}} \alpha_S p_S
\]
such that for any $S \in \mathcal{S}_{[f]}$, $\bar{f}_S$ is constantly equal to $1/2$ on $X$. It is then clear that $\mathcal{S}_{\bar{f}} = \mathcal{S}_{[f]}$. Since $\mathcal{S}(\psi_{\bar{f}}) = \mathcal{S}_{\bar{f}} = \mathcal{S}_{[f]} \supset \mathcal{S}_f = \mathcal{S}(\psi_{f}) $ and because for every $A \in U(\mathcal{S} \setminus \mathcal{S}(\psi_{\bar{f}}))= U(\mathcal{S} \setminus \mathcal{S}_{[f]})$, one has
\[
\psi_{\bar{f}}(A) = \psi_{f}(A),
\]
it follows that $\mathrm{supp}(\psi_f) \subset \mathrm{supp}(\psi_{\bar{f}})$ and thus $\psi_f \in [\psi_{\bar{f}}]$. Now for any $g \in [f]$, one similarly has
\[
\mathcal{S}(\psi_{g}) = \mathcal{S}_{g} \subset \mathcal{S}_{[f]} = \mathcal{S}(\psi_{\bar{f}})
\]
as well as for every $A \in U(\mathcal{S} \setminus \mathcal{S}_{[f]})=U(\mathcal{S} \setminus \mathcal{S}(\psi_{\bar{f}}))$
\[
\psi_{g}(A) = \psi_{\bar{f}}(A).
\]
Thus, $\mathrm{supp}(\psi_g) \subset \mathrm{supp}(\psi_{\bar{f}})$ and thus $\psi_g \in [\psi_{\bar{f}}]$, that is $g = \kappa(\psi_{g}) \in \kappa([\psi_{\bar{f}}])$. It follows that
\[
[\kappa(\psi_{\bar{f}})] =  [\bar{f}] = [f] \subset \kappa([\psi_{\bar{f}}]),
\]
which finishes the proof.
\end{proof}
Consider the isometric embedding $\mathrm{e} \colon X \to \mathrm{E}(d)$ given by
\[
x \mapsto ( d_x \colon y \mapsto d(x,y) )
\]
where $\mathrm{E}(d)$ is endowed with the metric $d_{\infty}(f,g) := \|f - g\|_\infty$. We say that $\mathrm{E}(d)$ is \textit{cell-decomposable} if every cell $C$ of $\mathrm{E}(d)$ is $X$-gated (cf. Definition \ref{Def:d1}). Now, we have the following:

\begin{Lem}\label{Lem:l18}
Let $(X,d)$ be a totally split-decomposable metric space with integer-valued metric satisfying the (LRC). Then, $\mathrm{E}(d)$ is cell-decomposable.
\end{Lem}
\begin{proof}
Let $C$ be a cell of $\mathrm{E}(d)=\mathrm{E}'(d)$. By Lemma \ref{Lem:l16} and \ref{Lem:l17}, there is $\bar{f} \in C$ such that $[\bar{f}] = C$ and such that $\psi_{\bar{f}} \in \bar{T}(S,\alpha)$ (as defined in the proof of Lemma \ref{Lem:l12}) satisfies
\[
C = [\bar{f}] = [\kappa(\psi_{\bar{f}})] = \kappa([\psi_{\bar{f}}]).
\]
as well as for every $g \in [\bar{f}]$:
\begin{equation}\label{eq:e117}
\mathrm{supp}(\psi_{g}) \subset \mathrm{supp}(\psi_{\bar{f}}).
\end{equation}
Let $x \in X$ be chosen arbitrarily. We want to show that 
\[
\kappa(\gamma^x_{[\psi_{\bar{f}}]}) \text{ is a gate for }d_x \text{ in }[\bar{f}].
\]
For an arbitrarily chosen $f \in [\bar{f}]$, let us set $\psi:=\psi_{f}$ which by \eqref{eq:e117} satisfies $\psi_{f} \in [\psi_{\bar{f}}]$. Now, by Lemma \ref{Lem:l11}, there must exist $y \in X$ such that
\[
\kappa(\psi_{\bar{f}})(x) + \kappa(\psi_{\bar{f}})(y) = \bar{f}(x) + \bar{f}(y) = d(x,y).
\]
Since $\kappa(\psi_{f}) = f \in [\bar{f}]$, one has $A(\bar{f}) \subset A(f)$, cf. \eqref{eq:e000000} and thus $\kappa(\psi)(x) + \kappa(\psi)(y) = d(x,y)$, hence
\begin{align}
d_1(\phi_x,\phi_y) \nonumber
&= d(x,y) \\ \nonumber
&= \kappa(\psi)(x) + \kappa(\psi)(y) \\ \nonumber
&= d_1(\phi_x,\psi)  + d_1(\psi,\phi_y)  \\ 
&= d_1(\phi_x, \gamma^x_{[\psi_{\bar{f}}]}) + d_1(\gamma^x_{[\psi_{\bar{f}}]}, \psi) + d_1(\psi,\gamma^y_{[\psi_{\bar{f}}]}) +  d_1(\gamma^y_{[\psi_{\bar{f}}]},\phi_y) \label{eq:e17}
\end{align}
and since $\kappa$ is $1$-Lipschitz, it follows that
\begin{align}
\left\| \kappa(\phi_x) - \kappa(\phi_y) \right\|_{\infty} 
&\le  
\begin{aligned}[t]
&\left\| \kappa(\phi_x) - \kappa(\gamma^x_{[\psi_{\bar{f}}]}) \right\|_{\infty}  + \left\|\kappa(\gamma^x_{[\psi_{\bar{f}}]}) - \kappa(\psi) \right\|_{\infty} \\ \nonumber
&+ \left\| \kappa(\psi) - \kappa(\gamma^y_{[\psi_{\bar{f}}]}) \right\|_{\infty}  + \left\| \kappa(\gamma^y_{[\psi_{\bar{f}}]}) - \kappa(\phi_y) \right\|_{\infty} 
\end{aligned}\\
&\le
\begin{aligned}[t]
&d_1(\phi_x, \gamma^x_{[\psi_{\bar{f}}]}) + d_1(\gamma^x_{[\psi_{\bar{f}}]}, \psi) \\ \nonumber
&+ d_1(\psi,\gamma^y_{[\psi_{\bar{f}}]}) +  d_1(\gamma^y_{[\psi_{\bar{f}}]},\phi_y)
\end{aligned}\\
&= d(x,y).\label{eq:e18}
\end{align}
It is easy to see that $\kappa(\phi_x) = d_x$ and $ \kappa(\phi_y)=d_y$ as well as $\left\| d_x - d_y \right\|_{\infty}= d(x,y)$ which implies that both inequalities above are actual equalities. Since $\kappa(\psi_{f}) = f$, we thus obtain
\begin{align*}
\left\| d_x - d_y \right\|_{\infty} 
&=  
\begin{aligned}[t]
&\left\| d_x - \kappa(\gamma^x_{[\psi_{\bar{f}}]}) \right\|_{\infty}  + \left\|\kappa(\gamma^x_{[\psi_{\bar{f}}]}) - f \right\|_{\infty} \\
&+ \left\| f - \kappa(\gamma^y_{[\psi_{\bar{f}}]}) \right\|_{\infty}  + \left\| \kappa(\gamma^y_{[\psi_{\bar{f}}]}) - d_y \right\|_{\infty}.
\end{aligned}
\end{align*}
In particular
\[
\left\| d_x - \kappa(\gamma^x_{[\psi_{\bar{f}}]}) \right\|_{\infty}  + \left\|\kappa(\gamma^x_{[\psi_{\bar{f}}]}) - f \right\|_{\infty}
= \left\| d_x - f \right\|_{\infty} 
\]
and this proves that $\kappa(\gamma^x_{[\psi_{\bar{f}}]})$ is a gate for $d_x$ in $[\bar{f}]$. This is the desired result.
\end{proof}

It is easy to see that \cite[Theorem~1.1]{HubKMII} generalizes to the case where $|X|=\infty$ as long as $\mathrm{E}'(X,d)=\mathrm{E}(X,d)$. To be self-contained, we give a proof of the theorem.

\begin{Thm}\label{Thm:AntipodalGateSetsCells}
Let $(X,d)$ be a (possibly) infinite metric space with integer-valued metric satisfying the (LRC). If $f \in \mathrm{E}(X,d)$ is such that $[f]$ is $X$-gated, then the following hold:
\begin{enumerate}[(i)]
\item $(G([f]),d_{\infty})$ is an antipodal metric space.
\item The map $\Phi \colon ([f],d_{\infty}) \to \mathrm{E}(G([f]),d_{\infty})$ given by 
\begin{equation}\label{eq:Definition1Phi}
g \mapsto \left(\gamma^x_{[f]} \mapsto  g(x) - \gamma^x_{[f]}(x) \right)
\end{equation}
is a bijective isometry as well as an isomorphism of polytopes.
\end{enumerate}
\end{Thm}
\begin{proof}
If $\mathrm{rank}(A(f))= 0$, the result clearly holds, hence let $f \in \mathrm{E}(d)$ be such that $\mathrm{rank}(A(f))\ge 1$. 

We first show that $(G([f]),d_{\infty})$ is an antipodal metric space. For each $x \in X$, consider
\[
\rho(x) := \{ y \in X : \{x,y\} \in A(f) \}.
\]
We define the map $\sigma \colon G([f]) \to G([f])$ by sending every gate $\gamma^x_{[f]}$ to the gate $\gamma^y_{[f]}$ where $y \in \rho(x)$ is chosen arbitrarily. To see that $\sigma$ is well-defined, note that for every $g \in [f]$ and if $\{x,y \} \in A(g)$, one has
\begin{align}
d_{\infty}(d_x,d_y) &= d(x,y) \nonumber\\
&= g(x) + g(y) \nonumber \\
&= d_{\infty}(d_x,g) + d_{\infty}(d_y,g) \nonumber \\
&= d_{\infty}(d_x,\gamma^x_{[f]}) + d_{\infty}(\gamma^x_{[f]},g)+ d_{\infty}(g,\gamma^y_{[f]}) + d_{\infty}(\gamma^y_{[f]},d_y) \label{eq:geomppties2}
\end{align}
which rearranging and using the triangle inequality gives that for any $y \in \rho(x)$ and any $g \in [f]$, one has
\begin{equation}\label{eq:geomppties1}
d_{\infty}(\gamma^x_{[f]},\gamma^y_{[f]}) = d_{\infty}(\gamma^x_{[f]},g)+ d_{\infty}(g,\gamma^y_{[f]}).
\end{equation}
Now, let $x' \in X$ be such that $\gamma^{x'}_{[f]}=\gamma^x_{[f]}$. For any $y' \in \rho(x')$, one can use \eqref{eq:geomppties1} to obtain
\begin{align*}
d_{\infty}(\gamma^{x'}_{[f]},\gamma^{y'}_{[f]}) 
= d_{\infty}(\gamma^x_{[f]},\gamma^{y'}_{[f]}) 
&= d_{\infty}(\gamma^x_{[f]},\gamma^y_{[f]}) - d_{\infty}(\gamma^y_{[f]},\gamma^{y'}_{[f]})\\
&= d_{\infty}(\gamma^{x'}_{[f]},\gamma^y_{[f]}) - d_{\infty}(\gamma^y_{[f]},\gamma^{y'}_{[f]}) \\
&= d_{\infty}(\gamma^{x'}_{[f]},\gamma^{y'}_{[f]}) - 2 d_{\infty}(\gamma^y_{[f]},\gamma^{y'}_{[f]})
\end{align*}
from which it follows that $\gamma^{y'}_{[f]}=\gamma^y_{[f]}$, and this is the desired result. This proves that $\sigma$ is well-defined. It is now clear that $\sigma$ is an involution which turns $(G([f]),d_{\infty})$ into an antipodal metric space.

We now show that $\Phi$ defines a bijective isometry. Note that $\Phi$ can also be expressed as
\begin{equation}\label{eq:Definition2Phi}
\Phi(g) : \gamma^x_{[f]} \mapsto d_{\infty}(g,d_x) - d_{\infty}(\gamma^x_{[f]},d_x) 
\end{equation}
and also as 
\begin{equation}\label{eq:Definition3Phi}
\Phi(g) = d^g_{\infty} : \gamma^x_{[f]} \mapsto d_{\infty}(g,\gamma^x_{[f]})
\end{equation}

It is easy to see by \eqref{eq:Definition3Phi} that $\Phi$ is well-defined, i.e. it does not depend on the choice of $x$ or $x'$ as long as $\gamma^x_{[f]}=\gamma^{x'}_{[f]}$. Moreover, for every $g \in [f]$, one clearly has $\Phi(g) \in \Delta(G([f]),d_{\infty})$ by \eqref{eq:Definition3Phi}. From \eqref{eq:geomppties1} and \eqref{eq:Definition3Phi}, it follows that for every $x \in X$, one has
\[
\Phi(g)(\gamma^x_{[f]}) + \Phi(g)(\sigma(\gamma^x_{[f]})) = d(\gamma^x_{[f]},\sigma(\gamma^x_{[f]})).
\]
which shows that $\Phi(g) \in \mathrm{E}(G([f]),d_{\infty})$. To see that $\Phi$ is surjective, let us define for any $h \in \mathrm{E}(G([f]),d_{\infty})$, the associated function
\[
g' : x \mapsto  h(\gamma^x_{[f]}) + \gamma^x_{[f]}(x) .
\]
We clearly have $\Phi(g') = h$ and thus we only need to show that $g' \in [f]$. We have
\begin{align}\label{eq:EdgesCell}
g'(x) + g'(y) &= h(\gamma^x_{[f]}) + \gamma^x_{[f]}(x) + h(\gamma^y_{[f]}) + \gamma^y_{[f]}(y) \nonumber \\
&\ge d_{\infty}(\gamma^x_{[f]}, \gamma^y_{[f]}) + d_{\infty}(\gamma^x_{[f]}, d_x) + d_{\infty}(\gamma^y_{[f]}, d_y) \nonumber\\
&\ge d_{\infty}(d_x, d_y) \nonumber\\
&= d(x,y),
\end{align}
hence in particular $g' \in \Delta(d)$. If $\{x,y\} \in A(f)$, one has by \eqref{eq:geomppties1}
\begin{align*}
&d_{\infty}(\gamma^x_{[f]}, \gamma^y_{[f]}) + d_{\infty}(\gamma^x_{[f]}, d_x) + d_{\infty}(\gamma^y_{[f]}, d_y) \\
&=d_{\infty}(\gamma^x_{[f]}, f)+ d_{\infty}(f, \gamma^y_{[f]}) + d_{\infty}(\gamma^x_{[f]}, d_x) + d_{\infty}(\gamma^y_{[f]}, d_y) \\
&=d_{\infty}(f,d_x)+ d_{\infty}(f, d_y) \\
&=f(x) + f(y) \\
&=d(x,y),
\end{align*}
and since 
$h \in \mathrm{E}(G([f]),d_{\infty})$, one has
\[
h(\gamma^x_{[f]}) + h(\sigma(\gamma^x_{[f]})) = d_{\infty}(\gamma^x_{[f]}, \sigma(\gamma^x_{[f]}))
\]
(recall that for any extremal function $f$, if $\{x,y\} \in A(f)$ and $xy + yz = xz$, then $\{x,z\} \in A(f)$). Hence if $\{x,y\} \in A(f)$, one can replace all inequalities in \eqref{eq:EdgesCell} by equalities. This shows that $g' \in [f]$ and thus $\Phi$ is surjective. Now, it is easy to see that $\Phi$ preserves distances since for every $g,h \in [f]$, one has by \eqref{eq:Definition2Phi} that 
\[
g(x) - h(x) 
= g(x) - \gamma^x_{[f]}(x) - \left(h(x) - \gamma^x_{[f]}(x) \right) 
= \Phi(g)(\gamma^x_{[f]}) - \Phi(h)(\gamma^x_{[f]})
\]
and thus $\Phi$ is a bijective isometry. Note now that by \cite{Lan} there is an affine isometry $\alpha \colon ([f],d_{\infty}) \to l^n_{\infty}$ and since $\mathrm{E}(G([f]),d_{\infty})$ consists of a unique maximal cell, another affine isometry $\beta \colon \mathrm{E}(G([f]),d_{\infty}) \to l^n_{\infty}$. It follows that the map
\[
\beta \circ \Phi \circ \alpha^{-1} \colon \alpha([f]) \to \beta(\mathrm{E}(G([f]),d_{\infty}))
\]
is a bijective isometry between convex subsets (with non-empty interior) of finite dimensional normed linear spaces. It follows by an extension of Mazur-Ulam Theorem (cf. \cite{Man}) that $\beta \circ \Phi \circ \alpha^{-1}$ is the restriction of an affine bijective isometry. It follows that $\Phi$ has the same property and is thus in particular a polytope isomorphism.
\end{proof}

Let $(X,d)$ be a totally split-decomposable metric space with integer-valued metric satisfying the (LRC). By Lemma \ref{Lem:l18}, we know that every cell $[f] \subset \mathrm{E}'(d)=\mathrm{E}(d)$ is $X$-gated. Hence if $(G([f]),d_{\infty})$ denotes the set of all $X$-gates of $[f]$, and if $d_{\infty} $ denotes the metric $d_{\infty}(f,g) = \left\| f-g \right\|_{\infty}$ (we adopt the same notation for restrictions of $d_{\infty}$), we obtain by Theorem~\ref{Thm:AntipodalGateSetsCells} that the following hold: 
\begin{enumerate}
\item $(G([f]),d_{\infty})$ is an antipodal metric space.
\item $[f]$ and $\mathrm{E}(G([f]),d_{\infty})$ are combinatorially equivalent polytopes.
\end{enumerate}
If we assume that $\mathrm{dim}([f]) = n$, then since $G([f]) \subset [f]$, it follows (cf. \cite[Proposition~3.5 and Theorem~4.3~(1)]{Lan}) that $\mathrm{E}(G([f]),d_{\infty})$ isometrically embeds into $([f],d_{\infty})$ through
\[
\mathrm{E}(G([f]),d_{\infty}) \hookrightarrow \mathrm{E}([f]) \cong  [f].
\]
Thus in particular, one has 
\begin{equation}\label{eq:e66}
\mathrm{dim}(\mathrm{E}(G([f]),d_{\infty})) \le n
\end{equation}
Finally, this implies by \cite[Theorem~1.2]{HubKMII} that
\[
|G([f])| \le 2n.
\]
Indeed, if $|G([f])| \ge 2(n+1)$ (possibly $|G([f])|=\infty$), then since $(G([f]),d_{\infty})$ is antipodal, we can select $(n+1)$ pairs of antipodal points in $G([f])$ to obtain an antipodal metric space $(A,d_{\infty})$ with
\[
A \subset G([f])
\]
and $|A|=2(n+1)$. It follows by \cite[Proposition~3.5]{Lan} that $\mathrm{E}(A,d_{\infty})$ is isometrically embedded in $\mathrm{E}(G([f]),d_{\infty})$ and thus again by \cite[Theorem~1.2]{HubKMII}, one has
\[
n + 1 \le \mathrm{dim}(\mathrm{E}(A,d_{\infty})) \le \mathrm{dim}(\mathrm{E}(G([f]),d_{\infty})
\]
which contradicts \eqref{eq:e66}. Hence, with $(2)$ above, this proves that 
\[
|G([f])| \le 2 \mathrm{dim}([f]).
\]

We continue with the following:

\begin{Lem}\label{Lem:l19}
Let $\kappa \colon (A,d) \to (A',d')$ be a map of metric spaces such that the following hold:
\begin{enumerate}[(i)]
\item $\kappa$ is $1$-Lipschitz,
\item $\kappa$ is surjective,
\item $(A,d)$ is an antipodal metric space and 
\item for any $x \in A$, there is $y \in A$ antipodal to $x$ such that
\[
d(x,y) = d'(\kappa(x),\kappa(y)).
\]
\end{enumerate}
Then, it follows that $\kappa$ is an isometry.
\end{Lem}
\begin{proof}
Let  $x,z \in A$ be chosen arbitrarily. By $(iv)$, there is $y$ antipodal to $x$ such that $d(x,y) = d'(\kappa(x),\kappa(y))$. Hence, one has:
\begin{align*}
d(x,z) + d(z,y) &= d(x,y) \\
&= d'(\kappa(x),\kappa(y)) \\
&\le d'(\kappa(x),\kappa(z)) + d'(\kappa(z),\kappa(y)) \\
&\le d(x,z) + d(z,y) \\
&= d(x,y).
\end{align*}
It follows that the above inequalities are actual equalities and using again that $\kappa$ is $1$-Lipschitz, it follows that
\[
d'(\kappa(x),\kappa(z)) = d(x,z).
\]
Since $x$ and $z$ were chosen arbitrarily, this proves the result.
\end{proof}

\begin{Lem}\label{Lem:l199}
Let $(X,d)$ be a totally split-decomposable metric space with integer-valued metric satisfying the (LRC). Let $[f]$ be any positive dimensional cell of $\mathrm{E}(d)$ and let $\bar{f}$ as well as $\psi_{\bar{f}}$ be defined as in Lemma \ref{Lem:l17}. Then, the map
\[
\bar{\kappa} := \kappa|_{\Gamma([\psi_{\bar{f}}])} \colon (\Gamma([\psi_{\bar{f}}]),d_1) \to (G([f]),d_{\infty})
\]
is an isometry.
\end{Lem}
\begin{proof}
We already know that $\bar{\kappa}$ is $1$-Lipschitz and it is surjective by the proof of Lemma \ref{Lem:l18}. Now, for any $x \in X$, there is $y \in X$ such that
\[
\bar{f}(x) + \bar{f}(y) = d(x,y).
\]
By the proofs of Lemmas \ref{Lem:l14} and \ref{Lem:l15}, it follows (by definition $\kappa(\psi_{\bar{f}})=\bar{f}$) that $\gamma^x_{[\psi_{\bar{f}}]}$ and $\gamma^y_{[\psi_{\bar{f}}]}$ are antipodal in $(\Gamma([\psi_{\bar{f}}]),d_1)$. Furthermore, by \eqref{eq:e17} and \eqref{eq:e18}, one has 
\[
d_1(\gamma^x_{[\psi_{\bar{f}}]},\gamma^y_{[\psi_{\bar{f}}]}) = \left\| \kappa(\gamma^x_{[\psi_{\bar{f}}]}) - \kappa(\gamma^y_{[\psi_{\bar{f}}]}) \right\|_{\infty}.
\]
We can thus apply Lemma \ref{Lem:l19} to deduce that $\bar{\kappa}$ is an isometry.
\end{proof}

For $x,y \in X$ arbitrarily chosen, it follows from the definitions of $\gamma^x_{[\psi_{\bar{f}}]}$ and $\gamma^y_{[\psi_{\bar{f}}]}$, that one has
\begin{equation}\label{eq:e199}
d_1(\gamma^x_{[\psi_{\bar{f}}]},\gamma^y_{[\psi_{\bar{f}}]}) = \sum_{S \in \mathcal{S}(\psi_{\bar{f}})} \alpha_S \delta_S(x,y)
\end{equation}
where $\mathcal{S}(\psi_{\bar{f}})$ is weakly compatible. It follows by Theorem \ref{Thm:thm3} that $(\Gamma([\psi_{\bar{f}}]),d_1)$ is a totally split-decomposable metric space. Moreover, for any metric space $(X,d)$, the \textit{underlying graph} $\mathrm{UG}(X,d)$ of $(X,d)$ is the graph $(X,E)$ where $\{x,y\} \in E$ if and only if $d(x,z) + d(z,y) > d(x,y)$ for any $z \in X \setminus \{x,y\}$. Furthermore, let $C_6$ denote the $6$-cycle metric graph and let $K_{3\times 2}$ denote the complete graph on six vertices with $3$ disjoint edges taken away (i.e., the $1$-skeleton of the octahedron). 

\begin{Rem}\label{Rem:AntSS}
Note that if $\mathcal{S}$ is an antipodal split system on $(X,d)$, then for any $(A_i)_{i \in I}$, if $\bigcup_{i \in I} A_i = X$, it follows that $\bigcap_{i \in I} A_i = \emptyset$. Indeed, if $x \in \bigcap_{i \in I} A_i$, there is a subsystem of pairwise different splits $\{S_i\}_{i \in I} \subset \mathcal{S}$ such that $A_i = S_i(x)$. Now, there is $y \in X$ such that $y \in \bigcap_{S \in \mathcal{S}} \overline{S(x)} \subset \bigcap_{i \in I} \overline{S_i(x)} = \bigcap_{i \in I} A^c_i = (\bigcup_{i \in I} A_i)^c$,  which implies that $\bigcup_{i \in I} A_i \neq X$. The octahedral split system is an example of antipodal split system.
\end{Rem}

We conclude this section with:

\begin{proof}[Proof of Theorem \ref{Thm:t11}]
The existence of $\mathrm{K}(X,d)$ and $\sigma$ in Theorem \ref{Thm:t11} follows immediately from Theorem~\ref{Thm:t12} (which is proved in the next section). Indeed, Theorem~\ref{Thm:t12} implies that if we remetrize $\mathrm{E}(X,d)$ by identifying each cell (which is a parallelotope by the first part of the Theorem \ref{Thm:t11}) with a corresponding unit hypercube (of same dimension) endowed with the euclidean metric, and considering the induced length metric, we obtain a complex $\mathrm{K}(X,d)$ which satisfies the CAT(0) link condition. Since $(X,d)$ satisfies the (LRC), it follows that $\mathrm{K}(X,d)$ is complete and locally CAT(0) (analogue to I.7.13~Theorem and II.5.2~Theorem in \cite{BriH}). By the (LRC), it also follows that $\mathrm{K}(X,d)$ is locally bi-Lipschitz equivalent to $\mathrm{E}(X,d)$, the topology induced by the length metric on $\mathrm{K}(X,d)$ is therefore the same as the topology on $\mathrm{E}(X,d)$ and thus $\mathrm{K}(X,d)$ is contractible as well. By Cartan-Hadamard Theorem, it follows that $\mathrm{K}(X,d)$ is globally CAT(0).

As an introductory remark, note that by Lemma \ref{Lem:l15}, Lemma \ref{Lem:l199} and \eqref{eq:e199}, it follows that $(G([f]),d_{\infty})$ is an antipodal totally split-decomposable metric space. By Theorem~\ref{Thm:AntipodalGateSetsCells}, $([f],d_{\infty})$ is combinatorially equivalent to $\mathrm{E}(G([f]),d_{\infty})$, which is by \cite[Theorem~1.2]{HubKMII} an $n$-dimensional combinatorial hypercube if $|G([f])| = 2n \ge 8$. Moreover, if $|G([f])| \le 4$, then $\mathrm{E}(G([f]),d_{\infty})$ is clearly a combinatorial hypercube as well. Now, assume that $|G([f])| = 6$. Since $(G([f]),d_{\infty})$ is antipodal, it follows by \cite[Corollary~3.3]{HubKMI} that $\mathrm{UG}(G([f]),d_{\infty})$ is either $K_{3\times 2}$ or $C_6$. If $\mathrm{UG}(G([f]),d_{\infty}) = C_6$, then by \cite[Theorem~1.2~(a)]{HubKMII}, $\mathrm{E}(G([f]),d_{\infty})$ is a $3$-dimensional combinatorial hypercube.

Assume now that $\mathrm{E}(G([f]),d_{\infty})$ is a combinatorial rhombic dodecahedron, i.e. \eqref{it:it2} in Theorem~\ref{Thm:t11} does not hold, then $\mathrm{UG}(G([f]),d_{\infty}) = K_{3\times 2}$ and it follows by the proof of \cite[Theorem~5.1, Case~2]{HubKMI} that
\[
d_{\infty}(\kappa(\gamma^x_{[\psi_{\bar{f}}]}),\kappa(\gamma^y_{[\psi_{\bar{f}}]}) )= \sum_{S \in \{S_1,S_2,S_3,S_4\}} \beta_S \delta_S(x,y)
\]
where $\{S_1,S_2,S_3,S_4\}$ is weakly compatible and the coefficients $\beta_S$ are all positive. Moreover, by \eqref{eq:e199} and Lemma \ref{Lem:l199}, we have
\[
d_{\infty}(\kappa(\gamma^x_{[\psi_{\bar{f}}]}),\kappa(\gamma^y_{[\psi_{\bar{f}}]}) )
=
\sum_{S \in \mathcal{S}(\psi_{\bar{f}})} \alpha_S \delta_S(x,y)
\]
where $\mathcal{S}(\psi_{\bar{f}})$ is weakly compatible and consists of $d$-splits of $X$. Note that the metric $d_{\infty}$ on $G([f])$ induces a pseudometric $\bar{d}$ on $X$ by setting 
\[
\bar{d}(x,y) := d_{\infty}(\kappa(\gamma^x_{[\psi_{\bar{f}}]}),\kappa(\gamma^y_{[\psi_{\bar{f}}]}) ).
\]
It follows by Theorem \ref{Thm:thm3} that the split systems $\mathcal{S}(\psi_{\bar{f}})$ and $\{S_1,S_2,S_3,S_4\}$ each consist of all the $\bar{d}$-splits of $X$, which implies that $\mathcal{S}(\psi_{\bar{f}})=\{S_1,S_2,S_3,S_4\}$. Therefore, the split system $\bar{\mathcal{S}}:=\{S_1,S_2,S_3,S_4\}$ consists of $d$-splits of $X$ and thus it is an octahedral split subsystem of $\mathcal{S}$. We can write the splits in $\bar{S}$ as
\begin{align}\label{eq:OctSS}
S_1 &:= \{ Y^1_1 \sqcup Y^1_2 \sqcup  Y^1_3 , Y^{-1}_1 \sqcup Y^{-1}_2 \sqcup Y^{-1}_3 \}, \nonumber \\
S_2 &:= \{ Y^1_1 \sqcup Y^1_2 \sqcup Y^{-1}_3 , Y^{-1}_1 \sqcup Y^{-1}_2 \sqcup Y^1_3 \}, \nonumber \\
S_3 &:= \{ Y^1_1 \sqcup Y^{-1}_2 \sqcup Y^1_3 , Y^{-1}_1 \sqcup Y^1_2 \sqcup Y^{-1}_3 \}, \nonumber \\
S_4 &:= \{ Y^1_1 \sqcup Y^{-1}_2 \sqcup Y^{-1}_3 , Y^{-1}_1 \sqcup Y^1_2 \sqcup Y^1_3 \},
\end{align}
where $X = Y^1_1 \sqcup Y^{-1}_1 \sqcup Y^1_2 \sqcup Y^{-1}_2 \sqcup Y^1_3 \sqcup Y^{-1}_3$ and all sets being non-empty. We know from the above that since we have assumed that $[f]$ is a combinatorial rhombic dodecahedron, then $[f]$ must be a maximal cell (in dimensions higher than three, a cell must be a hypercube and the same holds for all of its faces). Since our assumptions imply $\mathrm{E}(d) = \mathrm{E}'(d)$ and $[f]$ is a finite dimensional maximal cell, it follows by \cite{LanP} that the graph $(X,A(f))$ consists of three complete bipartite connected components that are given (cf. the proof of Lemma \ref{Lem:l199}) by their respective partitions, namely $Y^1_1 \sqcup Y^{-1}_1$, $Y^1_2 \sqcup Y^{-1}_2$ and $Y^1_3 \sqcup Y^{-1}_3$ (this is the only possibility since if $\{x,y\} \in A(f)$ then $S(x) \neq S(y)$ for every $S \in \bar{\mathcal{S}}$). For each $S:=\{A,B\} \in \mathcal{S} \setminus \bar{\mathcal{S}}$, there is $\{x,y\} \in A(f)$ such that $S(x) = S(y)$ by bipartiteness, let us say $\{x,y\} \subset A$. It follows from \eqref{eq:e1} that $\psi:=\psi_{\bar{f}} \in \bar{T}(\mathcal{S},\alpha)$ (where $\psi_{\bar{f}}$ is as defined in Lemma \ref{Lem:l17}, in particular $[\bar{f}]=[f]$ where $\kappa(\psi)=\bar{f}$) satisfies then $\psi(A) = 0$. Hence, $\psi(B) = \alpha_S/2$ and thus for every further $\{x',y'\} \in A(f)$, one has $\{x',y'\} \not \subset B$. This implies by bipartite completeness of $Y^1_i \sqcup Y^{-1}_i$ that there are $\sigma,\tau,\theta \in \{ \pm 1 \}$ such that $Y^{\sigma}_1 \cup Y^{\tau}_2 \cup Y^{\theta}_3 \subsetneq A$ which is equivalent to $\{A,B\}$ and $\{A',B'\} = \{ Y^{\sigma}_1 \cup Y^{\tau}_2 \cup Y^{\theta}_3, Y^{-\sigma}_1 \cup Y^{-\tau}_2 \cup Y^{-\theta}_3 \} \in \bar{\mathcal{S}}$ being compatible (i.e., $A' \subset A$).
It follows that \eqref{it:it1} in Theorem~\ref{Thm:t11} does not hold.

Conversely, assume that \eqref{it:it1} does not hold. Define $\psi \in H(\mathcal{S},\alpha)$ so that $\mathcal{S}(\psi) = \bar{\mathcal{S}}$ and $\bar{\mathcal{S}}$ consists of four splits as given in \eqref{eq:OctSS} and is a converse to \eqref{it:it1}. For any $S:=\{A,B\} \in \mathcal{S} \setminus \bar{\mathcal{S}}$, one has $\psi(A) = 0$ for $Y^{\sigma}_1 \cup Y^{\tau}_2 \cup Y^{\theta}_3 \subsetneq A$ and accordingly $\psi(B) = \alpha_S/2$. One has $\psi \in \bar{T}(\mathcal{S},\alpha)$ since for any $(C_i)_{i \in I} \subset \mathrm{supp}(\psi)$, we can consider for each $i \in I$, a corresponding $S'_i=\{A'_i,B'_i\} \in \bar{\mathcal{S}}$ such that $C_i \subset B'_i =:D_i$. It follows that if $\cup_{i \in I} C_i \neq \emptyset$, then $\cup_{i \in I} D_i \neq \emptyset$ and thus by Remark \ref{Rem:AntSS}, it follows that $\cap_{i \in I} D_i = \emptyset$ which implies $\cap_{i \in I} C_i \neq \emptyset$. It is then easy to see that $[\kappa(\psi)]$ is a combinatorial rhombic dodecahedron since for any $(x,y) \in Y^1_i \times Y^{-1}_i$, one has that $S(x) = S(y)$ implies that $S \in \mathcal{S} \setminus \mathcal{S}(\psi)$ and by definition of $\psi$ we have $\psi(S(x))=0=\psi(S(y))$ but since $\psi \in H(\mathcal{S},\alpha)$, we have equality in \eqref{eq:e1}, which implies that $\{x,y\} \in A(\kappa(\psi))$. This means that $(X,A(\kappa(\psi)))$ consists of the three complete bipartite connected components $X = \cup_{i \in \{1,2,3\}} (Y^1_i \sqcup Y^{-1}_i)$ which implies that $\mathrm{dim}([\kappa(\psi)])=3$. Setting $f:=\kappa(\psi) \in \mathrm{E}(d)$, it is easy to see that we have the decomposition $f := \sum_{S \in \mathcal{S}} \alpha_S f_S$ so that for $S:=\{A,B\} \in \mathcal{S}$, one has 
\begin{equation}\label{eq:Dec}
f_S(z) \; = \; \left\lbrace{\begin{matrix} 
\frac{\psi(A)}{\alpha_S/2} &\text{if }  z \in A , \\ \\
\frac{\psi(B)}{\alpha_S/2} &\text{if }  z \in B,
\end{matrix}}\right.
\end{equation}
and $\psi_f=\psi$ holds ($\psi_f$ is defined in the proof of Lemma \ref{Lem:l12}). By \eqref{eq:Cel1} and \eqref{eq:Cel2}, we have (in the notation of the proof of Lemma \ref{Lem:l17}) that $\mathcal{S}_f \subset \mathcal{S}_{[f]}$ and $|\mathcal{S}_{[f]}|= 4$. But \eqref{eq:Dec} shows that $|\mathcal{S}_f|=4$ since $|\mathcal{S}(\psi)|=4$.
It follows that $\mathcal{S}_f = \mathcal{S}_{[f]}$ and thus we can set $\bar{f}:=f$. We then have $\psi_{\bar{f}}=\psi_f=\psi$ and with Lemma \ref{Lem:l199} we obtain that $[f]$  is a combinatorial rhombic dodecahedron and thus \eqref{it:it2} in Theorem~\ref{Thm:t11} does not hold either.
\end{proof}

\section{The CAT(0) Link Condition for the Buneman Complex and the Cubical Injective Hull}
We start by considering $B(\mathcal{S},\alpha)$ which displays some similarities with the CAT(0) cube complex that is constructed in \cite{ChaN} and denoted by $X$.

\begin{Def}\label{Def:CAT0LinkConditionCubeComplex}
A polyhedral complex $\mathrm{K}$, whose cells are combinatorial hypercubes, is said to satisfy the CAT(0) \textit{link condition} if for every set of seven cells $C,C_1^1,C_1^2,C_1^3,C_2^1,C_2^2,C_2^3$ of $\mathrm{K}$, such that the following hold:
\begin{enumerate}[(A)]
\item $C = \bigcap_{i \in \{1,2,3\}} C_2^i$,
\item $C_1^j = \bigcap_{i \in \{1,2,3\} \setminus \{ j \}} C_2^i$,
\item $\mathrm{dim}(C) = k \ge 0$ and for each $i \in \{1,2,3\}$, one has $\mathrm{dim}(C_1^i) = k+1$ as well as $\mathrm{dim}(C_2^i) = k+2$,
\end{enumerate}
there exists a cell $\bar{C}$ of $\mathrm{K}$ such that $\mathrm{dim}(\bar{C}) = k+3$ and $\bigcup_{i \in \{1,2,3\}} C_2^i \subset \bar{C}$. 
\end{Def}

We now have:

\begin{Lem}\label{Lem:l31}
Let $(\mathcal{S},\alpha)$ be a split system pair on a set $X$. Then, the Buneman complex $B(\mathcal{S},\alpha)$ satisfies the CAT(0) link condition.
\end{Lem}
\begin{proof}
Let 
\[
[\mu],[\mu_1^1],[\mu_1^2],[\mu_1^3],[\mu_2^1],[\mu_2^2],[\mu_2^3] \subset B(\mathcal{S},\alpha)
\]
be cells of $B(\mathcal{S},\alpha)$ such that the following hold:
\begin{enumerate}[(i)]
\item $\mathrm{supp}(\mu) = \bigcap_{i \in \{1,2,3\}} \mathrm{supp}(\mu_2^i)$,
\item $\mathrm{supp}(\mu_1^j) = \bigcap_{i \in \{1,2,3\} \setminus \{ j \}} \mathrm{supp}(\mu_2^i)$,
\item $\mathrm{dim}([\mu]) = |\mathcal{S}(\mu)| = k \ge 0$ and for each $j \in \{1,2,3\}$, one has
\[
\mathrm{dim}([\mu_1^j]) = |\mathcal{S}(\mu_1^j)| = k+1,
\]
as well as 
\[
\mathrm{dim}([\mu_2^j])=|\mathcal{S}(\mu_2^j)| = k+2.
\]
\end{enumerate}
This implies that there are splits $\{S_1,S_2,S_3\} \subset \mathcal{S} \setminus \mathcal{S}(\mu)$ such that
\begin{enumerate}[(i)]
\item for $j \in \{1,2,3\}$, one has $\mathrm{supp}(\mu_1^j) = \mathrm{supp}(\mu) \cup S_j$ and
\item for $i \in \{1,2,3\}$, $\mathrm{supp}(\mu_2^i) = \mathrm{supp}(\mu) \cup \bigcup_{j \in \{1,2,3\} \setminus \{i\} } S_j$.
\end{enumerate}
But now, pick $\psi \in H(\mathcal{S},\alpha)$ such that
\[
\mathrm{supp}(\psi) = \mathrm{supp}(\mu) \cup \bigcup_{j \in \{1,2,3\} } S_j = \bigcup_{j \in \{1,2,3\} } \mathrm{supp}(\mu_2^j).
\]
It is then very easy to check that $\psi \in B(\mathcal{S},\alpha)$ and thus $[\psi]$ is a cell of $B(\mathcal{S},\alpha)$. Moreover, one has by definition: 
\begin{enumerate}[(i)]
\item $\bigcup_{i \in \{1,2,3\} } [\mu_2^i] \subset [\psi]$ and
\item $\mathrm{dim}([\psi]) = |\mathcal{S}(\psi)|=k+3$.
\end{enumerate}
This finishes the proof of the CAT(0) link condition for $B(\mathcal{S},\alpha)$. 
\end{proof}
Recall that a split system $\mathcal{S}$ is called antipodal if for every $x \in X$, there is $y \in X$ such that for every $S \in \mathcal{S}$, one has
\[
S(x) \neq S(y).
\]
As a preliminary to the proof of Theorem \ref{Thm:t12}, we have the following:
\begin{Lem}\label{Lem:l22}
Let $\mathcal{S}$ be a split system on a set $X$. Then,
\begin{enumerate}
\item Assume that $\mathcal{S}$ is a weakly compatible split system and assume that for all $i \in \{1,2,3\}$, the split system $\mathcal{S}(\mu_2^i) = \mathcal{S}(\mu) \cup [ \{ S_1,S_2,S_3 \} \setminus \{ S_i\} ]$ is antipodal. Then, $\mathcal{S}(\psi) := \mathcal{S}(\mu) \cup \{ S_1,S_2,S_3\}$ is antipodal as well. 

\item Let $(X,d)$ be a totally split-decomposable metric space (hence in particular, $\mathcal{S}$ is weakly compatible). Let $\{f_2^i\}_{i \in \{1,2,3\}} \subset \mathrm{E}'(d)$ be such that for the split systems given in (1), one has: $\mathcal{S}(\mu_2^i) = \mathcal{S}(\psi_{f_2^i})$. Then, for every $x \in X$, one can find $y \in X$ so that the following hold:
\begin{enumerate}
\item For some $i \in \{1,2,3\}$, one has $\{x,y\} \in A(f_2^i)$.
\item For every $S \in \mathcal{S}(\psi)$, one has $S(x) \neq S(y)$.
\end{enumerate}
\end{enumerate}
\end{Lem}
\begin{proof}
Let $x \in X$ be arbitrarily chosen. Since for every $i \in \{1,2,3\}$, $\mathcal{S}(\mu_2^i)$ is antipodal, there is $y^{i}_2 \in X$ such that for every $S \in \mathcal{S}(\mu_2^i)$, one has
\[
S(x) \neq S(y_2^i).
\]
Note now that if $S_{i}(x) \neq S_i(y_2^i)$, then for every $S \in \mathcal{S}(\psi)$, one has
\[
S(x) \neq S(y_2^i).
\]
Moreover, if $y_2^i = y_2^j$ with $i \neq j$, then
\[
S_i(x) \neq S_i(y_2^j) = S_i(y_2^i)
\]
and hence as above it follows that for every $S \in \mathcal{S}(\psi)$, one has
\[
S(x) \neq S(y_2^i).
\]
If we now assume that there are pairwise different points $\{x,y_2^1,y_2^2,y_2^3\}$ such that
\[
S_i(x) = S_i(y_2^j) \Longleftrightarrow i = j,
\]
then it follows that the points $\{x,y_2^1,y_2^2,y_2^3\}$ and the splits $\{S_1,S_2,S_3\}$ contradict the weak compatibility of $\mathcal{S}$. This proves the first assertion. 

The second assertion follows from the fact that by the last part of the statement of Lemma \ref{Lem:l14}, if for each $i \in \{1,2,3\}$, we pick $y_2^i \in X$ such that $\{x,y_2^i\} \in A(f_2^i)$, then $S(x) \neq S(y_2^i)$ for every $S \in \mathcal{S}(\mu_2^i)$, and thus by the above proof, we deduce that for some $i \in \{1,2,3\}$, one has $S(x) \neq S(y_2^i)$ for every $S \in \mathcal{S}(\psi)$ which implies that if we set $y:=y_2^i$, the second assertion follows.
\end{proof}

We now have:

\begin{Thm}\label{Thm:t12}
Let $(X,d)$ be a metric space with integer-valued totally split-decomposable metric satisfying the (LRC) and such that each cell of $\mathrm{E}(X,d)$ is a combinatorial hypercube. Then, $\mathrm{E}(X,d)$ satisfies the CAT(0) link condition.
\end{Thm}

\begin{proof}[Proof of Theorem~\ref{Thm:t12}]
Assume that there are cells
\[
[f],[f_1^1],[f_1^2],[f_1^3],[f_2^1],[f_2^2],[f_2^3] \subset \mathrm{E}'(d)
\]
such that
\begin{enumerate}[$(i)$]
\item $[f] = \bigcap_{i \in \{1,2,3\}} \mathrm{supp}[f_2^i]$,
\item $[f_1^j] = \bigcap_{i \in \{1,2,3\} \setminus \{ j \}} \mathrm{supp}[f_2^i]$,
\item $\mathrm{dim}([f]) = k \ge 0$ and for each $i \in \{1,2,3\}$, one has $\mathrm{dim}([f_1^i]) = k+1$ as well as $\mathrm{dim}([f_2^i]) = k+2$.
\end{enumerate}
Since all cells of $\mathrm{E}'(d)$ are hypercubes, it is easy to see (using Remark \ref{Rem:cellsdimension} to see that $\mathrm{dim}([\psi_f]) =|\mathcal{S}(\psi_f)|= k$ as well as for the other similar equalities in $(iii)$) that one has
\begin{enumerate}[$(i)$]
\item $[\psi_f] = \bigcap_{i \in \{1,2,3\}} \mathrm{supp}[\psi_{f_2^i}]$,
\item $[\psi_{f_1^j}] = \bigcap_{i \in \{1,2,3\} \setminus \{ j \}} \mathrm{supp}[\psi_{f_2^i}]$,
\item $\mathrm{dim}([\psi_f]) =|\mathcal{S}(\psi_f)|= k$ and for each $j \in \{1,2,3\}$, one has $\mathrm{dim}([\psi_{f_1^j}]) =|\mathcal{S}(\psi_{f_1^j})|= k+1$ as well as $\mathrm{dim}([\psi_{f_2^i}]) = |\mathcal{S}(\psi_{f_2^i})| = k+2$.
\end{enumerate}
By Lemma \ref{Lem:l31}, there is $\psi \in B(\mathcal{S},\alpha)$ such that
\begin{enumerate}[$(a)$]
\item $\mathrm{supp}(\psi) = \mathrm{supp}(\psi_f) \cup \bigcup_{i \in \{1,2,3\}} S_i$ and
\item $\mathcal{S}(\psi) = \mathcal{S}(\psi_f) \cup \{ S_1,S_2,S_3\}$.
\end{enumerate}
Let $x \in X$ be chosen arbitrarily, by Lemma \ref{Lem:l22} there is $y \in X$ such that for every $S \in \mathcal{S}(\psi)$, one has $S(x) \neq S(y)$ and without loss of generality $\{x,y\} \in A(f_2^1)$. For the sake of simplicity, we set $g := f_2^1$. It follows from \eqref{eq:1116} that 
\[
d_1(\phi_x,\phi_y) = d_1(\phi_x, \gamma^x_{[\psi_g]}) + d_1(\gamma^x_{[\psi_g]}, \gamma^y_{[\psi_g]}) + d_1( \gamma^y_{[\psi_g]},\phi_y)
\]
which is easily seen to imply
\begin{equation}\label{eq:e31}
\sum_{S \in \mathcal{S} \setminus \mathcal{S}(\psi_g)} \alpha_S \delta_S(x,y) =
d_1(\phi_x, \gamma^x_{[\psi_g]}) + d_1(\phi_y, \gamma^y_{[\psi_{g}]}).
\end{equation}
On the other hand, starting from the definition of $\phi_x$ and $\gamma^x_{[\psi_g]}$, we have 
\[
d_1(\phi_x, \gamma^x_{[\psi_g]}) = \sum_{\substack{A \in U(\mathcal{S} \setminus \mathcal{S}(\psi_g) ) \\ x \in A} } |0 - \psi_g(A)| +
\sum_{\substack{A \in U(\mathcal{S} \setminus \mathcal{S}(\psi_g) ) \\ x \notin A} } |\alpha_{S_A}/2 - \psi_g(A)|
\]
and analogously
\[
d_1(\phi_y, \gamma^y_{[\psi_g]}) = \sum_{\substack{A \in U(\mathcal{S} \setminus \mathcal{S}(\psi_g) ) \\ y \in A} } |0 - \psi_g(A)| +
\sum_{\substack{A \in U(\mathcal{S} \setminus \mathcal{S}(\psi_g) ) \\ y \notin A} } |\alpha_{S_A}/2 - \psi_g(A)|.
\]
Hence, using the fact that $g \in \mathrm{E}'(d)$ and thus $\psi_g \in H(\mathcal{S},\alpha)$, cf. proof of Lemma \ref{Lem:l12}, we obtain a second expression for the right-hand side of \eqref{eq:e31}, namely
\begin{align}
d_1(\phi_x, \gamma^x_{[\psi_g]}) + d_1(\phi_y, \gamma^y_{[\psi_g]}) = 
&\sum_{\substack{A \in U(\mathcal{S} \setminus \mathcal{S}(\psi_g) ) \\ x,y \in A} } 2 \psi_g(A) \nonumber \\
&+ \sum_{\substack{A \in U(\mathcal{S} \setminus \mathcal{S}(\psi_g) ) \\ x \in A, y \notin A} } \alpha_{S_A}/2 \nonumber \\
& + \sum_{\substack{A \in U(\mathcal{S} \setminus \mathcal{S}(\psi_g) ) \\ x \notin A, y \in A} } \alpha_{S_A}/2 \nonumber \\
&+ \sum_{\substack{A \in U(\mathcal{S} \setminus \mathcal{S}(\psi_g) ) \\ x,y \notin A} } 2 |\alpha_{S_A}/2 - \psi_g(A)|. \label{eq:e32}
\end{align}
Since the sum of the second and third term of the right-hand side of \eqref{eq:e32} amounts to $\sum_{S \in \mathcal{S} \setminus \mathcal{S}(\psi_g)} \alpha_S \delta_S(x,y)$, comparing \eqref{eq:e31} and \eqref{eq:e32} we obtain
\[
\sum_{\substack{A \in U(\mathcal{S} \setminus \mathcal{S}(\psi_g) ) \\ x,y \in A} } 2 \psi_g(A) +
\sum_{\substack{A \in U(\mathcal{S} \setminus \mathcal{S}(\psi_g) ) \\ x,y \notin A} } 2 |\alpha_{S_A}/2 - \psi_g(A)|
=0
\]
hence for every $A \in U(\mathcal{S} \setminus \mathcal{S}(\psi_g) )$ such that $x,y \in A$, one has $\psi_g(A) = 0$ or in other words, for any $S \in \mathcal{S} \setminus \mathcal{S}(\psi_g)$ such that $S(x) = S(y)$, one has $\psi_g(S(x)) = 0 = \psi_g(S(y))$. Now, note that by definition of $\psi$, one has $\mathrm{supp}(\psi_g)  \subset \mathrm{supp}(\psi)$ as well as $\mathcal{S} \setminus \mathcal{S}(\psi) \subset \mathcal{S} \setminus \mathcal{S}(\psi_g)$ and thus for every $\mathcal{S} \setminus \mathcal{S}(\psi)$ such that $S(x) = S(y)$, one has $\psi(S(x))=0$. It is easy to see that this implies that for every $S \in \mathcal{S} \setminus \mathcal{S}(\psi)$, since $\psi \in H(\mathcal{S},\alpha)$, one has
\begin{equation}\label{eq:e311}
\alpha_S \delta_S(x,y) = 2 \bigl[ \psi(S(x)) + |\phi_y(S(x)) - \psi(S(x)) |  \bigr].
\end{equation}
Moreover, one easily obtains
\begin{align}
&\sum_{S \in \mathcal{S} \setminus \mathcal{S}(\psi)} 2 \bigl[ \psi(S(x)) + |\phi_y(S(x)) - \psi(S(x)) | \bigr]  \nonumber \\
&=\sum_{S \in \mathcal{S} \setminus \mathcal{S}(\psi)} 
\begin{aligned}[t]
&\bigl[ |\phi_x(S(x)) -  \psi(S(x)) | + |\phi_y(S(x)) - \psi(S(x)) |  \nonumber  \\
&+  |\phi_x(\overline{S(x)}) -  \psi(\overline{S(x)}) | + |\phi_y(\overline{S(x)}) - \psi(\overline{S(x)}) | \bigr] 
\end{aligned}\\
&= \sum_{A \in U(\mathcal{S} \setminus \mathcal{S}(\psi))} \bigl[ |\phi_x(A) -  \psi(A) | + |\phi_y(A) - \psi(A) | \bigr].\label{eq:e3111}
\end{align}
Furthermore, $y$ was chosen so that for every $S \in \mathcal{S}(\psi)$, one has $S(x) \neq S(y)$. Thus for every $\bar{\psi} \in [\psi]$, one obtains:
\begin{align}
&\alpha_S\delta_S(x,y) \nonumber \\
&= \alpha_S \nonumber \\
&= [ 0 + \bar{\psi}(S(x))] + [\alpha_S/2 - \bar{\psi}(S(x))] + [\alpha_S/2 - \bar{\psi}(\overline{S(x)})] + [ \bar{\psi}(\overline{S(x)}) - 0] \nonumber \\
&=\sum_{A \in S} \left[ |\phi_x(A) -  \bar{\psi}(A) | + |\phi_y(A) - \bar{\psi}(A) | \right] \label{eq:e322}
\end{align}
and thus since $d_1(\gamma^x_{[\psi]},\gamma^y_{[\psi]}) = \sum_{S \in \mathcal{S}(\psi)} \alpha_S\delta_S(x,y)$ together with \eqref{eq:e322} and since for every $A \in U(\mathcal{S} \setminus \mathcal{S}(\psi))$, one has $\bar{\psi}(A) = \psi(A)$, it follows that
\begin{align}
d_1(\gamma^x_{[\psi]},\gamma^y_{[\psi]})
&= \sum_{A \in U(\mathcal{S}(\psi))} \left[ |\phi_x(A) -  \bar{\psi}(A) | + |\phi_y(A) - \bar{\psi}(A) | \right] \nonumber \\
&= d_1(\gamma^x_{[\psi]}, \bar{\psi} )+d_1(\bar{\psi}, \gamma^y_{[\psi]}).\label{eq:e3222}
\end{align}
Hence
\begin{align}
d_1(\phi_x,\phi_y) 
&= \sum_{S \in \mathcal{S}} \alpha_S\delta_S(x,y) \nonumber \\
&= \sum_{S \in \mathcal{S}(\psi)} \alpha_S\delta_S(x,y) + \sum_{S \in \mathcal{S} \setminus \mathcal{S}(\psi)} \alpha_S\delta_S(x,y) \nonumber \\
&= d_1( \gamma^x_{[\psi]},\gamma^y_{[\psi]}) + 
\sum_{S \in \mathcal{S} \setminus \mathcal{S}(\psi)} 2 \bigl[ \psi(S(x)) + |\phi_y(S(x)) - \psi(S(x)) | \bigr] \nonumber \\
&= d_1( \gamma^x_{[\psi]},\gamma^y_{[\psi]}) + 
\sum_{A \in U(\mathcal{S} \setminus \mathcal{S}(\psi))} \bigl[ |\phi_x(A) -  \psi(A) | + |\phi_y(A) - \psi(A) | \bigr] \nonumber \\
&=d_1(\phi_x, \gamma^x_{[\psi]}) + d_1(\gamma^x_{[\psi]},\gamma^y_{[\psi]}) +  d_1(\gamma^y_{[\psi]},\phi_y).\label{eq:e32222}
\end{align}
where the third equality follows from \eqref{eq:e311}, the fourth one from \eqref{eq:e3111} and the last one by our definitions. Hence, inserting \eqref{eq:e3222} into \eqref{eq:e32222}, one has:
\begin{align}
d_1(\phi_x,\phi_y) = d_1(\phi_x, \gamma^x_{[\psi]}) + d_1(\gamma^x_{[\psi]}, \psi) + d_1(\psi,\gamma^y_{[\psi]}) +  d_1(\gamma^y_{[\psi]},\phi_y). \label{eq:e333}
\end{align}
It follows from $d(x,y) = d_1(\phi_x,\phi_y)$ and \eqref{eq:e333} that
\begin{align}
d(x,y) = 
d_1(\phi_x, \psi) 
+ d_1(\psi,\phi_y)
=\kappa(\psi)(x) + \kappa(\psi)(y).
\end{align}
Since for any $x \in X$, there is such an $y \in X$, it follows that $\kappa(\psi) \in \mathrm{E}'(d)$. Moreover, by definition of $\psi$, one has
\[
\bigcup_{i \in \{1,2,3\} } \kappa([\psi_{f_2^i}]) 
= \kappa \left(\bigcup_{i \in \{1,2,3\} } [\psi_{f_2^i}] \right)
\subset \kappa([\psi])
\subset [\kappa(\psi)]
\]
where the last inclusion follows from Lemma \ref{Lem:l16}. Now, since $[\kappa(\psi)]$ is a hypercube, this proves that $\mathrm{E}(d) = \mathrm{E}'(d)$ satisfies the CAT(0) link condition.
\end{proof}

Two splits $S:=\{A,B\}$ and $S'=\{A',B'\}$ of $X$ are called \textit{incompatible} if 
\[
A \cap A', A \cap B', B \cap A', B\cap B' \neq \emptyset.
\]
A split system $\mathcal{S}$ is called \textit{incompatible} if any pair of splits in $\mathcal{S}$ is incompatible.

\begin{Rem}
For any split system pair $(\mathcal{S},\alpha)$ on a set $X$, the associated Buneman complex $B(\mathcal{S},\alpha)$ displays some similarities with the CAT(0) cube complex that is constructed in \cite{ChaN}. There, a split system pair is obtained by considering a wall space $\mathcal{W}$ (corresponding to $\mathcal{S}$) on a set $Y$ (corresponding to $X$) and taking the function $\alpha$ chosen to be constantly equal to one on $\mathcal{S}$.

For a Coxeter group, its Cayley graph is endowed with the standard word metric and there is a canonical decomposition (in general not weakly compatible, e.g. $\Delta(3,3,3)$) of this metric given by the splits of the form $S:=\{C(x,y),C(y,x)\}$ where $\{x,y\}$ is an edge in the Cayley graph and with $\alpha$ to be constantly equal to one.

The set $\mathrm{K}^0$ (denoted by $X^0$ in \cite{ChaN}) is defined to be consisting of all the \textit{admissible sections}, i.e. the maps $\sigma \colon \mathcal{S} \to U(\mathcal{S})$ such that for any $S \neq S'$: $\sigma(S) \cap  \sigma(S') \neq \emptyset$. Next, $\mathrm{K}^1$ is the graph with vertex set $\mathrm{K}^0$, where two vertices $\sigma$ and $\sigma'$ are connected by an edge if and only if there is a unique $S \in \mathcal{S}$ such that $\sigma(S) \neq \sigma'(S)$.  For an arbitrarily fixed point $p \in X$, one then lets $\Gamma_p$ be the path-connected component of $\sigma_p$ in $\mathrm{K}^1$ where for any $S$, one lets $\sigma_p(S):=S(p)$.

Let us define $B^p(\mathcal{S},\alpha):=\{ \psi \in B(\mathcal{S},\alpha): d_1(\psi,\phi_p) < \infty\}$, let $\Sigma^0(\Gamma_p)$ be the $0$-skeleton of $\Gamma_p$, and let $\Sigma^0(B^p(\mathcal{S},\alpha))$ and $\Sigma^1(B^p(\mathcal{S},\alpha))$ be the $0$- and $1$-skeleton of $B^p(\mathcal{S},\alpha)$. We define
\[
M \colon (\Sigma^0(\Gamma_p),d_1) \to (\Sigma^0(B^p(\mathcal{S},\alpha)),d_1)
\]
by sending every admissible section $\sigma \colon \mathcal{S} \to U(\mathcal{S})$  to a function $M(\sigma) \colon U(\mathcal{S}) \to \R$ defined by assigning (recall that $S_A:=\{A,A^c\}$):
\[
A  \; \mapsto \; \left\lbrace{\begin{matrix} 
\frac{\alpha_{S_A}}{2} &\text{if }  A = \sigma(S)^c , \\ \\
0 &\text{if }  A = \sigma(S).
\end{matrix}}\right.
\]
Now assume that $\{\sigma,\tau\}$ is an edge of $\Gamma_p$ which means that there is a unique $S':=\{A',B'\} \in \mathcal{S}$ such that $\sigma(S') \neq \tau(S')$. This is equivalent to the fact that $M(\sigma)(A') = M(\tau)(B')$, $M(\sigma)(B') = M(\tau)(A')$ and for any $A \in U(\mathcal{S} \setminus \{ S'\})$, one has $M(\sigma)(A)=M(\tau)(A)$. It is easy to see, that this is in turn equivalent to the fact that there is a function $\psi \in \Sigma^1(B^p(\mathcal{S},\alpha))$ such that $\mathrm{dim}([\psi])=1$, $S' \subset \mathrm{supp}(\psi)$ and so that for every $A \in U(\mathcal{S} \setminus \{ S'\})$, one has $M(\sigma)(A)=\psi(A) = M(\tau)(A)$. Therefore, $M$ extends bijectively to an isometric isomorphism of cell complexes 
\[
M_1 \colon (\Gamma_p,d_1) \to (\Sigma^1(B^p(\mathcal{S},\alpha)),d_1).
\]
Let us denote an edge $e_j$ of $\Gamma_p$ by its corresponding labeling split $S_j$ (the unique one on which the endpoints of $e_j$ differ). Now, \textit{$k$-corners} $(\sigma,\{e_1,\dots,e_k\})$ (in the terminology of \cite{ChaN}) are simply pairs of the form $(\sigma,\{S_1,\dots,S_k\})$ where $\sigma \in \Gamma_p$ and where the split system $\{S_1,\dots,S_k\}$ is incompatible. The complex $\mathrm{K}$ is then obtained by gluing a $k$-cube to every $k$-corner (one shows that the existence of a $k$-corner implies the existence of the $1$-skeleton of a $k$-hypercube contained in $\Gamma_p$ and containing this $k$-corner as a vertex). It is now easy to see that there is an isomorphism of cell complexes 
\[
i \colon \mathrm{K} \to B^p(\mathcal{S},\alpha)
\]
which extends $M_1$. To any $k$-dimensional cube $C$ giving rise to a cell of $\mathrm{K}$ corresponds by construction a $k$-corner $(\sigma,\{S_1,\dots,S_k\})$. It is then easy to see that defining $\psi \in H(\mathcal{S},\alpha)$ so that $\psi(A):=\sigma(A)$ for $A \in U(\mathcal{S} \setminus \{S_1,\dots,S_k\})$ and $\psi(A) := \alpha_{S_A}/4$ otherwise, we obtain that $\psi \in B^p(\mathcal{S},\alpha)$ and $[\psi]$ is a $k$-dimensional cell of $B^p(\mathcal{S},\alpha)$, hence a $k$-dimensional combinatorial hypercube. We can thus extend $M_1$ and map bijectively $C$ to $[\psi]$ with $i$. Conversely, let $\sigma \in \Sigma^0(B^p(\mathcal{S},\alpha))$ be a vertex of a $k$-dimensional cell $[\psi]$ of $B^p(\mathcal{S},\alpha)$. The pair $(\sigma,\mathcal{S}(\psi))$ has to be a $k$-corner in $\Gamma_p$ by incompatiblity of $\mathcal{S}(\psi)$ and thus the inverse image of $[\psi]$ under $i$ is the $k$-dimensional cube in $\mathrm{K}$ glued to the $k$-corner $(\sigma,\mathcal{S}(\psi))$.
\end{Rem}

\section{Examples}

\begin{Expl}\label{Expl:example0}
Let $(X,d)$ be an infinite connected graph endowed with the shortest-path metric. It is easy to see that if $X$ is bipartite, then the system $\mathcal{S}$ of all $d$-splits of $(X,d)$ is octahedral-free. Indeed, assume on the contrary that $X = Y^1_1 \sqcup Y^{-1}_1 \sqcup Y^1_2 \sqcup Y^{-1}_2 \sqcup Y^1_3 \sqcup Y^{-1}_3$ is a partition of $X$ into six non-empty subsets such that 
\begin{align}
S_1 &:= \{ Y^1_1 \sqcup Y^1_2 \sqcup  Y^1_3 , Y^{-1}_1 \sqcup Y^{-1}_2 \sqcup Y^{-1}_3 \}, \nonumber \\
S_2 &:= \{ Y^1_1 \sqcup Y^1_2 \sqcup Y^{-1}_3 , Y^{-1}_1 \sqcup Y^{-1}_2 \sqcup Y^1_3 \}, \nonumber\\
S_3 &:= \{ Y^1_1 \sqcup Y^{-1}_2 \sqcup Y^1_3 , Y^{-1}_1 \sqcup Y^1_2 \sqcup Y^{-1}_3 \}, \nonumber\\
S_4 &:= \{ Y^1_1 \sqcup Y^{-1}_2 \sqcup Y^{-1}_3 , Y^{-1}_1 \sqcup Y^1_2 \sqcup Y^1_3 \}, \label{eq:octahedralsystem}
\end{align}
and $\{S_1,\dots,S_4\} \subset \mathcal{S}$. Since $X$ is a connected graph, there must be an edge $\{x,y\}$ between two non-antipodal sets in the partition, for instance an edge joining $x \in Y^1_1$ to $y \in Y^{-1}_1$. Now, we see that both $S_3(x) \neq S_3(y)$ and $S_4(x) \neq S_4(y)$. This is a contradiction to the fact that since $X$ is bipartite, there is for any edge $\{x,y\}$ in $X$, at most one $d$-split separating $x$ and $y$ namely the split given by $\{C(x,y),C(y,x)\}$.
\end{Expl}

\begin{Expl}\label{Expl:example1}
For $n \in \mathbb{N}$, let
\[
C_{2n+1}:=\bigl(\{x_1,\dots,x_{2n+1}\},\{\{x_i,x_{i+1}\}\}_{i \in \{1,\dots,2n+1\}} \bigr) \ \text{ where } \ x_{2n+2}:=x_1.
\]
The graph $C_{2n+1}$ is the odd cycle with $2n+1$ vertices and we endow it with the shortest path metric $d$. We use the fact that the metric $d$ is totally split-decomposable to give an explicit description of the injective hull $\mathrm{E}(C_{2n+1},d)$. 

One can easily verify that $d = \frac{1}{2} \sum_{S \in \mathcal{S}} \delta_S$ where $\mathcal{S}$ is the set of all $d$-splits of $X$, hence $\alpha \colon \mathcal{S} \to (0,\infty)$ can be chosen to be constantly equal to $\frac{1}{2}$. One has $\mathcal{S} =\{S_1,\dots,S_{2n+1}\}$ where for $i \in \{1,\dots,2n+1\}$, $S_i := \{A_i,B_i\}$ is the unique split of $C_{2n+1}$ such that $x_i \in A_i$, $|A_i|=n+1$, $|B_i|=n$ and such that both $C(x_{i+1},x_i) \subset A_i$ and $C(x_i,x_{i+1}) \subset B_i$. Hence, (with indices taken modulo $2n+1$), one has 
\[
S_i = \{ A_i,B_i \} = \{ \{x_{i+1}, \dots,x_{i+n} \}, \{x_{i+n+1}, \dots,x_i \} \}.
\]
Note moreover that both $S_i$ and $S_{i+n}$ cut the edge $\{x_i,x_{i+1}\}$ for $i \in \{1,\dots,n+1\}$ and both $S_i$ and $S_{i-n}$ cut the edge $\{x_i,x_{i+1}\}$ for $i \in \{n+2,\dots,2n+1\}$.

It is now not difficult to prove that the assumptions of Theorem \ref{Thm:t11} are fulfilled for $(X,d) = (C_{2n+1},d)$. Since we are in the case of a finite metric space, the (LRC) is trivially satisfied. Moreover, $d = \frac{1}{2} \sum_{S \in S} \delta_S$, $(ii)$ where $\mathcal{S}$ is the family of all $d$-splits of $X$. Finally, it is not difficult to see that $\mathcal{S}$ is octahedral-free and thus that \eqref{it:it1} in Theorem \ref{Thm:t11} holds as well. Indeed, if a subsystem $\bar{\mathcal{S}}:=\{\bar{S}_1,\bar{S}_2,\bar{S}_3,\bar{S}_4\} \subset \mathcal{S}$ is octahedral, that is, it is induced by a partition into six non-empty subsets
\[
C_{2n+1} = Y^1_1 \sqcup Y^{-1}_1 \sqcup Y^1_2 \sqcup Y^{-1}_2 \sqcup Y^1_3 \sqcup Y^{-1}_3
\]
as in \eqref{eq:octahedralsystem}, then for any $x \in Y^{\sigma}_i$, for any $y \in Y^{-\sigma}_i$ and for every $S \in \bar{\mathcal{S}}$, one has $S(x) \neq S(y)$.
In order for such a pair $\{x,y\}$ to exist for every $x \in C_{2n+1}$, it follows that if $\bar{S}_j \in \bar{\mathcal{S}}$ with $1\le j \le n+1$ (the case $n+2\le j \le 2n+1$ is similar), then $\bar{S}_j,\bar{S}_{j+n+1} \notin \bar{\mathcal{S}}$. Indeed, there are exactly two splits that cut one of the edges $\{ \{x_j,x_{j+1} \} , \{x_{j+n},x_{j+n+1} \} \}$ which are cut by $\bar{S}_j$, namely $\bar{S}_{j+n}$ which cuts $\{x_{j+n},x_{j+n+1} \}$ and $\bar{S}_{j+n+1}$ which cuts $\{x_j,x_{j+1} \}$. Since we are considering $C_{2n+1}$, it follows that $\bar{\mathcal{S}}$ must induce a partition into eight non-empty subsets 
\[
C_{2n+1} = Z^1_1 \sqcup Z^{-1}_1 \sqcup Z^1_2 \sqcup Z^{-1}_2 \sqcup Z^1_3 \sqcup Z^{-1}_3 \sqcup Z^1_4 \sqcup Z^{-1}_4
\]
such that $S(x) \neq S(y)$ for every $S \in \bar{\mathcal{S}}$ if and only if $x \in Z^{\sigma}_i$ and $y \in Z^{-\sigma}_i$. It follows that $\bar{\mathcal{S}}$ is not octahedral and thus $\mathcal{S}$ must be octahedral-free, which shows in particular that \eqref{it:it1} in Theorem~\ref{Thm:t11} holds. We deduce that for every $n \in \mathbb{N} \cup \{0 \}$, $\mathrm{E}(C_{2n+1},d)$ is a finite cube complex satisfying the CAT(0) link-condition (and simply connected since it is an injective hull).

We can furthermore describe explicitely the dimension and gluing pattern of the maximal cells of $\mathrm{E}(C_{2n+1},d)$ by studying the different split susbsystems of $\mathcal{S}$. Note first that by finiteness, we have $B(\mathcal{S},\alpha)= \bar{T}(\mathcal{S},\alpha)$ (see \cite{DreHM}) and by Remark~\ref{Rem:inducedcell} the map 
\[
\kappa \colon \bar{T}(\mathcal{S},\alpha) \to \mathrm{E}(C_{2n+1},d)
\]
is in particular an isomorphism of cell complexes. The family of maximal cells of $\mathrm{E}(C_{2n+1},d)$ is thus in bijection with the family of  maximal cells of $B(C_{2n+1},\mathcal{S})$ which in turn bijectively corresponds with the family $\mathfrak{M}$ of maximal incompatible split subsystems $\mathcal{M} \subset \mathcal{S}$. 

Observe that for any such $\mathcal{M} \in \mathfrak{M}$, we can consider a corresponding element $\psi \in \bar{T}(\mathcal{S},\alpha)$ such that $\mathcal{S}(\psi)=\mathcal{M}$. We have $\kappa(\psi)(x) + \kappa(\psi)(y)=d(x,y)$ if and only if $S(x) \neq S(y)$ for every $S \in \mathcal{M}$. Since $\mathcal{M}$ is a maximal incompatible split subsystem of $\mathcal{S}$, it follows that for any $S =\{A,B\} \in \mathcal{S} \setminus \mathcal{M}$ where $|A|=k$ and $|B|=k+1$, one has $\psi(A) = \alpha_S/2 = 1/4$ and $\psi(B) =0$.

For the gluing pattern, we have for any two maximal cells $[\psi],[\mu] \subset \bar{T}(\mathcal{S},\alpha)$, $[\psi] \cap [\mu] \neq 0$ if and only if $\psi(A) = \mu(A)$ for every $A \in U(\mathcal{S} \setminus (\mathcal{S}(\psi) \cup \mathcal{S}(\mu)))$ and in this case, $[\psi] \cap [\mu]$ is the set of all functions $\phi \in B(\mathcal{S},\alpha)$ such that $\phi(A)=\psi(A)$ for $A \in U(\mathcal{S}(\psi) \setminus \mathcal{S}(\mu)))$, $\phi(A)=\mu(A)$ for $A \in U(\mathcal{S}(\mu) \setminus \mathcal{S}(\psi)))$ and $\mathcal{S}(\phi) = \mathcal{S}(\psi) \cap \mathcal{S}(\mu)$.

Note that for any $\bar{S}_j \in \bar{\mathcal{S}}$, the only two splits that are not incompatible with $\bar{S}_j$ are the only two splits that cut an edge already cut by $\bar{S}_j$.

To compute $|\mathfrak{M}|$, it is easier to describe the split system $\mathcal{S}$ in a different way, by assigning for $i \in \{1,\dots,n+1\}$, to every edge $\{x_i,x_{i+1}\}$ of $C_{2k+1}$, the pair of splits $S^1_i=\{A^1_i,B^1_i\}$ and $S^{-1}_i=\{A^{-1}_i,B^{-1}_i\}$ which cut the edge $\{x_i,x_{i+1}\}$ and which are determined by the requirements $x_i \in A^1_i$ and $|A^1_i|=k+1$ as well as $x_i \in A^{-1}_i$ and $|A^{-1}_i|=k$. Note that we have $S^1_1=S^{-1}_{n+1}$ and the connection with our previous description is given by $S^1_i=S_i$ and $S^{-1}_i=S_{n+i+1}$.


We divide the family $\mathcal{M}$ of maximal incompatible split subsystems of $\mathcal{S}$ into three subfamilies 
\[
\mathfrak{M}= \mathfrak{M}^a \cup \mathfrak{M}^b \cup \mathfrak{M}^c
\]
so that for $\mathcal{M} := \{S_{j_1}^{\sigma_1},\dots,S_{j_k}^{\sigma_k}\}  \in \mathfrak{M}$ with $1 \le j_1 < j_2 < \dots < j_{k-1} < j_k \le 2n+1$, one has the following three cases:

\pagebreak

\begin{enumerate}[$(a)$]
\item $\mathcal{M} \in \mathfrak{M}^a$ if and only if $S_{j_1}^{\sigma_1}=S^1_1$. In this case, we have $|\mathfrak{M}^a|=\Sigma^a_{(n-1,n-1)} + \Sigma^a_{(n-1,n)}$ where $\Sigma_{(n-1,n-1)}:=\Sigma^a_{(n-1,n-1)}$ stands for the case where $S_{j_k}^{\sigma_k}=S_n^{-1}$ and $\Sigma_{(n-1,n)}:=\Sigma^a_{(n-1,n)}$ for the case where $S_{j_k}^{\sigma_k}=S_n^{1}$. The notation $\Sigma_{(n-1,n-1)}$ refers to the fact that $|B_{j_1}^{\sigma_1} \cap A_{j_k}^{\sigma_k}|=n-1$ and $|A_{j_1}^{\sigma_1} \cap B_{j_k}^{\sigma_k}|=n-1$ (i.e. starting with $x_2$ and going counterclockwise we count $n-1$ points until we hit $S_{j_k}^{\sigma_k}$ and similarly when starting with $x_{n+2}$ and until we hit $S_{j_k}^{\sigma_k}$ we count $n-1$ points as well).

\begin{figure}[h]
\centering
\def\svgwidth{0.8\textwidth}
\input{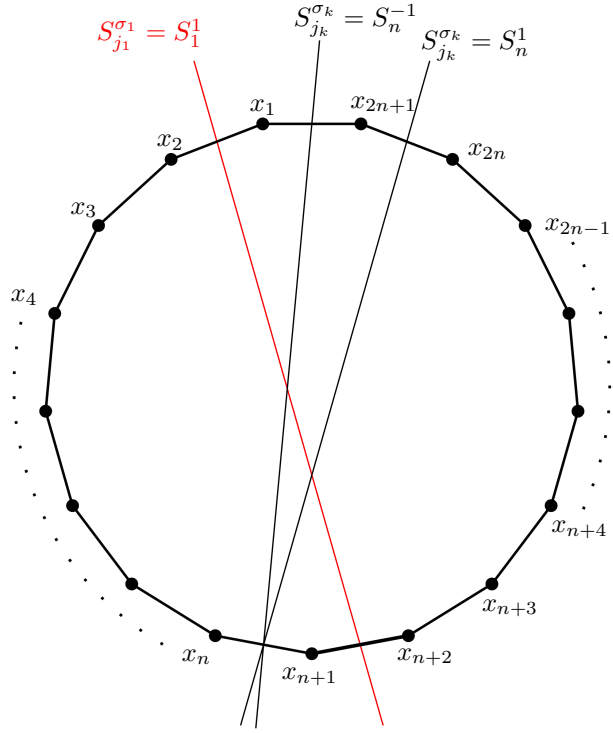}
\caption{$\Sigma^a_{(n-1,n-1)}$ corresponds to the number of maximal incompatible split systems $\{S_{j_1}^{\sigma_1},\dots,S_{j_k}^{\sigma_k}\}$ with $S_{j_1}^{\sigma_1}=S^1_1$ and $S_{j_k}^{\sigma_k}=S_n^{-1}$. 
$\Sigma^a_{(n-1,n)}$ is the number of those with $S_{j_1}^{\sigma_1}=S^1_1$ and $S_{j_k}^{\sigma_k}=S_n^{1}$.}
\label{fig:f1}
\end{figure}

\item $\mathcal{M} \in \mathfrak{M}^b$ if and only if $S_{j_1}^{\sigma_1}=S^{-1}_1$. In this case, we have $|\mathfrak{M}^b|=\Sigma^b_{(n-1,n-1)} + \Sigma^b_{(n-1,n)}$
where by symmetry $\Sigma^b_{(n-1,n-1)}=\Sigma_{(n-1,n-1)}$ stands for the case where $S_{j_k}^{\sigma_k}=S_n^{-1}$ and $\Sigma^b_{(n-1,n)}=\Sigma_{(n-1,n)}$ for the case where $S_{j_k}^{\sigma_k}=S_{n+1}^{1}$.

\begin{figure}[h]
\centering
\def\svgwidth{0.8\textwidth}
\input{mael-injective-hull-odd-cycles-recurrence1-version3-case1b-27-01-15.tex}
\caption{$\Sigma^b_{(n-1,n-1)}$ corresponds to the number of maximal incompatible split systems $\{S_{j_1}^{\sigma_1},\dots,S_{j_k}^{\sigma_k}\}$ with $S_{j_1}^{\sigma_1}=S^{-1}_1$ and $S_{j_k}^{\sigma_k}=S_n^{-1}$. 
$\Sigma^b_{(n-1,n)}$ is the number of those with $S_{j_1}^{\sigma_1}=S^{-1}_1$ and $S_{j_k}^{\sigma_k}=S_{n+1}^1$.}
\label{fig:f2}
\end{figure}

\pagebreak

\item $\mathcal{M} \in \mathfrak{M}^c$ if and only if $S_{j_1}^{\sigma_1}=S^1_2$. In this case, we have $|\mathfrak{M}^c|=\Sigma^c_{(n-1,n-1)}$ where by symmetry $\Sigma^c_{(n-1,n-1)}=\Sigma_{(n-1,n-1)}$ stands for the case where $S_{j_k}^{\sigma_k}=S_{n+1}^1$.
\end{enumerate}

\begin{figure}[h]
\centering
\def\svgwidth{0.8\textwidth}
\input{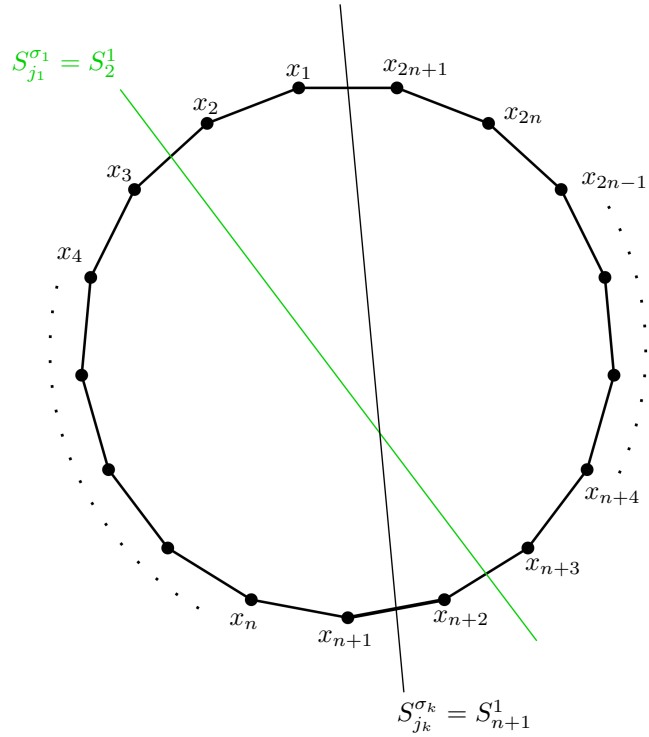}
\caption{$\Sigma^c_{(n-1,n-1)}$ corresponds to the number of maximal incompatible split systems $\{S_{j_1}^{\sigma_1},\dots,S_{j_k}^{\sigma_k}\}$ with $S_{j_1}^{\sigma_1}=S^1_2$ and $S_{j_k}^{\sigma_k}=S_{n+1}^{1}$.}
\label{fig:f3}
\end{figure}



\pagebreak

Summing up, we obtain the formula 
\[
|\mathfrak{M}| = \Theta_n = 3 \Sigma_{(n-1,n-1)}+2 \Sigma_{(n-1,n)}
\]
and it is easy to see that we furthermore have the following recurrence relations: $\Sigma_{(n-1,n-1)} = \Sigma_{(n-2,n-2)} + \Sigma_{(n-3,n-2)}$ and $ \Sigma_{(n-1,n)} = \Sigma_{(n-2,n-1)} + \Sigma_{(n-2,n-2)}$ and the initial conditions $\Sigma_{(0,0)}:=1$, $\Sigma_{(0,1)}:=0$ and $\Sigma_{(1,1)}:=1$. If one considers the three roots $\{\sigma_1,\sigma_2,\sigma_3\}$ of the equation $x^3-x-1=0$ given by
\begin{align*}
\sigma_1 &= \frac{1}{3} \left(\frac{27}{2} - \frac{3 \sqrt{69}}{2} \right)^{\frac{1}{3}} + \frac{\left(\frac{1}{2} \left(9 + \sqrt{69}\right) \right)^{\frac{1}{3}}}{3^{\frac{2}{3}}}, \\
\sigma_2 &= 
-\frac{1}{6} \left(1 + i \sqrt{3}\right) \left( \frac{27}{2} 
- \frac{3 \sqrt{69}}{2}\right)^{\frac{1}{3}}
- \frac{\left(1 - i \sqrt{3}\right)\left(\frac{1}{2} \left(9 + \sqrt{69}\right) \right)^{\frac{1}{3}}}{2 \cdot 3^{\frac{2}{3}}}, \\
\sigma_3 &= 
-\frac{1}{6} \left(1 - i \sqrt{3}\right) \left( \frac{27}{2} 
- \frac{3 \sqrt{69}}{2}\right)^{\frac{1}{3}}
- \frac{\left(1 + i \sqrt{3}\right)\left(\frac{1}{2} \left(9 + \sqrt{69}\right) \right)^{\frac{1}{3}}}{2 \cdot 3^{\frac{2}{3}}} 
\end{align*}
and the three roots $\{\tau_1,\tau_2,\tau_3\}$ of the equation $x^3-2 x^2+x-1=0$ given by
\begin{align*}
\tau_1 &= \frac{1}{3} \left( 2 + \left(\frac{25}{2} - \frac{3 \sqrt{69}}{2} \right)^{\frac{1}{3}} + \left(\frac{1}{2} \left(25 + 3 \sqrt{69}\right) \right)^{\frac{1}{3}} \right), \\
\tau_2 &= \frac{2}{3} 
- \frac{1}{6}\left(1 + i \sqrt{3}\right)\left(\frac{25}{2} - \frac{3 \sqrt{69}}{2} \right)^{\frac{1}{3}}
- \frac{1}{6}\left(1 - i \sqrt{3}\right)\left(\frac{25}{2} + \frac{3 \sqrt{69}}{2} \right)^{\frac{1}{3}}, \\
\tau_3 &= \frac{2}{3} 
- \frac{1}{6}\left(1 - i \sqrt{3}\right)\left(\frac{25}{2} - \frac{3 \sqrt{69}}{2} \right)^{\frac{1}{3}}
- \frac{1}{6}\left(1 + i \sqrt{3}\right)\left(\frac{25}{2} + \frac{3 \sqrt{69}}{2} \right)^{\frac{1}{3}},
\end{align*}
one can verify that
\[
\Theta_n = \sigma_1 \tau^n_1 + \sigma_3 \tau^n_2 + \sigma_2 \tau^n_3.
\]
Note that $\Theta_0$ is one. 
The first values of $\Theta_n$ for $n \ge 1$ are listed in the following table:
\begin{center}
    \begin{tabular}{ | p{1cm} | l | l | l | l | l | l | l | l | l | l |}
    \hline
    $n$  & 1 & 2 & 3 & 4 & 5 & 6 & 7 & 8 & 9 & 10 \\ \hline
    $C_{2n+1}$ &$C_3$ & $C_5$ & $C_7$ & $C_9$ & $C_{11}$ & $C_{13}$ & $C_{15}$ & $C_{17}$ & $C_{19}$ & $C_{21}$ \\ \hline
    $\Theta_n$ & 3 & 5 & 7 & 12 & 22& 39& 68& 119 & 209 & 367 \\ \hline
    \end{tabular}
\end{center}
The numbers $\Theta_n$ can also be proved to correspond to the coefficient of $z^n$ in the power series expansion in a neighborhood of the origin of the analytic function 
\[
z \mapsto \frac{3z-z^2}{1-2z+z^2-z^3}.
\]

Finally, it is not difficult to compute the number of $k$-dimensional cells of $\mathrm{E}(C_{2n+1})$ as a function of $k$ and $n$ as it is done in \cite{Su}. Indeed, note that for a $0$-cell $\psi \in \bar{T}(\mathcal{S},\alpha)$, it is easy to see that $\psi$ is in the maximal cell $[\mu]$ if and only if for any split $S=\{A,B\} \in \mathcal{S}$ where $|A|=n+1$ and $A \in \mathrm{supp}(\psi)$, one has $S \in \mathcal{S}(\mu)$. More generally, for every $k \in \{0,\dots, n\}$, there is a bijection between the $k$-dimensional cells of $\mathrm{E}(C_{2n+1})$ and the set of pairs $(\bar{\mathcal{S}},\bar{\mathcal{S}}')$ where $\bar{\mathcal{S}}$ is a split subsystem of $\mathcal{S}$ with $k$ elements and $\bar{\mathcal{S}} \subset \bar{\mathcal{S}}'$. The correspondence between $[\psi]$ and $(\bar{\mathcal{S}},\bar{\mathcal{S}}')$ is given by picking for every $S=\{A,B\} \in \mathcal{S}$ where $|A|=n+1$, the function $\psi$ in such a way that $\mathcal{S}(\psi):=\bar{\mathcal{S}}$, $\psi(B) = 0$ if $S \in \bar{\mathcal{S}}' \setminus \bar{\mathcal{S}}$ and $\psi(A) = 0$ if $S \in \mathcal{S} \setminus \bar{\mathcal{S}}'$.

\begin{figure}[h]
\centering
\def\svgwidth{0.8\textwidth}
\input{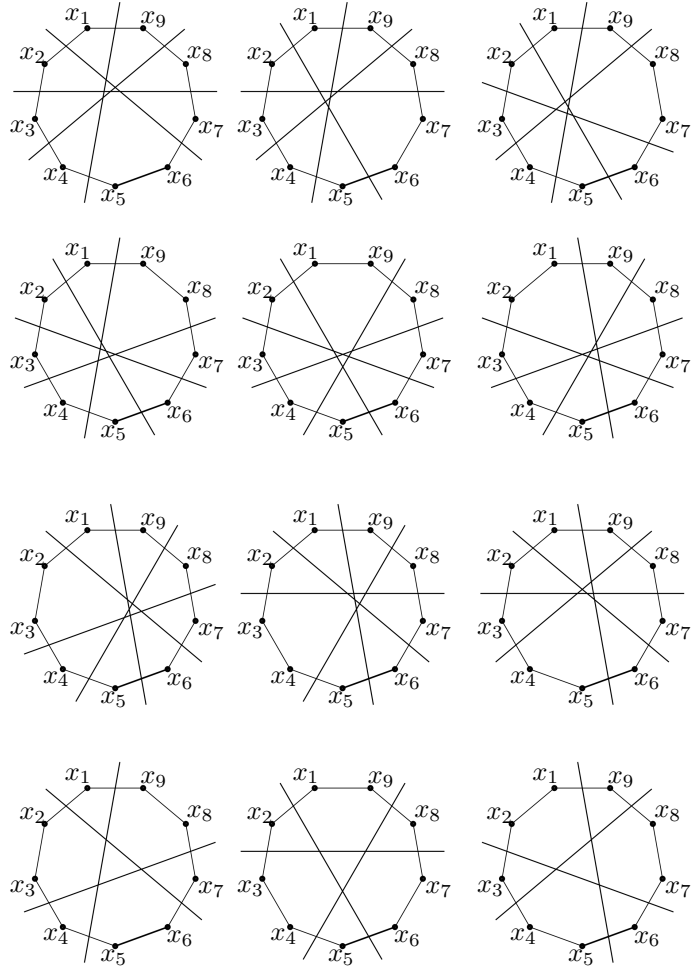}
\caption{List of all $12$ maximal incompatible split subsystems for $C_{2n+1}=C_9$. The first three lines correspond to the nine $4$-dimensional maximal cells of $\mathrm{E}(C_{2n+1})$ and the last line corresponds to the three $3$-dimensional maximal cells.}
\label{fig:f1}
\end{figure}
\end{Expl}

\pagebreak

\begin{Expl}\label{Expl:example2}
Let $(X,d)$ be an infinite connected bipartite $(4,4)$-graph endowed with the shortest-path metric. Let $S=\{A,B\}$ be an alternating split on $X$ (cf. \cite{BanC}). Assume by contradiction that $S$ has isolation index $\alpha_S^d \in \{0,1/2\}$. We show that $\alpha_S = 1$. Since all isolation indices of splits on $X$ are in $\frac{1}{2}\mathbb{Z}$, it follows that one can find four points $r,s,u,v \in X$ such that $r,s \in A$, $u,v \in B$ and $\alpha_{\{ \{r,s\},\{u,v\}\}}^d = \alpha_S^d$. We can now consider a finite subgraph $Y$ of $X$ such that
\begin{enumerate}[$(a)$]
\item $Y$ is a bipartite $(4,4)$-graph and
\item $I(r,s)\cup I(r,u) \cup I(r,v) \cup I(s,u) \cup I(s,v) \cup I(u,v) \subset Y$.
\end{enumerate}
In general, the restriction $d|_{Y \times Y}$ does not coincide with the shortest path metric $d_Y$ on $Y$. However, by $(b)$, it follows that for any $a,b \in \{r,s,u,v\}$, one has $d_Y(a,b) = d|_{Y \times Y}(a,b)$. Note that $S_Y := \{A \cap Y, B \cap Y\}$ is an alternating split of $Y$ (restrictions of alternating splits to $(4,4)$-subgraphs are easily seen to be alternating again) and 
\[
\alpha_{S_Y}^{d_Y} 
\le \alpha_{\{ \{r,s\},\{u,v\}\}}^{d_Y}
= \alpha_{\{ \{r,s\},\{u,v\}\}}^{d|_{Y \times Y}}
= \alpha_S^d 
\le 1/2.
\]
However, by $(a)$, $(Y,d_Y)$ is a finite bipartite $(4,4)$-graph and thus by \cite[Proposition~8.8]{BanC}, it follows that $\alpha_{S_Y}^{d_Y}=1$ which contradicts the above. If now $\mathcal{S}_A$ denotes the system of all alternating splits of $(X,d)$, note that for any $x,y \in X$ such that $d(x,y)=1$, there is a unique $S \in \mathcal{S}_A$ such that $S(x)\neq S(y)$. Note that $d_0 := d - \sum_{S \in \mathcal{S}_A} \delta_S$ is a pseudometric by Theorem~\ref{Thm:thm2}. Hence for any $a,b \in X$, consider a path $a=x_0,\ x_1, \ \dots,\ x_{m-1}, \ x_m=b$ in $(X,d)$, we have
\[
d_0(a,b) \le \sum_{i=0}^{m-1} d_0(x_i,x_{i+1})=0.
\]
It follows that $d_0$ is identically zero and thus $d$ is totally split-decomposable. Moreover, $(X,\bar{d})$ satisfies the (LRC). Indeed, the isometric cycles in $(X,\bar{d})$ are gated (cf. \cite[Theorem~8.7]{BanC}), it follows that $(X,\bar{d})$ has $1$-stable intervals and thus by the proof of \cite[Theorem~1.1]{Lan}, we obtain the desired result. Examples of such infinite bipartite $(4,4)$-graphs are given for $m \ge 4$, for $\sigma$ any element of the symmetric group $\mathfrak{S}_m$ and for $\{r_{\sigma(i)\sigma(i+1)}\}_{i \in \{1,\dots,m\}} \subset \mathbb{N} \cap [2,\infty) \cup \{\infty\}$ by the Cayley graph of Coxeter groups of the form
\begin{align*}
C = \left\langle s_1,\dots,s_m | \right.
&(s_{\sigma(1)} s_{\sigma(2)} )^{r_{\sigma(1)\sigma(2)}}=1,\dots, (s_{\sigma(m-1)} s_{\sigma(m)} )^{r_{\sigma(m-1)\sigma(m)}}=1,  \\
&(s_{\sigma(m)} s_{\sigma(1)} )^{r_{\sigma(m)\sigma(1)}}=1 \left.  \right\rangle\ .
\end{align*}
The restriction on the number of relations ensures that the Cayley graph is planar, $m \ge 4$ ensures that the degree is at least four and the condition $r_{\sigma(i)\sigma(i+1)}\ge 2$ for every $i$ ensures that each face contains at least four vertices. 
\end{Expl}
 

{\bf Acknowledgements.}
I am very grateful to my PhD advisor Prof. Dr. Urs~Lang for numerous stimulating conversations about the topic, and for reading this work as well as preliminary versions. This research was supported by the Swiss National Science Foundation.





\end{document}